\documentclass[11pt,letterpaper]{amsart}
\usepackage{amssymb,latexsym,amsmath, amsthm, tikz-cd, enumerate,setspace, mathtools, mathdots, bm}
\usepackage{mathrsfs}
\usepackage[hidelinks]{hyperref}

\input{xypic}
\xyoption{all}
\theoremstyle{plain}

\newcommand{\A}{{\mathbb A}}
\newcommand{\C}{{\mathbb C}}
\newcommand{\D}{{\mathbb D}}

\newcommand{\Q}{{\mathbb Q}}
\newcommand{\R}{{\mathbb R}}
\newcommand{\Z}{{\mathbb Z}}

\newcommand{\cM}{{\mathcal M}}
\newcommand{\cS}{{\mathcal S}}

\newcommand{\RR}{{\mathcal R}}
\newcommand{\cA}{{\mathcal A}}
\newcommand{\cV}{{\mathcal V}}
\newcommand{\cW}{{\mathcal W}}
\newcommand{\cN}{{\mathcal N}}
\newcommand{\calD}{{\mathcal D}}
\newcommand{\HH}{{\mathcal H}}

\newcommand{\bpi} {{ \boldsymbol \pi}}

\newcommand{\brho} {{ \boldsymbol \rho}}

\newcommand{\sI} {\mathscr{I}}

\newcommand{\sL} {\mathscr{L}}

\newcommand{\fm}{{\mathfrak m}}
\newcommand{\g}{{\mathfrak g}}
\newcommand{\fk}{{\mathfrak k}}
\newcommand{\fa}{{\mathfrak a}}
\newcommand{\fs}{{\mathfrak s}}

\newcommand{\q}{\mathfrak{q}}
\newcommand{\fu}{\mathfrak{u}}

\newcommand{\fI}{\mathfrak{I}}

\def\k{\mathfrak k}

\newcommand{\M}{{\mathcal M}}

\renewcommand{\O}{{\mathcal O}}
\newcommand\SGK{\mathcal{S}^G_{K_f}}
\newcommand\HGK{\mathcal{H}^G_{K_f}}
\newcommand\SG{\mathcal{S}^G}

\newcommand\tM{\widetilde{\mathcal{M}}}
\newcommand\BSC{ \bar{\mathcal{S}}^G_{K_f}}
\newcommand\PBSC{\partial\SGK}
\newcommand\uSMP{\ul{\mathcal{S}}^{M_P}}
\newcommand\uSM{\ul{\mathcal{S}}^{M}}

\newcommand \iso{ \buildrel \sim \over\longrightarrow}

\newcommand{\GL}{{\rm GL}}
\newcommand{\SO}{{\rm SO}}
\newcommand{\rO}{{\rm O}}

\newcommand{\SL}{{\rm SL}}

\newcommand{\st}{{\rm st}}
\newcommand{\norm}{{\rm norm}}
\newcommand{\Eis}{{\rm Eis}}
\newcommand{\loc}{{\rm loc}}
\newcommand{\arith}{{\rm arith}}
\def\Orth{{\rm O}}

\newcommand{\so}{\mathfrak{so}}

\def\tM{\widetilde{\mathcal{M}}}

\def\Ind{{\rm Ind}}
\def\Res{{\rm Res}}
\def\aInd{{}^{\rm a}{\rm Ind}}

\def\Hom{{\rm Hom}}
\def\Spec{{\rm Spec}}
\def\ul{\underline}

\def\ss{{\sf s}}
\def\sr{{\sf r}}
\def\sv{{\sf v}}

\newtheorem{theorem}[equation]{Theorem}
\newtheorem*{theorem*}{Theorem}
\newtheorem{corollary}[equation]{Corollary}
\newtheorem{lemma}[equation]{Lemma}
\newtheorem{proposition}[equation]{Proposition}

\newtheorem{definition}[equation]{Definition}
\newtheorem{remark}[equation]{Remark}

\title[Special values of $L$-functions for orthogonal groups]{Eisenstein cohomology for orthogonal groups and the 
special values of $L$-functions for $ \GL_1 \times \Orth(2n) $}
\author{\bf Chandrasheel Bhagwat \ \ \& \ \ A. Raghuram}
\address{Chandrasheel Bhagwat: Indian Institute of Science Education and Research, Dr.\,Homi Bhabha Road, Pashan, Pune 411008,  INDIA.} 
\email{cbhagwat@iiserpune.ac.in}

\address{A. Raghuram: Department of Mathematics, Fordham University at Lincoln Center, New York, NY 10023, USA.}
\email{araghuram@fordham.edu}

\date{\today}

\begin{document}
\begin{abstract} 
For an even positive integer $n$, we study  rank-one Eisenstein cohomology of the split orthogonal group ${\rm O}(2n+2)$ over a totally real number field $F.$   
This is used to prove a rationality result for the ratios of successive critical values of degree-$2n$ Langlands $L$-functions associated to the group ${\rm GL}_1 \times  {\rm O}(2n)$ over $F$. The case $n=2$ specializes to classical results of Shimura on the special values of Rankin--Selberg $L$-functions attached to a pair of Hilbert modular forms. 
\end{abstract}

\maketitle

Let $f = \sum a_nq^n$ 
and $g = \sum b_nq^n$ 
be primitive holomorphic modular cuspforms for $\Gamma_0(N),$ of weights 
$k$ and $l$, with nebentypus characters $\chi$ and $\psi,$ respectively. 
Let $\Q(f,g)$ be the number field obtained by adjoining the Fourier coefficients $\{a_n\}$ and $\{b_n\}$ to $\Q.$ 
Assume that $k > l.$ A well-known theorem of Shimura \cite{shimura-mathann} says that for 
$D_N(s, f,g),$ the degree-$4$ Rankin--Selberg $L$-function attached to the pair $(f,g)$, and for any integer $m$ with 
$l \leq m < k,$ we have: 
$D_N(m,f,g) \ \approx \ (2\pi i)^{l+1-2m} \, \g(\psi) \, u^+(f)u^-(f),$
where $\approx$ means equality up to an element of $\Q(f,g)$, $u^\pm(f)$ are the two periods attached to $f$ by Shimura, and 
$\g(\psi)$ is the Gauss sum of $\psi.$ The result may be refined and stated as a reciprocity law. 
The integers $l \leq m < k$ are all the critical points for $D_N(s, f,g).$ There are no critical points if $l = k.$ 
Suppose $k \geq l+2$, and we look at two successive critical values then the only change in the right hand side is 
$(2\pi i)^{-2}$ which may be seen to be exactly accounted for 
by the $\Gamma$-factors at infinity. Suppose $L(s, f \times g)$ denotes the completed degree-$4$ $L$-function attached to $(f,g)$, 
then we deduce: 
$$
L(l, f \times g) \ \approx \ L(l+1, f \times g) \ \approx \ \cdots \ \approx \ L(k-1, f \times g). 
$$
The above result is a statement for $L$-functions for $\GL_2 \times \GL_2$ over $\Q.$ 
Shimura also generalized this to Hilbert modular forms \cite{shimura-duke}, i.e., for $\GL_2 \times \GL_2$ over a totally real field $F.$ Since $(\GL_2 \times \GL_2)/\Delta\GL_1 \simeq {\rm GSO}(4),$ one may construe Shimura's result as a theorem for the degree-$4$ $L$-functions for certain orthogonal groups in four variables.  The main theorem of this paper (see Thm.\,\ref{thm:main:introduction} below) generalises Shimura's result as above 
to $L$-functions for 
$\GL_1 \times \rO(2n)$ over a totally real number field $F$ for an even positive integer $n$.  
The principal innovation of this article is that it offers new results on the arithmetic properties of $L$-functions for classical groups, outside the framework of general linear groups, via the Langlands--Shahidi theory of automorphic $L$-functions. We generalise the work of Harder and the second author \cite{harder-raghuram} to 
study Eisenstein cohomology for $\rO(2n+2)$, while using Arthur's classification \cite{arthur} as refined by Atobe--Gan \cite{atobe-gan} for even orthogonal 
groups, to give a cohomological interpretation to certain aspects of the Langlands--Shahidi machinery  \cite{shahidi-book}.

\medskip

We consider the split orthogonal group $\rO(2n)$ defined by the matrix $J_{2n}$ (see \eqref{defn:J}) which has $1$'s along the anti-diagonal and $0$'s elsewhere. For now we let $G = {\rm Res_{F/\Q}}(\rO(2n)/F)$ be the group over $\Q$ given by Weil restriction of scalars from $F/\Q$. (Later, $G$ will be a bigger orthogonal group and ${\rm Res_{F/\Q}}(\rO(2n)/F)$ will be an almost simple factor of the Levi quotient of a particular parabolic subgroup of $G$.) 
Let $\A$ denote the ad\`ele ring of $\Q$, 
and $\A_f$ be the finite ad\`eles. For an open-compact subgroup $K_f$ of $G(\A_f)$
we denote the ad\`elic locally symmetric space with level structure $K_f$ by $\cS^{G}_{K_f}.$ These notions are defined in 
Sect.\,\ref{section-preliminaries-notations}. Let $T$ be the restriction of scalars of the diagonal torus in $\rO(2n)$. We take a large enough finite Galois extension $E$ of $\Q$ containing a copy of $F$. Let $\lambda \in X^*(T \times E)$ be a dominant integral weight, and $\cM_{\lambda,E}$ be the algebraic finite-dimensional absolutely irreducible representation of $G \times E$ with highest weight $\lambda$. Let $\widetilde{\cM}_{\lambda,E}$ be the associated sheaf of $E$-vector spaces on $\cS^{G}_{K_f}.$
A fundamental object of study is the cohomology $H^{\bullet} (\cS^{G}_{K_f}, \widetilde{\cM}_{\lambda,E})$ of the space $\cS^{G}_{K_f}$ with coefficients in 
the sheaf $\tM_{\lambda,E}.$  
A basic tool to study these cohomology groups is the long exact sequence attached to the Borel--Serre compactification $\overline \cS^{G}_{K_f} = \cS^{G}_{K_f} \cup \partial \cS^{G}_{K_f}:$ 
\begin{eqnarray*}
\cdots \xrightarrow{}
{H^{i}_{c}(\cS^{G}_{K_f},  \widetilde{\cM}_{\lambda,E} )}
\xrightarrow{\mathfrak{i}^{\bullet}} 
H^{i} (\cS^{G}_{K_f},  \widetilde{\cM}_{\lambda,E} ) 
\xrightarrow{\mathfrak{r}^{\bullet}} 
H^{i} (\partial \cS^{G}_{K_f},  \widetilde{\cM}_{\lambda,E} ) \xrightarrow{} \\
\xrightarrow{\partial^{\bullet}} 
H^{i+1}_c(\cS^{G}_{K_f},  \widetilde{\cM}_{\lambda,E} ) 
\xrightarrow{}
\cdots 
\end{eqnarray*}
This is a sequence of $\mathcal H^{G}_{K_{f}}$-modules, where $\mathcal H^{G}_{K_{f}}$ is a suitable Hecke-algebra over $\Q$ defined as the restricted 
tensor-product of local Hecke algebras. For a finite set $S$ 
of places including the archimedean and ramified places, we let $\mathcal H^{G,S}$ denote the central subalgebra of $\mathcal H^{G}_{K_{f}}$ by taking
the restricted tensor product over all places not in S. We often consider the above cohomology groups as modules over $\mathcal H^{G,S}$.

\medskip

Inner (or interior) cohomology $H^{\bullet}_{!}$ is defined as the image of $\mathfrak{i}^{\bullet}$, i.e., the image of cohomology with compact supports inside total cohomology, and complementary to it is Eisenstein cohomology defined as the image 
of $\mathfrak{r}^{\bullet}$, i.e., image of the total cohomology in the cohomology of the boundary. 
In Sect.\,\ref{strongly-inner-cohomology} we define {\it strongly-inner cohomology} $H^{\bullet}_{!!}$ which is a subspace of $H^{\bullet}_{!}.$ The definition is motivated by \cite{harder-raghuram}, where strongly-inner cohomology is defined for $\GL_N/F$ having the virtue that under any 
embedding $\iota : E \to \C$, rendering the 
context transcendental, strongly-inner cohomology base-changes to cuspidal cohomology. However, the general linear group dealt with in \cite{harder-raghuram} is 
misleadingly simple. In the context of orthogonal groups, we appeal to Arthur's classification of the discrete spectrum as expounded by 
Atobe--Gan \cite{atobe-gan}, and offer a definition of strongly-inner cohomology defined over $E$ which captures an essential part of the cuspidal cohomology of $G$ (see Def.\,\ref{def:strongly-inner-spectrum}).   
If $\sigma_f$ is a $\mathcal H^{G}_{K_{f}}$-module appearing in  $H^{\bullet}_{!!}$, then for any 
embedding $\iota : E \to \C$, there is a cuspidal automorphic representation 
${}^\iota\sigma$ of $\rO(2n)$ over $F$ such that (i) ${}^\iota\sigma$ is globally generic for a pre-specified Whittaker datum $\psi$; 
(ii) the Arthur parameter $\Psi({}^{\iota}\sigma)$ of ${}^\iota\sigma$ is a cuspidal representation of $\GL_{2n}/F$; 
(iii) for each archimedean place $v$ of $F$, ${}^{\iota}\sigma_v$ is the $\psi_{v}$-generic discrete series representation 
of $\rO(n,n)(\R)$ (recall that $n$ is even) 
in the local $L$-packet defined by  the weight ${}^{\iota}\mu^{v}$; (iv) $({}^\iota\sigma)_f^{{}^\circ C_f} = \sigma_f \otimes_{E, \iota} \C$ 
as $\mathcal H^{G}_{K_{f}}$-modules over $\C.$ 
It is part of the defining property of strongly-inner cohomology that it provides a rational structure to an essential part of cuspidal cohomology of $G$.
Strongly-inner cohomology is an ad hoc replacement of Clozel's results \cite{clozel} on rational structure on cuspidal cohomology for 
$\GL_n.$  

\medskip

Let $\chi^{\circ}$ be a finite order Hecke character of $F$ taking values in $E,$ and $d$ be an integer.  
Let $\chi$ stand for the algebraic Hecke character of $F$ with values in $E$ which for an embedding $\iota: E \to \C$ 
gives a continuous homomorphism ${}^\iota\chi : F^\times\backslash \A_F^\times \to \C^\times$ 
of the form ${}^\iota\chi = {}^\iota\chi^\circ \otimes |\ |^{-d}.$   
For a strongly-inner Hecke summand $\sigma_f$, 
given the cuspidal automorphic representation ${}^\iota\sigma$ as above, we consider the degree-$2n$ completed 
$L$-function $L(s, {}^\iota\chi \times {}^\iota\sigma).$ This $L$-function should be construed as the Langlands--Shahidi 
$L$-function (see case $(D_{n,i})$ of Appendix A in Shahidi's book \cite{shahidi-book}) attached to any summand in the 
restriction of ${}^\iota\sigma$ to $\SO(2n)/F$ and the character ${}^\iota\chi.$ 
An integer $m$ is said to be {\it critical} for $L(s, {}^\iota\chi \times {}^\iota\sigma)$ if 
the $L$-factors at infinity on either side of the functional equation are regular at $s = m.$ Our main result on rationality properties 
of critical values of $L(s, {}^\iota\chi \times {}^\iota\sigma)$ is the following: 

\medskip

\begin{theorem}
\label{thm:main:introduction}
 
Let $\mu \in X^*(T \times E)$ be a dominant integral weight; suppose $\mu = (\mu^\tau)_{\tau : F \to E}$, with 
$\mu^\tau = (\mu^\tau_1,\dots, \mu^\tau_n)$, $\mu^\tau_j \in \Z,$ and $\mu^\tau_1 \geq \cdots \geq \mu^\tau_{n-1} \geq |\mu^\tau_n|.$   
Define 
$$
\mu_{\rm min} \ := \ 
\min \{\ |\mu^\tau_n| \ \}_{\tau : F \to E}. 
$$ 
Assume that $\mu_{\rm min} \geq 1.$ 
Let $\sigma_f$ be a Hecke summand appearing in the strongly-inner cohomology of $\rO(2n)/F$ with coefficients in $\widetilde{\cM}_{\mu,E}.$
For an embedding $\iota : E \to \C$, let ${}~^{\iota}\sigma$ be the cuspidal automorphic representation of $\rO(2n)/F.$ 
Let $\chi{}^\circ$ be a finite-order Hecke character of $F$ that takes values in $E$ and $d \in \Z,$ giving a character 
${}^\iota\chi = {}^\iota\chi^\circ \otimes |\ |^{-d}$ as above.  We have: 

\medskip
\begin{enumerate}
\item[(i)]
The critical set of $L(s, {}^\iota\chi \times {}^\iota\sigma)$ is given by the $2 \mu_{\rm min}$ consecutive integers
$$
\{d+1-\mu_{\rm min}, \ d+2-\mu_{\rm min}, \ \ldots \ , \ d+\mu_{\rm min}-1, \ d+\mu_{\rm min} \}. 
$$
The critical set is independent of $\iota.$ 
\end{enumerate}

\medskip
Since $L(s, {}^\iota\chi \times {}^\iota\sigma) = L(s - d, {}^\iota\chi^\circ \times {}^\iota\sigma)$, for rationality of critical values we consider 
the critical values of $L(s, {}^\iota\chi^\circ \times {}^\iota\sigma)$ without any loss of generality.

\begin{enumerate}
\medskip
\item[(ii)] The assumption $\mu_{\rm min} \geq 1$ guarantees the existence of at least two successive critical points. Suppose $m$ and $m+1$ are critical, 
and $m \neq 0$, then we have: 
$$
\frac{L(m, {}^\iota\chi^\circ \times {}^\iota \sigma) }{L(m+1, {}^\iota\chi^\circ \times {}^\iota \sigma)} \ \in \ \iota(E). 
$$
The above conclusion is also valid for $m=0$ provided we assume furthermore that ${}^\iota\sigma$ is tempered at all finite places. 

\medskip
\item[(iii)]
For every $\eta \in {\rm Gal}(\bar{\Q}/ \Q)$, we have the reciprocity law:
$$
\eta 
\left(
\frac{L(m, {}^\iota\chi^\circ \times {}^\iota \sigma)}{L(m+1, {}^\iota\chi^\circ \times {}^\iota \sigma)}\right) 
\ = \ 
\frac{L(m, {}^{\eta\circ\iota}\chi^\circ \times {}^{\eta\circ\iota} \sigma)}{L(m+1, {}^{\eta\circ\iota}\chi^\circ \times {}^{\eta\circ\iota}\sigma)}. 
$$
\end{enumerate}
\end{theorem}

\medskip

The calculation of the critical set in $(i)$ follows from knowing the Langlands parameter of discrete series representations, which, 
together with the archimedean case of the local Langlands correspondence, gives us the local $L$-factor at infinity. 
See Sect.\,\ref{sec:critical-set}.

\medskip
For the rest of the theorem, we follow the general framework of 
\cite{harder-raghuram}, however, we encounter serious difficulties in 
that we are working with an orthogonal group and not a general linear group. For $(ii)$ and $(iii)$, as in \cite{harder-raghuram}, there is an underlying {\it combinatorial lemma} (Sect.\,\ref{sec:kostant-comb-lem}) which says that we need to prove a rationality result for one specific ratio of critical values for a general family, and then the result automatically follows for all ratios of successive critical values. The critical set calculation gives that the combinatorial condition 
\begin{equation}
\label{eqn:comb-cond}
1 - \mu_{\rm min} \ \leq \ -(d+n) \ \leq \ \mu_{\rm min} - 1
\end{equation}
is equivalent to the fact that $s = -n$ and $s = 1-n$ are critical for $L(s, {}^\iota\chi \times {}^\iota\sigma).$ Under this condition, we prove  
\begin{equation}\label{eqn:s=-n}
\frac{L(-n, {}^\iota\chi \times {}^\iota \sigma) }{L(1-n, {}^\iota\chi \times {}^\iota \sigma)} \ \in \ \iota(E), 
\end{equation}
and also prove a reciprocity law as in $(iii)$ for this particular ratio; then both $(ii)$ and $(iii)$ of the theorem follow by letting $d$ vary as long as the combinatorial condition \eqref{eqn:comb-cond} is satisfied. Turns out that the constraint imposed on $d$ by the inequalities in \eqref{eqn:comb-cond} 
exactly captures every successive ratio of critical values for $L(s, {}^\iota\chi^\circ \times {}^\iota \sigma),$ {\it no more and no less!}

\medskip
The technical heart of the proof of \eqref{eqn:s=-n} concerns rank-one Eisenstein cohomology for $\rO(2n+2)/F$ and 
Langlands's constant term theorem viewed in terms of maps in cohomology. Changing notations, 
henceforth we let $G = \Res_{F/\Q}(\rO(2n+2)/F),$ and $P = \Res_{F/\Q}(P_0),$ where $P_0$ is the maximal parabolic subgroup of $\rO(2n+2)/F$ whose 
Levi quotient is $M_0 = \GL_1/F \times \rO(2n)/F.$ The combinatorial lemma also gives that the algebraically and parabolically induced representation 
$\aInd_{P(\A_f)}^{G(\A_f)}(\chi_f \times \sigma_f)$ appears in the cohomology of the boundary in a very special degree; this involves combinatorial subtleties on Kostant representatives in Weyl groups. We prove a `Manin--Drinfeld principle' (Thm.\,\ref{thm:Manin-Drinfeld}) which 
says that this induced representation together with its partner across an intertwining operator splits off as an isotypic component; this theorem uses results of Arthur and Atobe--Gan on the classification of the discrete spectrum for orthogonal groups. We then prove our main theorem on rank-one Eisenstein cohomology (Thm.\,\ref{thm:rank-one-eis-coh}) 
which says that the image of Eisenstein cohomology in this isotypic component is analogous to a line in a two-dimensional plane. If one passes to a transcendental situation using an embedding $\iota : E \to \C$, then via Langlands's constant term theorem, the slope of this line is the required ratio of $L$-values.  

\medskip

There are certain archimedean problems that need to be solved. First, Harish-Chandra's classification of discrete series representations works best for a connected semisimple Lie group. For us, the (almost-)simple factor of the Levi group $M_0(\R)$ is $\rO(n,n)(\R)$, the real split orthogonal group, whose group of connected components has order $4$. The action of this group of connected components on relative Lie algebra cohomology of the discrete 
series representation of $\rO(n,n)(\R)^\circ = \SO(n,n)(\R)^\circ$ necessitates some arduous book-keeping;  
this is reviewed in Sect.\,\ref{section-discrete-series-representations-and-cohomology}, and applied later in \ref{sec:coh-induced-rep}. 
Next, we need a delicate analysis of the rationality properties of archimedean representations and intertwining operators; see Sect.\,\ref{sec:archimedean}. 
As in other contexts involving Eisenstein cohomology (\cite{harder-raghuram}, \cite{raghuram}), the point of evaluation $s = -n$
being a critical point, and if we are {\it on the right of the unitary axis} (see Sect.\,\ref{sec:holomorphy}), 
then from general results of Casselman--Shahidi \cite{casselman-shahidi}, it follows that the induced representations at hand are irreducible and 
the standard intertwining operator is well-defined and finite at this point. We show in Prop.\,\ref{prop:intertwining-operator-in-cohomology} that the map induced in relative Lie algebra cohomology by the archimedean standard intertwining operator contributes, up to nonzero rational numbers, the ratio of local archimedean $L$-values in \eqref{eqn:s=-n}.

\medskip

As mentioned earlier, what's new in this article is that it represents the first serious study of the arithmetic properties of $L$-functions for classical groups, outside the framework of general linear groups, via a cohomological interpretation of certain parts of the Langlands--Shahidi theory of $L$-functions \cite{shahidi-book}. We expect our methods to pave the way to study the arithmetic of other families of automorphic $L$-functions. 
For $n=2,$ our results specialize to results of Shimura, however, for all even $n \geq 4$ the results are new. A technically subtle point, 
the full import of which we hadn't appreciated in our announcement \cite{bhagwat-raghuram-cras}, 
is our use of $\rO(2n)$ and not $\SO(2n)$; 
Arthur's classification \cite{arthur} of the discrete spectrum for $\SO(2n)$ is only up to conjugation by the ambient $\rO(2n)$; this was finessed by Atobe and Gan \cite{atobe-gan} who gave a satisfactory classification for the discrete spectrum for $\rO(2n).$ In Sect.\,\ref{sec:final-comment}, amplifying on a comment made 
in  \cite{bhagwat-raghuram-cras}, we explain the need to work intrinsically with $\rO(2n)$, and not transfer to $\GL(2n)$, to prove our main theorem on the special values of $L$-functions. 
Finally, our theorem on the rationality of ratios of successive critical values is compatible with Deligne's conjecture \cite{deligne} on the critical values of motivic $L$-functions. The reader is referred to the forthcoming article of Deligne and the second author \cite{deligne-raghuram} 
which gives a motivic explanation for a growing body of results on ratios of successive critical values including the main result of this article.

\medskip

{\small
{\it Acknowledgements:} 
It is a pleasure to thank G\"unter Harder, who pioneered the use of Eisenstein cohomology to study the special values of $L$-functions, for his constant enthusiasm and encouragement during this project which would not have been possible without the insights obtained in \cite{harder-raghuram}. 
We thank Michael Harris and Freydoon Shahidi for their interest and encouragement. We are grateful to Wee Teck Gan for helpful correspondence on Arthur parameters for orthogonal groups, Galois actions, and for drawing our attention to his work with Atobe that turned out to be crucial. 
C.B. acknowledges support from his research grants INSPIRE faculty award IFA-11MA-5 of Department of Science and Technology, Government of India, and Early Career Research Award ECR/2016/001642. Both C.B. and A.R. acknowledge support from MATRICS research grants MTR/2018/000102 and MTR/2018/000918, respectively, from the Science and Engineering Research Board, Department of Science and Technology, Government of India.}

\medskip 
\tableofcontents

\medskip
\section{Preliminaries on the cohomology of orthogonal groups} 
\label{section-preliminaries-notations}

\medskip
\subsection{The base field and algebraic groups} 
\label{subsection-base-field}
Let $F$ be a totally real number field of degree ${\sf r}_F = [F: \Q].$ Let $\Sigma_F := {\rm Hom}(F, \R) = {\rm Hom}(F,\C)$ be the set of all embeddings of $F$ into $\R$ or $\C,$ and denote by $S_{\infty}$ the set of archimedean places of $F$. Of course, $S_\infty$ is in bijection with $\Sigma_F.$ 
The completion of $F$ at any place $v$ is denoted $F_v.$ 
For $v \in S_{\infty}$, we identify $F_{v}$ with $\R$.  Let $n = 2r \geq 2$ be an even integer. 
Let $G_{0} =   \rO(2n+2) / F$ be the split orthogonal group defined by the equation ${}^\intercal g~ J_{2n+2}~ g = J_{2n+2} $, 
where ${}^\intercal g$ is the transpose of $g$, and for any positive integer $m$, the matrix $J_{m}$ is given by:
\begin{equation}\label{defn:J}
J_{m} =  \left( \begin{array}{ccc} 
& & 1 \\  
& \iddots \\ 
1 & &  \end{array}
\right)_{m \times m} . 
\end{equation}
When the size $m$ is clear from context we will denote $J_m$ simply by $J.$ 
The subgroup $B_{0}$ consisting of all upper triangular matrices in $G_{0}$ defines a  Borel subgroup of $G_0$.  Let $T_{0} \subset B_{0}$ be a split torus of $G_0$ consisting of the diagonal matrices. Let $N_0$ be the unipotent radical of $T_0$ in $B_0$. So we have $B_0 = T_0 N_0$.
Let $G = R_{F / \Q} (G_{0} / F)$ be the group obtained by the Weil restriction of scalars to $\Q$.  Similarly, for $B_0$, $T_0,$ and $N_0$, we denote the $\Q$-groups obtained from by restriction of scalars by $B$, $T,$ and $N$, respectively.

\medskip
\subsection{The groups at infinity} 
\label{sec:groups-at-infty}

The group of real points of $G$ is given by 
$$
G(\R) = \prod \limits_{v \in S_{\infty}} G(F_v) = 
\prod \limits_{v \in S_{\infty}} \rO(n+1,n+1)(\R).
$$ 
The map  $g \mapsto {}^\intercal g^{-1}$ defines a Cartan involution on $G(\R)$ and its fixed points define a maximal compact subgroup $K_{\infty}$ of $G(\R)$: 
$$
K_{\infty}  \ \simeq \ \prod \limits_{v \in S_{\infty}} {\rm O}(n+1)(\R) \times {\rm O}(n+1)(\R). 
$$ 
To describe this isomorphism explicitly, consider the map 
$$
h : {\rm O}(n+1)(\R) \times {\rm O}(n+1)(\R) \ \longrightarrow \ \rO(n+1,n+1)(\R), 
$$ 
which is defined by: 
\begin{equation}
\label{eqn:h-map}
(k, l) \mapsto h(k, l) =  \dfrac{1}{2} 
\begin{pmatrix} k + J l J & k J - J l \\ 
J k - l J & J k J + l \end{pmatrix}, 
\end{equation}
where $k,l \in \rO(n+1)(\R),$ i.e., they satisfy ${}^\intercal k \cdot k = I_{n+1} = {}^\intercal l \cdot l,$ and $J$ in the formula above is $J_{n+1}.$ We leave it 
to the reader to check that $h$ is a homomorphism and that the image of $h$ lies in $\rO(n+1,n+1)(\R)$; i.e., the map $h$ gives an embedding of 
 ${\rm O}(n+1)(\R) \times {\rm O}(n+1)(\R) \hookrightarrow \rO(n+1,n+1)(\R)$. For $v \in S_{\infty}$, let $K_{v}$ be the image of  
 ${\rm O}(n+1)(\R) \times {\rm O}(n+1)(\R)$  inside $G_{v}$, and put $K_{\infty} = \prod \limits_{v \in S_{\infty}}K_{v}$; its  
connected component of the identity is $K^{\circ}_{\infty}$ which is isomorphic to $\prod \limits_{v \in S_{\infty}} \SO(n+1)(\R) \times \SO(n+1)(\R)$ and the component groups $ \pi_{0}(G(\R)) = \pi_{0}(K_{\infty}) \ \simeq \ (\Z/ 2\Z \times \Z/ 2\Z)^{\sf r_F}.$ For any $m \geq 1$, let's write 
$\delta_m = {\rm diag}(-1,1,\dots,1),$ and put $\SO(m)(\R)^\dagger := \SO(m)(\R)\delta_m,$ then $\SO(m)(\R)$ and $\SO(m)(\R)^\dagger$ are the connected components of $\rO(m)(\R).$ For $v \in S_\infty$, 
 define a subgroup $\overline {K}_{v}$ of $K_v$ by 
\begin{equation}
\label{eqn:K-infty-to-divide-by-for-G}
\overline {K}_{v} \ := \ h(\SO(n+1)(\R) \times \SO(n+1)(\R)) \cup h(\SO(n+1)(\R)^\dagger \times \SO(n+1)(\R)^\dagger)
\end{equation}
and let  $\overline {K}_{\infty} : =  \prod \limits_{v  \in S_{\infty}} \overline {K}_{v} $. The component group  of $\overline {K}_{\infty}$ will be described later, and will be used in the context of boundary cohomology.

\medskip
\subsection{Ad\`elic locally symmetric spaces} 
Let $\A$ be the ring of ad\`eles of $\Q$, and $\A_f$ denote the finite ad\`eles.  Let $K_{f} = \prod \limits_{v \notin S_{\infty}} K_v \subset G(\A_{f})$ be an open compact subgroup. We will work with the ad\`elic symmetric space for $G$  defined by
$$ \cS^{G}: = G(\A) / \overline{K}_{\infty} K_{f} = G(\R)/\overline{K}_{\infty} \times G(\A_f)/K_{f}.$$
The discrete subgroup $G(\Q)$ of $G(\A)$ acts properly discontinuously via left-translations on $G(\A) / \overline{K}_{\infty} K_{f} $. The ad\`elic locally symmetric space
$\cS^{G}_{K_{f}}$, with level structure $K_{f} $ is the quotient space given by 
$$ 
G(\A) / \overline{K}_{\infty} K_{f}  \twoheadrightarrow G(\Q) \backslash G(\A) /\overline{K}_{\infty} K_{f}. 
$$ 
Note the technical artifice of dividing by $\overline{K}_{\infty}$; this will be useful later (see Sect.\,\ref{sec:coh-induced-rep}). 
In principle, we can divide by any subgroup in between 
$K_\infty^\circ$ and $K_\infty.$ When we associate similar locally symmetric spaces for Levi subgroups of $G$, and compute the action of 
the group of components on the cohomology of a discrete series representation, in preparation to describe parts of the cohomology of the Borel--Serre 
boundary, then the virtue of dividing by $\overline{K}_\infty$ will become clear. 
The level structure $K_f$ will be assumed to be {\it neat} which 
ensures that for any $\ul g_f \in G(\A_f)$, the discrete subgroup $G(\Q) \cap \ul g_f K_f \ul g_f^{-1}$ of $G(\R)$ is torsion free. Given any open compact subgroup $K_f$, we may pass to a subgroup of finite index which is neat. Later we will also assume that $K_f$ is a principal congruence subgroup 
(see \ref{sec:para-manin-drinfeld}).

\medskip
\subsection{The character module and dominant integral weights} 
\label{subsection-character-module} 
Let $E$ be a `large enough' finite Galois extension of $\Q$ such that  ${\rm Hom}(F, \R) = {\rm Hom}(F, E)$. The meaning of `large enough' is usually clear from the context; we will want some Hecke summands in inner cohomology to split over $E$. Let
$X^{\ast}(T \times E) = {\rm Hom}_{{\rm alg}} (T \times E, \mathbb G_m)$ be the group of characters of $T \times E.$ We have: 
$$
X^*(T \times E) = \bigoplus_{\tau : F \to E} X^{\ast}(T_{0} \times_{F, \tau} E).  
$$
A character $\lambda \in X^*(T \times E),$ also called an integral weight for $T$, 
will be written in the above decomposition 
as $\lambda = (\lambda^\tau)_{\tau : F \to E},$ with 
$\lambda^\tau \in X^*(T_0 \times_{F,\tau} E) = {\rm Hom}_{{\rm alg}} (T_0 \times_{F,\tau} E, \mathbb G_m);$ 
furthermore, we write 
$\lambda^{\tau} = (\lambda_0^{\tau}, \lambda_1^{\tau}, \ldots, \lambda_n^{\tau}),$ with integrality of $\lambda$ meaning that 
$\lambda^\tau_j \in \Z$, and describe its values on elements of the diagonal torus by: 
$$
\lambda^{\tau}(\text{diag} [ t_0, \ldots, t_n, t_n^{-1}, \ldots, t_{0}^{-1} ])  \ = \ 
\prod \limits_{i=0}^{n} t_i^{\lambda_i^{\tau}}.
$$
The Galois group ${\rm Gal} (E / \Q)$ acts on $X^{\ast}(T \times E)$ by: 
$
({}^{\eta} \lambda)^{\tau} : = \lambda^{\eta^{-1} \circ \tau },$
for all $\lambda \in X^{\ast}(T \times E),$ $\tau : F \to E$, and $\eta \in {\rm Gal} (E / \Q).$
An integral weight $\lambda = (\lambda^{\tau})_{\tau: F \rightarrow E} \in X^{\ast}(T \times E)$ is said to be a dominant with respect to Borel subgroup $B$ 
if for all $\tau : F \to E$ we have: 
$$
\lambda_0^{\tau} \geq \lambda_1^{\tau} \geq  \ldots \geq \lambda_{n-1}^{\tau} \geq |\lambda_{n}^{\tau}|.
$$
We denote the set of dominant integral weights for $T$ by $X^{\ast}_{+}(T \times E)$.

\medskip  
 \subsection{The sheaf $\widetilde{\mathcal M}_{\lambda, E}$ and its cohomology} 
 \subsubsection{}
Let $G^{\circ} = \Res_{F/\Q}(\SO(2n+2)/F)$ denote the connected component of the identity of $G.$ Given  
 $\lambda = (\lambda^{\tau})_{\tau: F \rightarrow E} \in X^{\ast}_+(T \times E)$ as above, let $(\rho_{\lambda}, \mathcal N_{\lambda, E})$ be the absolutely irreducible finite-dimensional representation of the group $G^{\circ} \times E$ with highest weight $\lambda$. Consider the induced representation 
  $\Ind^{\, G \times E }_{G^{\circ} \times E}   (\rho_{\lambda});$ we denote its representation space as $\mathcal M_{\lambda, E}$. We may parse 
  $\M_{\lambda, E}$ 
over the embeddings $\tau : F \to E$ as follows: 
let $\mathcal N_{\lambda^{\tau}, E}$ be the absolutely irreducible finite-dimensional representation of $\SO(2n+2) \!\times_{F, \tau}\!E$ with 
highest weight $\lambda^{\tau}$; induce $\mathcal N_{\lambda^{\tau}, E}$ to a representation $\M_{\lambda^{\tau}, E}$ of $\rO(2n+2) \!\times_{F, \tau}\!E$, then  
$$ 
\mathcal \M_{\lambda, E}  = 
\bigotimes_{\tau: F \rightarrow E} \mathcal M_{\lambda^{\tau}, E}.
$$
The tensor product is of $E$-vector spaces. 

\medskip 
  
 \subsubsection{} 
 Consider the quotient map:  
 $ \pi: G(\mathbb A) / \overline{K}_{\infty} K_{f}  \rightarrow G(\Q) \backslash G(\mathbb A) / \overline{K}_{\infty} K_{f} $.  The representation 
 $\M_{\lambda,E}$ gives a sheaf $\widetilde{\mathcal M}_{\lambda,E}$ 
 of $E$-vector spaces on the space ${\mathcal {S}}^G_{K_{f}}$ defined as: the sections over any open subset $V$ of ${\mathcal {S}}^G_{K_{f}}$ are given by 
\begin{multline*} 
  \widetilde{\mathcal M}_{\lambda,E}(V) \ := \ 
   \{s: \pi^{-1}(V) \rightarrow   \mathcal M_{\lambda,E} \ | \  \mbox{$s$ is smooth, and} \\
s(\gamma \ul x) = \rho_{\lambda}(\gamma) s(\ul x),   \ 
 \forall~\gamma \in G(\Q),  \ \ul x \in \pi^{-1}(V) \}, 
\end{multline*}
where, {\it smooth} means locally constant at non-archimedean places, and infinitely differentiable at the archimedean places.

\subsubsection{}
Analogous to \cite[Lem.\,2.3]{harder-raghuram}, we need to ask the question: under what conditions is the sheaf $\tM_{\lambda, E}$ nonzero? 
In our context this 
has an easy answer: suppose $Z_0$ is the center of $G_0$ and $Z = \Res_{F/\Q}(Z_0)$, then $Z(\Q) = Z_0(F) = \{\pm I\}$. It's easy to see that 
$Z(\Q) \cap \overline{K}_\infty$ is trivial. Hence, we deduce that the sheaf $\tM_{\lambda, E}$ is nonzero. Of course, since the center 
is finite, and since we may replace $K_f$ by a finite-index subgroup, we are guaranteed that $Z(\Q) \cap \overline{K}_\infty K_f$ is trivial.

 \subsection{The cohomology of $\widetilde{\mathcal M}_{\lambda, E}$}

\smallskip
\subsubsection{}
\label{sec:def-iota-lambda}
A fundamental object of interest is the cohomology $H^\bullet(\SGK,  \tM_{\lambda, E})$ of the space 
$\SGK$ with coefficients in the sheaf $\tM_{\lambda, E}.$ At this point, the reader is referred to \cite[Sect.\,2.3]{harder-raghuram} for generalities on 
these cohomology groups. We will be brief here in an effort to make this article reasonably self-contained. Let 
$\HGK = C^\infty_c(G(\A_f)/ \! \!/K_f)$ 
be the set of all locally constant and compactly supported bi-$K_f$-invariant $\Q$-valued functions on $G(\A_f).$  
Take the Haar measure on $G(\A_f)$ to be the product of local Haar measures, and for every prime $p$, the local measure is normalized so that 
$\mathrm{vol}(G(\Z_p)) = 1.$ Then $\HGK$ is a $\Q$-algebra under convolution. We have an action of $\pi_0(G(\R)) \times \HGK$ on $H^\bullet(\SGK, \tM_{\lambda, E}).$ Functorial properties of the sheaves $\tM_{\lambda, E}$ upon changing the field $E$ are relevant for Galois actions on the cohomology groups $H^\bullet(\SGK,  \tM_{\lambda, E})$, and also for the relation of these cohomology groups with 
automorphic forms using an embedding $\iota : E \to \C$ that gives a map 
$\iota^\bullet : H^\bullet(\SGK,  \tM_{\lambda, E}) \to H^\bullet(\SGK,  \tM_{{}^\iota\!\lambda, \C}).$ 
The embedding $\iota$ gives a bijection $\iota_* : \Hom(F,E) \to \Hom(F,\C)$ via composition.  
The weight ${}^\iota\!\lambda$ may be described by: 
${}^\iota\!\lambda = (({}^\iota\!\lambda)^\tau)_{\tau : F \to \C}$ where $({}^\iota\!\lambda)^\tau = \lambda^{\iota_*^{-1}\tau}.$

\medskip
\subsubsection{\bf A long exact sequence}
\label{sec:long-exact-seq}
 
 Let $\BSC$ be the Borel--Serre compactification \label{index:borel-serre-2} of $\SGK$; i.e., 
$\BSC = \SGK \cup \partial\SGK$, where the boundary is stratified as 
$\partial \SGK  = \cup_P \partial_P\SGK,$ with $P$ running through the $G(\Q)$-conjugacy classes of proper parabolic subgroups defined over $\Q$. (See Borel--Serre \cite{borel-serre}.) 
The sheaf $\tM_{\lambda, E}$ on $\SGK$ naturally extends to a sheaf on 
$\BSC$ that we also denote by $\tM_{\lambda, E}$. The inclusion $\SGK \hookrightarrow \BSC$ is a homotopy equivalence, and hence the 
restriction from $\BSC$ to $\SGK$ induces an isomorphism in cohomology 
$H^\bullet(\BSC, \tM_{\lambda, E}) \ \iso \  H^\bullet(\SGK, \tM_{\lambda, E}).$
The cohomology of the boundary $H^\bullet(\partial \SGK, \tM_{\lambda, E})$ and the cohomology 
with compact supports $H^\bullet_c(\SGK, \tM_{\lambda, E})$ are naturally 
modules for $ \pi_0(G(\R)) \times \HGK.$ We have the following long exact sequence of $\pi_0(G(\R)) \times \HGK$-modules: 
\begin{multline*}
\cdots  \longrightarrow H^i_c(\SGK, \tM_{\lambda, E}) 
\stackrel{\mathfrak{i}^\bullet}{\longrightarrow}   H^i(\BSC, \tM_{\lambda,E}) 
\stackrel{\mathfrak{r}^\bullet}{\longrightarrow } H^i(\partial \SGK, \tM_{\lambda,E}) \\
\xrightarrow{\partial^{\bullet}} 
H^{i+1}_c(\SGK, \tM_{\lambda,E}) \longrightarrow \cdots .
\end{multline*}

\medskip
\subsubsection{\bf Inner and Eisenstein cohomology}
The image of cohomology with compact supports inside the full cohomology is called \textit{inner} or \textit{interior} cohomology and is denoted 
$H^{\bullet}_{\, !}(\SGK, \tM_{\lambda, E}) := \mathrm{Image}(\mathfrak{i}^\bullet).$ 
Complementary to inner cohomology is Eisenstein cohomology, defined as: 
$H^\bullet_{\rm Eis}(\PBSC, \tM_{\lambda,E}) := \mathrm{Image}(\mathfrak{r}^\bullet).$
Let us  note that neither \textit{inner cohomology} nor \textit{Eisenstein cohomology} 
is not an honest-to-goodness cohomology theory, however, the terminology is very convenient and helps 
with our geometric intuition with the nature of the cohomology classes therein. Functorial properties on changing the base field $E$  
apply verbatim to the cohomology groups
$H^\bullet_?(\SGK, \tM_{\lambda,E}),$ where $? \in \{empty, c,!,\partial\};$ 
 by $H^\bullet_\partial(\SGK, \tM_{\lambda,E})$ we mean 
 $H^\bullet(\PBSC, \tM_{\lambda,E}).$ 
 
 \medskip
\subsubsection{\bf Cohomology of a boundary stratum $\partial_P\SGK$}
There is a spectral sequence whose $E_1^{pq}$-term is 
 built from the cohomology of various boundary strata $\partial_P\SGK$ that converges to the cohomology 
 of $\PBSC$ (see \cite[Sect.\,4.1]{harder-raghuram}). 
 For the moment $P$ will denote any (representative of a $G(\Q)$-conjugacy class of a) 
proper maximal parabolic subgroup of $G$, but from Sect.\,\ref{sec:root-system} it will denote 
a particular parabolic subgroup. The cohomology of any individual boundary strata $\partial_P\SGK$ is described in terms of representations induced from the cohomology of the Levi quotient of that boundary strata; we will briefly recall the essentials here while referring the reader to \cite[Sect.\,4.2]{harder-raghuram} 
for more details. It is convenient to pass to the limit over all open-compact subgroups $K_f$ and define
$$
H^\bullet(\partial_P\SG, \tM_{\lambda,E}) \ := \ \varinjlim_{K_f} \, 
H^\bullet(\partial_P\SGK, \tM_{\lambda,E}),
$$
on which there is an action of $G(\A_f)$; for a given $K_f$ we may retrieve $H^\bullet(\partial_P\SGK, \tM_{\lambda,E})$ by taking the $K_f$ invariants: 
$H^\bullet(\partial_P\SG, \tM_{\lambda,E})^{K_f}.$ 
Let $\kappa_P : P \to P/U_P = M_P$ be the canonical map from $P$ onto its Levi quotient $M_P$ by going modulo its unipotent radical $U_P.$ 
Define $K^{M_P}_\infty = \kappa_P(P(\R) \cap \overline{K}_\infty);$ later, we will carefully analyse the inner structure of $K^{M_P}_\infty,$ for a 
particular maximal parabolic subgroup. For any level structure $C_f \subset M_P(\A_f)$, define a locally symmetric space
\begin{equation}
\label{eqn:ul-S-M}
\uSMP_{C_f} \ = \ M_P(\Q) \backslash M_P(\A) / K^{M_P}_\infty C_f.
\end{equation}
Again, it is convenient to pass to the limit over all open-compact subgroups $C_f;$ define
$$
H^\bullet(\uSMP, \tM) \ := \ \varinjlim_{C_f} \, 
H^\bullet(\uSMP_{C_f}, \tM).
$$
We recall Prop.\,4.2 of {\it loc.\,cit.}: 

\begin{proposition}
\label{prop:bdry-coh-1}
The cohomology of the boundary stratum for $P$ is given by: 
$$
H^\bullet(\partial_P\SG, \tM_{\lambda,E}) \ = \ 
\aInd_{\pi_0(P(\R)) \times P(\A_f)}^{\pi_0(G(\R)) \times G(\A_f)} 
\left( H^\bullet (\uSMP, \widetilde{H^\bullet(\fu_P, \M_{\lambda,E})} ) \right),
$$
where the notations are: the Lie algebra of $U_P$ is $\fu_P$; 
$H^\bullet(\fu_P, \M_{\lambda,E})$ is the Lie algebra cohomology of $\M_{\lambda,E}$ that is naturally an $M_P$-module giving us 
a sheaf on $\uSMP_{C_f};$ the cohomology group on the right hand side is a module for $\pi_0(M_P(\R)) \times M_P(\A_f);$ note
that $\pi_0(M_P(\R)) = \pi_0(P(\R))$; 
$\aInd$ stands for algebraic, or un-normalized, induction which is the usual notion for $P(\A_f)$ to $G(\A_f)$, whereas induction from $\pi_0(P(\R))$  to $\pi_0(G(\R))$
means that we take invariants under the kernel of the canonical map $\pi_0(P(\R)) \to \pi_0(G(\R)).$ 
\end{proposition}

A celebrated theorem of Kostant \cite{kostant} asserts that as an $M_P \times E$-module, one has a multiplicity-free decomposition: 
\begin{equation}
\label{eqn:kostant}
H^q(\fu_P, \M_{\lambda, E}) \ \simeq \  \bigoplus_{\substack{w \in W^P \\ l(w) = q}} \M_{w \cdot \lambda, E}.
\end{equation}
where $W^P$ is a special subset of the Weyl group of $G$ consisting of all {\it Kostant representatives} (see Sect.\,\ref{sec:kostant-reps} below), and 
for any element $w$ in the Weyl group, $l(w)$ will denote its length defined in terms of simple reflections; the twisted action of $w$ on a weight $\lambda$ 
is denoted $w \cdot \lambda;$ see Sect.\,\ref{sec:weyl-group-new}. 
When $P$ is of the form $\Res_{F/\Q}(P_0)$ for a parabolic subgroup $P_0$ of $G_0$, the elements of $W^P$ and the above
decomposition of unipotent cohomology can be parsed over $\tau : F \to E$ as described in \cite[Sect.\,4.2.2]{harder-raghuram}. Now, 
Prop.\,\ref{prop:bdry-coh-1} together with \eqref{eqn:kostant} gives: 
 
\medskip
\begin{proposition}
\label{prop:bdry-coh-2}
The cohomology of $\partial_P\SG$ is given by
$$
H^q(\partial_P\SG, \tM_{\lambda, E}) \ = \ 
\bigoplus_{w \in W^P}
\aInd_{\pi_0(P(\R)) \times P(\A_f)}^{\pi_0(G(\R)) \times G(\A_f)}
\left(H^{q - l(w)}(\uSMP, \tM_{w \cdot \lambda, E}) \right).
$$
\end{proposition}

 \medskip
 
 We will be especially interested in certain summands in the cohomology in some special degrees of the boundary stratum $\partial_P\SG$ for one particular 
 maximal parabolic subgroup $P$ of $G$ that we describe in the next subsection.

\medskip
\subsection{The root system for $G_0$ and the maximal parabolic subgroup $P_0$}
\label{sec:root-system}
The root system for $G_0$ with respect to the split torus $T_0$, is given by 
$\Delta_{G_0}:  =  \left\{ \pm e_i \pm e_j: 0 \leq i < j \leq n \right\},$ 
where $e_i$ is the character defined by $e_i (\text{diag} [ t_0, \ldots, t_n, t_n^{-1}, \ldots, t_{0}^{-1} ])= t_i$ for each $0 \leq i \leq n$.
The positive roots with respect to the Borel subgroup $B_0$ are: 
$\Delta^{+}_{G_0} = \left\{  e_i \pm e_j: 0 \leq i < j \leq n \right\}$. Half the sum of positive roots for $G_0$ is:  
$$ 
{\brho}_{G_{0}} =  \sum \limits_{j=0}^{n} (n-j) e_j  = ne_0 + (n-1) e_1 \ldots + 1e_{n-1} + 0e_n.
$$
The simple roots are: 
\[\Pi_{G_0}
:=  
 \{ \alpha_0 = e_0 - e_1, \ \ \alpha_1 = e_1 - e_2, \ \ \ldots \ \, \ 
 \alpha_{n-1} = e_{n-1} - e_n,  \ \ \alpha_n = e_{n-1} + e_n\}. 
 \]

 \medskip
 
Henceforth, we let $P_0$ be the maximal parabolic subgroup of $G_0$ corresponding to deleting the {\it first} simple root $\alpha_0 = e_0 - e_1$. The unipotent radical 
$U_{P_0}$ of $P_0$ is given by 
$$
\left\{ u_{y_1, y_2, \ldots, y_{2n}} = 
\left( \begin{array}{ccccccc} 
  1 & y_1 & y_2 & \ldots & y_{2n} & 0 \\
                          & 1  &  & &  &-y_{2n} \\ 
                          & & 1  & & & -y_{2n-1} \\  
                          & &   &   \ddots & & \vdots \ \\
                          & &  & &  \ddots  &  -y_{1} \\
                          & &   & & & 1
                          \end{array}\right) \ \ | \ \  
                          y_1,\dots,y_{2n} \in   \mathbb G_{a} \right\}. 
$$ 
The Levi decomposition is $P_0 = M_{P_0} U_{P_0},$ where 
$$
M_{P_0} = P_0 / U_{P_0}   \cong A_0 \times {}^{\circ}M_{P_0}, \quad A_0 \cong \GL_1/F, \quad {}^{\circ}M_{P_0} \cong \rO(2n)/ F.
$$ 
In matrix form, $M_{P_0}$ is given by
$$
M_{P_0} =
  \left\{   
\left(  \begin{array}{ccc}
x&& \\ &g& \\&&x^{-1}
  \end{array} \right) \ | \ 
  x \in  \mathbb \GL(1), \ g \in \rO(2n) 
 \right\}.
 $$

\medskip

We put $P = \Res_{F/\Q}(P_0),$ then $M_P = \Res_{F/\Q}(M_{P_0})$ is its Levi quotient. The parabolic subgroup $P$ of $G$ being fixed, we often drop $P$ from the notation. The diagonal torus $T_0$ 
of $G_0$ is contained in $M_0 := M_{P_0}$, and we can write it as 
$T_0 = A_0 \times {}^\circ T_0$ where ${}^\circ T_0$ is the diagonal torus in ${}^{\circ}M_0 := {}^\circ M_{P_0}.$ After restricting scalars from $F$ to $\Q$, we
have $T = A \times {}^\circ T.$  The discussions about locally symmetric spaces, finite-dimensional algebraic representations, the corresponding sheaves and their cohomology for $G$, are all applicable {\it mutatis mutandis} to 
${}^\circ \! M := {}^\circ \! M_P = \Res_{F/\Q}({}^\circ \! M_{P_0}).$

 \medskip
\subsection{Dimensions of certain ad\`elic locally symmetric spaces}
\label{sec:dim-loc-sym-spaces}
Some piquant numerology involving various dimensions plays a role in our cohomology computations; for later use, we record some relevant dimensions. 
For notational simplicity, we will henceforth denote the groups $M_{P_0}/F$ and $^\circ M_{P_0}/F$ by $M_0$ and $^\circ M_{0}$, respectively. 
The dimension of the locally symmetric space $\SGK$ for $G = \Res_{F/\Q}(\rO(2n+2)/F)$ with level structure $K_f$ is 
$$
{\rm dim}(\mathcal S^G_{K_f}) \ = \ {\sf r_F} \cdot (n+1)^2.
$$ 
The dimension of its Borel--Serre boundary $\partial \mathcal S^G_{K_f}$ is 
$$
{\rm dim}(\partial \mathcal S^G_{K_f}) \ = \ {\sf r_F} \cdot (n+1)^2 - 1.
$$ 

\smallskip
Now consider $^{\circ}\!M = \Res_{F/\Q}(\rO(2n)/F)$;  the connected component of the identity of the maximal compact subgroup $K_{{}^\circ\!M(\R)}$ of 
${}^\circ\!M(\R)$ is: 
$$
K^{\circ}_{{}^\circ\!M(\R)} \ = \  \prod \limits_{\tau \in \Sigma_F} \SO(n) \times \SO(n), 
$$
and for any open-compact subgroup $K_{{}^{\circ}\!M_f}$ of ${}^{\circ}\!M(\A_f)$ we define: 
\begin{equation}
\label{eqn:S-circ-M}
\cS^{{}^\circ\!M}_{K_{{}^{\circ}\!M_f}} \ = \ {}^{\circ}\!M(\Q) \backslash {}^{\circ}\!M(\A) / K^{\circ}_{{}^\circ\!M(\R)} K_{{}^{\circ}\!M_f}.
\end{equation}
Whereas for $\SGK$ we divided by $\overline{K}_\infty$ which is in between the full maximal compact subgroup at infinity and its connected component, 
for ${\rm dim}(S^{{}^{\circ}\!M}_{K_{{}^{\circ}\!M_f}})$ we divide by the connected component of the maximal compact subgroup at infinity. We have 
$$
{\rm dim}(\cS^{{}^{\circ}\!M}_{K_{{}^{\circ}\!M_f}}) \ = \ {\sf r_F} \cdot n^2, 
$$
and its middle dimension (since $n$ is even) is: 
$$
q_m \ := \ \dfrac{1}{2} ~{\rm dim} ~\mathcal S^{^{\circ}M}_{K_{{}^{\circ}M,f}} \ = \ \dfrac{ {\sf r_F} \cdot n^2}{2}.
$$ 
This particular middle-dimension is interesting, since, it is only in this middle-degree that a discrete series representation of ${}^{\circ}M(\R)$ has nontrivial cohomology.

\smallskip
Write $M = A \times {}^{\circ}\!M$ where $A = \Res_{F/\Q}(\GL_1)$. Note that $A$ contains a copy of $\mathbb G_m/\Q$, which we denote by $S$. 
We have the group of 
real points: 
$$
A(\R) = \prod \limits_{\tau \in \Sigma_F} \R^{\times}, \quad {}^{\circ}\!M(\R) = \prod \limits_{\tau \in \Sigma_F} \rO(n,n)(\R), \quad
M(\R) = A(\R) \times {}^{\circ}\!M(\R). 
$$
The group $S(\R) = \R^\times$ sits diagonally in $A(\R)$. Let $\kappa_P : P \to P/U_P = M$ be the canonical surjection. Let $C_f$ be the open compact subgroup of $M(\A_f)$ defined by 
$$
C_f = \kappa_P (K_f \cap P(\A_f)).
$$
We suppose that $K_f$ is such that we can write $C_f = B_f \times {}^{\circ}C_f,$ for an open compact subgroup $B_f$ of $A(\A_f)$ and an open compact subgroup ${}^{\circ}C_f$ of ${}^{\circ}M(\A_f)$. Define the locally symmetric space for $A$ as 
$$
\cS^{A}_{B_f} \ := \ A(\Q) \backslash A(\A) / S(\R)^{\circ}B_f \ = \ F^{\times} \backslash \A_{F}^{\times} / S(\R)^{\circ}B_f; \quad 
\dim(\cS^{A}_{B_f}) = {\sf r}_F-1, 
$$
and the locally symmetric space for $M$ as:
\begin{equation}
\label{eqn:S-M}
\cS^{M}_{C_f} \ := \ \cS^{A}_{B_f} \times \cS^{{}^{\circ}\!M}_{{}^{\circ}C_f}.
\end{equation}
Hence, we see: 
$$
\dim(\cS^{M}_{C_f}) \ = \ {\sf r}_F -1 + {\sf r}_F \cdot n^2 \ = \ 2q_m + {\sf r}_F -1.
$$
For later use, we denote
\begin{equation}
\label{eqn:qb-and-qt-for-M}
\begin{split}
q_b \ & := \ q_m \ = \ {\sf r}_F \cdot n^2/2, \\ 
q_t \ & := \ q_m + {\sf r}_F -1 \ = \ {\sf r}_F \cdot n^2/2 + {\sf r}_F -1. 
\end{split}
\end{equation}
Note: $q_b + q_t = \dim(\cS^{M}_{C_f}).$ Furthermore, let's denote: 
\begin{equation}
\begin{split}
\label{eqn:qb-and-qt-for-G}
\q_b  \ & := \ q_b + \tfrac{1}{2}\dim(U_P) \ = \ {\sf r}_F(n^2/2 + n), \\ 
\q_t  \ & := \ q_t + \tfrac{1}{2}\dim(U_P) \ = \ {\sf r}_F(n^2/2 + n +1)-1.
\end{split}
\end{equation}
Note: $\q_b + \q_t = {\rm dim}(\partial \mathcal S^G_{K_f}).$ The subscripts $b,$ $m,$ and $t$ are meant to suggest `bottom', `middle', and `top'.

\medskip
\subsection{Cohomology of spaces associated to the Levi subgroup $M_P$}
\label{sec:coh-Levi} 
 
There is a delicate interplay between the various `locally symmetric spaces' attached to the Levi quotient $M_P$ of our parabolic subgroup $P.$ As before, 
we abbreviate $M_P$ as $M.$ 
Here we clarify 
the mutual relationship between these spaces and their cohomology. Fix a level structure $C_f \subset M(\A_f)$ which we assume to be of the form $B_f \times {}^\circ\!C_f$ as in the previous subsection. Define 
$\tilde{\cS}^{A}_{B_f} := A(\Q) \backslash A(\A) /B_f,$ and let $\theta_{A} : \tilde{\cS}^{A}_{B_f} \to \cS^{A}_{B_f}$
denote the canonical map which is a $S(\R)^\circ$-fibration. Now define: 
\begin{equation}
\label{eqn:tilde-S-M}
\tilde{\cS}^{M}_{C_f} \ := \ \tilde{\cS}^{A}_{B_f} \times \cS^{{}^{\circ}\!M}_{{}^{\circ}C_f}, 
\end{equation}
which comes with the fibration $\theta_{M} : \tilde{\cS}^{M}_{C_f} \to \cS^{M}_{C_f},$ defined as 
$\theta_M = \theta_A \times 1_{\cS^{{}^{\circ}\!M}_{{}^{\circ}C_f}}.$
We have the following spaces: 

\smallskip
\begin{itemize}
\item $\cS^{M}_{C_f}$; defined in \eqref{eqn:S-M}; it's the space that we can best refer to as the locally symmetric space attached to $M$. 

\smallskip 
\item $\tilde{\cS}^{M}_{C_f}$, defined in \eqref{eqn:tilde-S-M}; it comes with the fibration $\theta_M : \tilde{\cS}^{M}_{C_f} \to \cS^{M}_{C_f}.$ 

\smallskip 
\item $\uSM_{C_f}$, defined in \eqref{eqn:ul-S-M}; note that we have a finite-cover $\tilde{\cS}^{M}_{C_f} \to \uSMP_{C_f}$ whose fibre is 
the kernel of the map $\pi_0(M(\R)) \to \pi_0(G(\R)).$ 
\end{itemize}

\medskip
Let us  clarify the relation between the cohomology groups attached to these spaces and a sheaf $\tM$ (which for us will be of the form 
$\tM_{\mu,E}$ for a dominant integral weight $\mu$ for $M$) that is defined compatibly on all these three spaces. 
Denote 
\begin{equation}
\label{eqn:kernel-pi-0}
\overline{\pi_0}(M(\R)) \ := \ {\rm Kernel}\left(\pi_0(M(\R)) \to \pi_0(G(\R)) \right).
\end{equation}
In the description of the cohomology of a boundary stratum (see Prop.\,\ref{prop:bdry-coh-2}) 
one is naturally led to the cohomology of $\uSM_{C_f}$, which is related to 
the cohomology of $\tilde{\cS}^{M}_{C_f}$ via: 
 \begin{equation}
\label{eqn:coh-tilde-S-M-coh-ul-S-M}
H^\bullet(\uSM_{C_f}, \tM) \ = \ H^\bullet(\tilde{\cS}^{M}_{C_f}, \tM)^{\overline{\pi_0}(M(\R))}. 
\end{equation}
Next, the cohomology of the locally symmetric space $\cS^{M}_{C_f}$ can be pulled-back, via the map induced by 
$\theta_{M}$, to the cohomology of $\tilde{\cS}^{M}_{C_f}$, i.e., we have a map: 
\begin{equation}
\label{eqn:coh-S-M-coh-tilde-S-M}
\theta_M^* : H^\bullet(\cS^{M}_{C_f}, \tM) \ \to \ H^\bullet(\tilde{\cS}^{M}_{C_f}, \tM). 
\end{equation}

\medskip
\section{Discrete series representations and relative Lie algebra cohomology}
\label{section-discrete-series-representations-and-cohomology}

For an archimedean place $v$ of $F$, in this section we review some basics for discrete series representations of 
${}^{\circ}\!M_{P_0}(F_v)$ which is the real split even orthogonal group $\rO(n,n)(\R).$ Some care needs to be exercised because of the disconnectedness 
of $\rO(n,n)(\R).$ We are especially interested in the action of its group of connected components on relative Lie algebra cohomology of discrete series representations.

\subsection{Discrete series representations for a connected semisimple Lie group} 
\label{sec:dis-repn-general}
We briefly review the general theory which is due to Harish-Chandra, and refer the reader to Borel--Wallach \cite[II.5]{borel-wallach} for more details and references. Just for this subsection, we let $G$ stand for a connected semisimple Lie group, and $K$ a maximal compact subgroup of $G.$  Let $\mathfrak g$ and $\mathfrak k$ be the Lie algebras of $G$ and $K$, respectively. Denote their complexfications 
as $\mathfrak{g}_\C$ and $\mathfrak{k}_\C,$ respectively.
Suppose that ${\rm rank}(G) = {\rm rank}(K)$.
Let $\mathfrak h$ be a Cartan subalgebra of both $\mathfrak g$ and $\mathfrak k$, and 
$\mathfrak h_{\C}$ its complexification.
Let $\Phi$ be the root system of $\mathfrak g_{\C}$ with respect to  $\mathfrak h_{\C}$ and let $P(\Phi)$ be the weight lattice for $\Phi$. A weight $\Lambda \in P(\Phi)$ is called regular if $\langle \Lambda, \alpha \rangle \neq 0 ~\text{for any} ~\alpha \in \Phi.$ A regular weight $\Lambda$ gives a 
positive system of roots and their half-sum: 
$$
\Phi^+: = \{ \alpha \in \Phi: \langle \Lambda, \alpha \rangle > 0 \}, \quad \brho:  = \dfrac{1}{2} \sum \limits_{\alpha \in \Phi^+ } \alpha.
$$
Denote the set of compact roots by $\Phi_K$; a compact root is a root whose rootspace lies in $\mathfrak k_{\C}$; 
the positive system $\Phi^+$ gives a positive system for compact roots, and their half-sum: 
 $$
 \Phi^+_K: = \Phi_K \cap \Phi^+; \quad \brho_K:  = \frac{1}{2} \sum \limits_{\alpha \in \Phi^+_K } \alpha. 
 $$
Let $\gamma : Z(\mathfrak U(\mathfrak g_{\C})) \iso {\rm Sym}(\mathfrak h_{\C})^{W_{G}} $ be the Harish-Chandra isomorphism 
for the center $Z(\mathfrak U(\mathfrak g_{\C}))$ of the enveloping algebra $\mathfrak U(\mathfrak g_{\C})$ of $\mathfrak g_{\C};$ 
the Weyl group of $G$ is denoted $W_G.$ 
For each regular weight $\Lambda \in P(\Phi)$, we can associate a discrete series representation  $\pi_{\Lambda}$ of $G$ such that: 

\begin{enumerate}
\item The infinitesimal character of  $\pi_{\Lambda}$ is $\chi_{\Lambda}$ defined by
$\chi_{\Lambda}(z) = \Lambda(\gamma(z))$ for all $z \in Z(\mathfrak U(\mathfrak g_{\C})).$
Thus, $\pi_{\Lambda}$ and $\pi_{\Lambda'}$ have same  infinitesimal character  if and only if there is $w \in W_G$   
such that $w\lambda = \lambda'$.

\item The lowest $K$-type of ${\pi_{\Lambda}}|_{K}$ has multiplicity one and its highest weight (called the Blattner parameter) is 
given by $ \Lambda + \brho - 2\brho_K$. 

\item The representations $\pi_{\Lambda}$ and $\pi_{\Lambda'}$ are equivalent if there is $w \in W_K,$ the Weyl group of $K$, 
such that $w\Lambda = \Lambda'$.\medskip
\end{enumerate}

The following proposition describes the cohomology of a discrete series representation; see \cite[Prop.\,II.5.3]{borel-wallach}.

\begin{proposition} 
\label{cohomology-discrete-series}  Let $(\pi_{\Lambda}, V)$ be a discrete series representation as above. Let $V_K$ be the $(\mathfrak g, K)$-module of $K$-finite vectors in $V$. Let $\cM$ be a finite-dimensional complex irreducible representation of $G$ whose contragredient is denoted $\cM^{\sf v}.$

\begin{enumerate}
\item If the highest weight  of $\cM \neq \Lambda - \brho$, then ${\rm Ext}^{\bullet}_{\mathfrak g, \mathfrak k} (\cM, V_K) = 0.$
\item If the highest weight  of $\cM = \Lambda - \brho$, then ${\rm Ext}^{j}_{\mathfrak g, \mathfrak k} (\cM, V_K) = 0,$ unless 
$j = \tfrac{1}{2}\, {\rm dim}(G/K),$ and in this `middle-degree' it is one-dimensional. 
\item The relative Lie algebra cohomology $H^{\bullet}(\mathfrak g, \mathfrak k;\, \cM^{\sf v} \otimes H)$ is non-vanishing only when the highest weight  of 
$\cM$ is $\Lambda - \brho$ and the degree $\bullet = \tfrac{1}{2}\, {\rm dim}(G/K)$. In this case it is one-dimensional and supported on the Blattner parameter.
\end{enumerate}
\end{proposition}

\medskip

\subsection{Discrete series representations of orthogonal groups}
\label{sec:dis-ser-rep-orth}

Fix an archimedean place $v$ of $F.$ Just for this Sect.\,\ref{sec:dis-ser-rep-orth}, for brevity, we introduce these notation: 

\smallskip
${}^\circ\!M := {}^{\circ}\!M_{P_0 v} = {}^{\circ}\!M_{P_0}(F_v) = \rO(n,n)(\R),$ the real split even orthogonal group; 

\smallskip
${}^\circ\! M^\circ := {}^{\circ}\!M_{P_0 v}^\circ = \SO(n,n)(\R)^\circ,$ the connected component of identity in ${}^\circ\!M$; 

\smallskip
$K := K_{{}^{\circ}\!M_{P_0 v}} = \{g  \in   M \ | \  {}^{\intercal}g\,g = I_{2n} \}$, the maximal compact subgroup of ${}^\circ\!M$; 

\smallskip
$K^\circ := K_{{}^{\circ}\!M_{P_0 v}}^\circ,$ the connected component of identity in $K,$

\smallskip
${}^{\circ}\mathfrak m$ the Lie algebra of ${}^\circ \!M$ or of ${}^\circ \!M^\circ$, and $\mathfrak k$ the Lie algebra of $K$ or of $K^\circ$.

\subsubsection{\bf Representatives for the group of connected components}
\label{sec:maximal-compacts}
The compact group $K$ is isomorphic to ${\rm O}(n)(\R) \times {\rm O}(n)(\R)$, the isomorphism being the map $h$ defined in \eqref{eqn:h-map}, but after replacing $n+1$ by $n$. 
Inclusion induces an equality $\pi_0(K) = \pi_0({}^{\circ}\!M)$ of component groups; we give explicit representatives for elements
of these component groups. Recall that $\delta_n = {\rm diag}(-1,1,\dots,1)$. Then $\pi_0(\rO(n)(\R)) = \{I_n, \delta_n\}.$ Further, if $g \in \rO(n)(\R)$ then 
clearly $J_n\,g\,J_n \in \rO(n)(\R).$ Hence, $J_n \, \delta_n \, J_n$ also represents the nontrivial element of $\pi_0(\rO(n)(\R)).$ We take as representatives for the connected components of ${\rm O}(n)(\R) \times {\rm O}(n)(\R)$ the elements 
$$
\{(I_n,I_n), \ (I_n, \delta_n), \ (J_n\delta_n J_n, I_n), \ (J_n\delta_n J_n, \delta_n)\}.
$$
For any integer $m \geq 1$, let 
\begin{equation}
\label{eq:kappa-2n}
s_{2m}: = 
\begin{pmatrix} I_{m-1} & & \\ & -I_2 & \\ & & I_{m-1} \end{pmatrix}, \quad \text{and} \quad 
\kappa_{2m} : = \begin{pmatrix} I_{m-1} & & \\ & J_2 & \\ & & I_{m-1} \end{pmatrix}.
 \end{equation} 
One can check that $h(I_n,\delta_n) = \kappa_{2n}$, and $h(J_n\delta_n J_n, \delta_n) = s_{2n}.$ Furthermore, one sees: 
$$
\kappa_{2n} \in \rO(n,n)(\R) \setminus \SO(n,n)(\R), \quad s_{2n} \in \SO(n,n)(\R) \setminus \SO(n,n)(\R)^{\circ}.
$$ 
The component group $\pi_0(K)$ of $K$ (as also $\pi_0({}^{\circ}M)$ of ${}^{\circ}M$) is 
isomorphic to $\Z/2\Z \times \Z/2\Z,$ and as representatives we take: 
$$
\left\{I_{2n}, \ s_{2n}, \ \kappa_{2n}, \ s_{2n} \kappa_{2n} \right\}.
$$
Note that $s_{2n} \kappa_{2n} = \kappa_{2n}s_{2n}.$

\subsubsection{\bf Compact Cartan subgroups and subalgebras}
\label{Compact Cartan subalgebras}
At the level of  Lie algebras, the same map $h$ gives us $\mathfrak t_K$,  a compact Cartan subalgebra that is shared by both 
${}^\circ \mathfrak{m}$ and $\k_{{}^\circ \!M}$, the Lie algebras of ${}^{\circ}\!M$ and $K,$ respectively. 
Let $a_{\theta}: = \left(\begin{smallmatrix} 0 & -\theta \\ \theta & 0 \end{smallmatrix}\right)$  for $\theta \in \R$. Given $\theta_1, \ldots, \theta_n \in \R$, let $k$ and $l$ be elements of ${\mathfrak s \mathfrak o}(n)$, the Lie algebra of the compact $\SO(n)(\R)$, defined by 
$$ 
k =  \begin{pmatrix} a_{\theta_1} && \\ &\ddots& \\ && a_{\theta_{r}} \end{pmatrix},
\quad 
l =   \begin{pmatrix} a_{\theta_{r +1}} && \\ &\ddots& \\ && a_{\theta_{n}} \end{pmatrix}; \quad \mbox{recall that $n = 2r$}.
$$
Then, we check that $h(k,l)$ is
$$ 
\begin{pmatrix} 
 a_{(\theta_1-\theta_n)/2} &&&&& a_{(\theta_1+\theta_n)/2}J_{2} \\ 
 & \ddots && &\iddots & \\
 && a_{(\theta_r-\theta_{r+1})/2}  & \quad a_{(\theta_r+\theta_{r+1})/2}J_2 && \\
 \\
 && a_{-(\theta_r+\theta_{r+1})/2} J_2 & \quad a_{(\theta_{r+1}-\theta_{r})/2} && \\
 & \iddots && & \ddots & \\
a_{-(\theta_1+\theta_n)/2}J_2 &&&&& a_{(\theta_n-\theta_1)/2} 
\end{pmatrix}. 
$$ 
Reparameterizing our variables: 
$$
t_1 = \frac{\theta_1-\theta_n}{2}, \ \ \ t_2 = \frac{\theta_1+\theta_n}{2}, \ \ \dots, \ \ \ t_{n-1} = \frac{\theta_r-\theta_{r+1}}{2}, \ \ \ t_n = \frac{\theta_r+\theta_{r+1}}{2},
$$
we can describe the elements of the compact Cartan subalgebra$ \mathfrak{t}_K$ as the set of all 
$$
\left\{X_{t_1,\dots,t_n} :=  
\begin{pmatrix} 
 a_{t_1} &&&&& a_{t_2}J_{2} \\ 
 & \ddots && &\iddots & \\
 && a_{t_{n-1}}  & \quad a_{t_n}J_2 && \\
 && a_{-t_n} J_2 & \quad a_{-t_{n-1}} && \\
 & \iddots && & \ddots & \\
a_{-t_2}J_2 &&&&& a_{-t_1} 
\end{pmatrix} \right\} 
$$
with $t_1, \dots, t_n \in \R .$ 
Let $T_K$ be the analytic subgroup associated to $\mathfrak t_K$ inside $K^{\circ}$. Consider the complexification $\mathfrak t_{K,\C}$ of $\mathfrak t_K$; 
its elements are of the form $ X_{t_1, \ldots, t_n},$ where $t_j \in \C$. 
It can be observed that all $ X_{t_1,\dots,t_n}  \in  \mathfrak t_{K,\C}$ are simultaneously conjugate to elements of the complexified diagonal algebra 
$\mathfrak t_{\C}$. The conjugating matrix can be taken in ${\rm O}(n,n)(\C)$. Under this correspondence, the matrix $ X_{t_1,\dots,t_n}  \in  \mathfrak t_{K,\C}$ is conjugate to the diagonal matrix 
${\rm diag}[{\mathbf i}\theta_1, \ldots, {\mathbf i}\theta_n, -{\mathbf i}\theta_n, \ldots,  -{\mathbf i}\theta_1] \in \mathfrak t_{\C}.$

\medskip
\subsubsection{\bf Roots and compact roots for $\rO(n,n)(\R)^\circ$}
\label{sec:Compact-roots}

Recalling that ${\mathbf i}\R$ is the Lie algebra of the circle group $S^1$, we define the linear functionals $e_i^K : \mathfrak t_K \to {\mathbf i} \R,$ for $1 \leq i \leq n,$ as: 
\begin{equation}\label{compact weights}
\begin{split}
e_1^K (X_{t_1,\dots,t_n}) &  :=  \ {\mathbf i}(t_1 +  t_2) = {\mathbf i} \theta_1, \quad e_2 ^K(X_{t_1,\dots,t_n}) : =  \ {\mathbf i}(t_1 - t_2) = -{\mathbf i} \theta_n, \\
 & \vdots  \\
e_{n-1}^K (X_{t_1,\dots,t_n}) &: = \ {\mathbf i}(t_{n-1}+t_n) = {\mathbf i} \theta_r, \quad 
e_n^K (X_{t_1,\dots,t_n}) :=  {\mathbf i}(t_{n-1}-t_n) = -{\mathbf i} \theta_{r+1}. 
 \end{split}
 \end{equation}
It's clear then that $\{e_1^K,e_3^K,\dots,e_{n-1}^K\}$ (resp., $\{e_2^K,e_4^K,\dots,e_n^K\}$)
are the linear functionals corresponding to the first (resp., second) summand of 
$\so(n) \oplus \so(n).$
The root system $\Phi$ of ${}^{\circ}\mathfrak{m}_\C$ with respect to 
$\mathfrak t_{K,\C}$ is given by 
$$
\Phi = \{ \pm e_i^K \pm e_j^K \ | \  1 \leq i <  j \leq n \}. 
$$
The set of all compact roots $\Phi_K$ is: 
 $$
 \Phi_K = \{ \pm e_i^K \pm e_j^K \ | \ 1 \leq i <  j \leq n, \ \ i \equiv j \ (\rm {mod} \ 2)\}.
 $$
The weight lattice is given by 
$$
P_{\Phi}: = \{ \sum \limits_{1 \leq j \leq n}   a_j e_j^K \ | \  a_j \in \Z \}.
$$
An element $\Lambda = \sum \limits_{1 \leq j \leq n}   a_j e_j^K \in P_{\Phi}$ is regular if and only if  $\langle \Lambda, \alpha \rangle \neq 0$ for all 
$\alpha \in  \Phi$, which is equivalent to $ a_i \neq \pm a_j$ for all $i \neq j.$

\medskip
\subsubsection{\bf Weyl groups}

Let $W_{{}^\circ\! M}$  and $W_{K}$ be the Weyl groups of ${}^\circ\!M$ and $K$, respectively. 
Recall that for an even integer $2m$, the Weyl group of the root system of type $D_m$ is isomorphic to $\{\pm 1\}^{m-1} \ltimes S_m$ which is the subgroup 
of the group $\{\pm 1\}^m \ltimes S_m$ of all signed permutations where the total number negative signs is even. For us, 
this applies to ${}^\circ\!M  = \rO(n,n)(\R)$, and also to each of the two factors of $K \simeq \rO(2r) \times \rO(2r).$ In particular, $W_{K}$ is a subgroup of $W_{ {}^\circ\!M}$ of index 
$ 2 (n!)/(r!)^2$.

\medskip
\subsubsection{\bf Discrete series representations of ${}^\circ\!M^\circ = \SO(n,n)(\R)^\circ$}
The Weyl group
$W_{{}^\circ\! M}$ acts on the regular elements in $P_{\Phi}$.  
There is a bijective correspondence between the $W_{K}$ orbits of regular elements in $P_{\Phi}$ and equivalence classes of discrete series 
representations of ${}^\circ\!M^\circ$.
Fix a regular $\Lambda.$ 
Define ${\brho}^c$ to be the half-sum of all roots $\alpha$ for which $\langle \Lambda, \alpha \rangle > 0$ and define $\brho_K^c$ to be the half-sum of all compact roots $\alpha$ for which $\langle \Lambda, \alpha \rangle > 0$. 
Let $\pi_{\Lambda}$ be the discrete series representation of the connected simple group ${}^\circ\!M^\circ$ 
associated to a regular element $\Lambda$ under the above bijective correspondence. The infinitesimal character of $\pi_{\Lambda}$ is $\chi_{\Lambda},$ and the Blattner parameter of $\pi_{\Lambda}$ is $\Lambda + \brho^c - 2\brho_K^c$.

\medskip
\subsubsection{\bf Dominant weights, discrete series, and cohomology}
\label{sec:weights-DS-cohomology}

Suppose $\mu_1, \ldots, \mu_n \in \Z$ are such that $\mu_1 \geq \mu_2 \geq \ldots \geq \mu_{n-1} \geq | \mu _{n} |$. 
Furthermore, we will assume that $|\mu_n| \geq 1.$ 
Define $\Lambda$ by:
$$ 
\Lambda =   \sum \limits_{1 \leq j \leq n}   (\mu_j + n - j) \, e_j^K.
$$
Thus, $ \Lambda $ is a regular element of $P_{\Phi}$ and the set of roots $\alpha$ for which $\langle \Lambda, \alpha \rangle > 0$ is precisely the 
set $\{ e_i ^K\pm e_j^K\ | \ 1 \leq i < j \leq n \}$. So $\brho^c =  \sum \limits_{1 \leq j \leq n}   ( n - j) \, e_j^K$ and 
$$
\Lambda - \brho^c =  \sum \limits_{1 \leq j \leq n}   \mu_j  \, e_j^K \ =: \ \mu.
$$
Let $D_{\mu}$ be a discrete series representation of ${}^\circ \! M^{\circ} = \SO(n,n)(\R)^\circ$ equivalent to $\pi_{\Lambda}$.  By an abuse of notation, we also denote the $({}^{\circ}\mathfrak {m}, K^\circ)$-module of $K^\circ$-finite vectors by $D_{\mu}$.
Let $\mathcal N_{\mu}$ be the finite-dimensional irreducible representation of ${}^\circ \!M = \SO(n,n)(\R)$ with highest weight $\mu$ with respect to the maximal compact torus $T_K$. Let us  bear in mind that the representation $\cN_\mu$ admits a $\Q$-structure since it comes from a purely algebraic theory of highest weight modules for the real points of the split {\it algebraically} connected algebraic group $\SO(n,n)$.
For any even positive integer $2m$, since the dimension of $\SO(2m)$ is $m(2m-1),$ 
we see that the dimension of the symmetric space $\rO(n,n)(\R)/(\rO(n) \times \rO(n))$ is $n^2.$
From Prop.\,\ref{cohomology-discrete-series}, we conclude that 
$$
H^{q}({}^{\circ}\fm, \fk;\, D_{\mu} \otimes  \mathcal N_{\mu}^{\sf v} ) = 
\begin{cases}
0 & \mbox{if $q \neq q_0 := \dfrac{n^2}{2}$, and} \\ 
\C & \mbox{if $q = q_0$}. 
\end{cases} 
$$
Note that $\mathcal{N}_\mu$ is a self-dual representation, i.e., $\mathcal N_{\mu}^{\sf v} = \mathcal N_{\mu};$ see, for example, 
\cite[Lem.\,5.0.2]{baskar-raghuram}.

\medskip
\subsubsection{\bf The action of the group of connected components on cohomology}
\label{sec:def-boldface-D-mu}

As noted in Sect.\,\ref{sec:maximal-compacts}, 
$s_{2n} \in \SO(n,n)(\R) \backslash  \SO(n,n)(\R)^{\circ}$ and 
$\kappa_{2n}  \in {\rm O}(n,n)(\R) \backslash  \SO(n,n)(\R)$ 
generate the group of connected components of $K.$ 
The actions of $s_{2n}$ and $\kappa_{2n}$ on $\mathfrak t_K$ and on the weights $e_j^K$ are given by: 

\smallskip
\begin{itemize}
\item ${\rm Ad}(\kappa_{2n})$ maps $X_{t_1, \ldots, t_n}$ to 
$X_{t_1, \ldots, t_{n-2}, t _{n}, t_{n-1}},$  and thus: 
$$
{\rm Ad}(\kappa_{2n}) \cdot e_j^K = e_j^K, \  1 \leq j \leq n-1, \quad  {\rm Ad}(\kappa_{2n}) \cdot e_n^K =  - e_n^K.
$$

\smallskip
\item ${\rm Ad}(s_{2n})$ maps $X_{t_1, \ldots, t_n}$ to $X_{t_1, \ldots, t_{n-2}, -t _{n-1}, -t_{n}},$ and thus:
$$
{\rm Ad}(s_{2n}) \cdot e_j ^K= e_j^K, \  1 \leq j \leq n-2,
$$
$$
 {\rm Ad}(s_{2n}) \cdot e_{n-1}^K = - e_{n-1}^K, \quad  {\rm Ad}(s_{2n}) \cdot e_n^K =  - e_n^K.
$$
\end{itemize}

\smallskip

In particular, ${\rm Ad}(\kappa_{2n})$ is not the effect of an element of $W_{{}^\circ\!M}.$ Also, ${\rm Ad}(s_{2n})$ is obtained by the action of an element of 
$W_{{}^{\circ}\!M}$, however, from Sect.\,\ref{sec:Compact-roots} it follows that 
it is not realizable via $W_{K}.$ Since $|\mu_n| \geq 1$, 
we conclude that the discrete series representations 
$$
\{D_{\mu}, \quad {}^{s_{2n}}\!D_{\mu}, \quad {}^{\kappa_{2n}}\!D_{\mu}, \quad {}^{ s_{2n} \, \kappa_{2n}}\!D_{\mu}\}
$$  
obtained by conjugating by representatives of the group of connected components of ${}^\circ\!M$ are pairwise not equivalent to each other. 
It follows that 
$$
\mathbb D_{\mu} \ := \ 
{\rm Ind}^{\rO(n,n)(\R)}_{\SO(n,n)(\R)^{\circ}}(D_{\mu})
$$ 
is an irreducible representation of $\rO(n,n)(\R)$. 
Furthermore, its restriction to $\SO(n,n)(\R)^{\circ}$ is given by
$$
{\mathbb D_{\mu}}|_{\SO(n,n)(\R)^{\circ}} \ = \ 
D_{\mu} \ \oplus \ {}^{s_{2n}}\!D_{\mu} \ \oplus \ D_{^{\kappa_{2n}} \mu} \ \oplus \ {}^{s_{2n}}\!D_{^{\kappa_{2n}} \mu},
$$
after noting that ${}^{\kappa_{2n}}\!D_{\mu} = D_{^{\kappa_{2n}} \mu}$ by using the action of the torus on a highest weight vector, where, 
if $\mu = \mu_1  \, e_1^K +  \cdots + \mu_n  \, e_n^K $ then 
${}^{\kappa_{2n}}\mu = \mu_1  \, e_1^K + \cdots + \mu_{n-1}  \, e_{n-1}^K  - \mu_n  \, e_n^K.$
If we restrict ${\mathbb D}_\mu$ to the intermediate subgroup $\SO(n,n)(\R)$, the four representations fuse two at a time to 
give a multiplicity free direct sum of two irreducible representations of $\SO(n,n)(\R):$ 
$$
{\mathbb D_{\mu}}|_{\SO(n,n)(\R)} \ = \ \mathcal D_{\mu} \oplus \mathcal D_{^{\kappa_{2n}}\mu}, 
$$
where, 
$$
\mathcal D_{\mu}|_{\SO(n,n)(\R)^{\circ}} \ := \ D_{\mu} \oplus {}^{s_{2n}}\!D_{\mu}, \ \ \mbox{and} \ \ 
\mathcal D_{{}^{\kappa_{2n}}\mu}|_{\SO(n,n)(\R)^{\circ}} \ := \ D_{{}^{\kappa_{2n}}\mu} \oplus {}^{s_{2n}}\!D_{{}^{\kappa_{2n}}\mu}. 
$$
For the finite-dimensional coefficient systems, let's define
$$
\mathcal M_{\mu} \ := \ {\rm Ind}^{\rO(n,n)(\R)}_{\SO(n,n)(\R)^{\circ}} (\mathcal N_{\mu}).
$$
Restricting back to $\SO(n,n)(\R)^{\circ}$, we get
$$
\cM_{\mu}|_{\SO(n,n)(\R)^{\circ}} \ = \ \cN_\mu \ \oplus \ {}^{\kappa_{2n}}\cN_\mu.
$$
since $\cN_\mu$ is an irreducible representation of $\SO(n,n)(\R).$
Now we observe that the cohomology $H^{q_0}({}^{\circ}\fm, \fk;\,  \mathbb D_{\mu}  \otimes \mathcal M_{\mu})$ is a direct sum of four one-dimensional cohomology spaces: 
\begin{multline}
\label{eqn:m-k-coh-D-mu}
H^{q_0}({}^{\circ}\fm, \fk;\, D_{\mu}  \otimes \mathcal N_{\mu} ) 
 \ \oplus \ 
H^{q_0}({}^{\circ}\fm, \fk;\,  {}^{s_{2n}}D_{\mu} \otimes \mathcal N_{\mu} ) \ \oplus \\
H^{q_0}({}^{\circ}\fm, \fk;\,D_{{}^{\kappa_{2n}} \mu}  \otimes  \mathcal N_{^{\kappa_{2n}} \mu}  ) 
\ \oplus \   
H^{q_0}({}^{\circ}\fm, \fk;\,  {}^{s_{2n}} D_{^{\kappa_{2n}} \mu} \otimes \mathcal N_{^{\kappa_{2n}} \mu} ). 
\end{multline}

\medskip
\subsubsection{\bf The parameter of the discrete series representation ${\mathbb D_{\mu}}$}
\label{sec:dis-ser-param}

Let us  recall some preliminaries on the Weil group of $\R$: $W_\R = \C^\times \cup j\cdot \C^\times,$ with $jzj^{-1} = \bar{z}$ for $z \in \C^\times.$  
For an integer $\ell$, define $\chi_\ell : \C^\times \to \C^\times$ as $\chi_\ell(z) = (z/\bar{z})^{\ell/2}$ or $\chi_\ell(re^{i\theta}) = e^{i\ell\theta}$ for
$z = re^{i\theta} \in \C^\times.$ Define $I(\ell)$ to be the $2$-dimensional induced representation $\Ind_{\C^\times}^{W_\R}(\chi_\ell).$
Then $I(\ell) \simeq I(-\ell)$; $I(\ell)$ is irreducible if $\ell \neq 0$; $I(0) = 1\!\!1 \oplus {\rm sgn}.$ 

\smallskip
Now, given $\mu$ as above with $\mu_1 \geq \cdots \geq \mu_{n-1} \geq |\mu_n|$, put
\begin{equation}
\label{eqn:ell-and-mu}
\ell_1 = 2(\mu_1+n-1), \ \ \ell_2 = 2(\mu_2+n-2), \ \ \dots \ \ , \ \ell_{n-1} = 2(\mu_{n-1}+1), \ \ \ell_n = 2|\mu_n|.
\end{equation}
The Langlands parameter of ${\mathbb D_{\mu}}$ is the homomorphism $\varphi({\mathbb D}_\mu) : W_\R \to \rO(2n)(\C)$ defined by:
$$
\C^\times \ni z \mapsto (a_1,\dots,a_n,a_n^{-1},\dots,a_1^{-1}), \quad a_i = (z/\bar{z})^{\ell_i/2}; \quad j \mapsto J_{2n}. 
$$ 
This follows from from the very general \cite[Ex.\,10.5]{borel-corvallis}; for the particular example at hand, see also \cite[Sect.\,5.4]{raghuram-sarnobat}. 
The parameter is a direct sum of $n$ induced representations: 
\begin{equation}
\label{eqn:L-paramater-D_mu}
\varphi({\mathbb D}_\mu) \ = \ I(\ell_1) \oplus \cdots \oplus I(\ell_{n-1}) \oplus I(\ell_n).
\end{equation}
Since $|\mu_n| \geq 1,$ each $\ell_i > 0,$ ensuring the irreducibility of the $n$-summands. 

\smallskip 
Let us  add a comment that will be relevant in the next section. The parameter $\varphi({\mathbb D}_\mu)$ thought of as a parameter of a 
representation of $\GL_{2n}(\R)$, does not parametrise a cohomological representation. However, after any half-integral Tate twist, namely, for
any $k \in \Z$, the parameter $\varphi({\mathbb D}_\mu) \otimes |\ |^{(2k+1)/2}$ is indeed 
a parameter of a cohomological representation of $\GL_{2n}(\R).$ See, for example, 
\cite[2.4.1.2]{raghuram-forum}.

\subsubsection{{\bf Choice of a generic representation in $L$-packet}} \label{subsec:choice-generic-member-L-packet}

We fix a Whittaker datum $\psi$, i.e., a nontrivial additive character on the unipotent radical of the standard Borel subgroup of $^\circ M$ that is nontrivial on all simple root spaces. For a weight $\mu = (\mu^v)_{v \in S_{\infty}}$ and an archimedean place $v$, consider the discrete series representation ${\mathbb D}_{\mu^{v}} $ of ${\rm O}(n,n)(\R)$ as in Sect.\,\ref{sec:def-boldface-D-mu}. In the local $L$-packet (which is also a local $A$-packet in our context)  of ${\mathbb D}_{\mu^{v}} $, there is a unique discrete series representation ${\mathbb D}^{\dagger}_{\mu^{v}}$ which is $\psi_v$-generic, i.e., has a local Whittaker model with respect to $\psi_v$. 
In Sect.\,\ref{strongly-inner-cohomology}, we consider this representation ${\mathbb D}^{\dagger}_{\mu^{v}}$ while defining the strongly inner cohomology.

\medskip
\section{Strongly inner cohomology}
\label{sec:strongly-inner-coh}

We return to the global notations from Sect.\,\ref{section-preliminaries-notations}. In particular, recall that ${}^{\circ}\!M_0 \cong \rO(2n)/ F$, and 
${}^\circ\!M = \Res_{F/\Q}({}^\circ\!M_0).$

\medskip
\subsection{Strongly inner cohomology}
\label{strongly-inner-cohomology}
Let $\mu \in X^*_+({}^\circ T \times E)$ be a dominant integral weight for ${}^\circ\!M.$  
Recall that $\mu = (\mu^\tau)_{\tau \in \Hom(F,E)},$ where $\mu^\tau$ is an $n$-tuple of integers ordered as:  
$\mu^\tau_1 \geq \mu^\tau_2 \geq \cdots \geq \mu^\tau_{n-1} \geq |\mu^\tau_n|.$ 
Define 
$$
\mu_{{\rm min}} \ := \ {\rm min} \{ \, |\mu_n^{\tau} | \ : \ \tau \in {\rm Hom}(F,E) \}.
$$ 
We henceforth assume that $\mu_{\rm min} \geq 1.$ For an open-compact subgroup ${}^\circ C_f$ of 
${}^\circ\!M(\A_f)$, consider the sheaf $ \widetilde{\mathcal M}_{\mu,E}$ on 
$\cS^{{}^{\circ}\!M}_{{}^\circ C_f}.$

The inner cohomology $H^{\bullet}_{!} (\mathcal S^{{}^\circ \! M}_{{}^{\circ}C_{f}}, \widetilde{\mathcal M}_{\mu,E})$ is a module for the Hecke algebra
$\HH^{^{\circ}\!M}_{^{\circ}C_{f}} =  \otimes_{v \notin S_\infty} {\mathcal H}^{^{\circ}\!M_{v}}_{^{\circ}C_{v}}$ over $\Q$.  
Let $S$ be a finite set of places containing $S_\infty$, the ramified places of $F,$ and any place where ${}^{\circ}C_v$ is a proper subgroup of the maximal compact subgroup. 
Then, ${\mathcal H}^{{}^{\circ}\!M, S} = \mathbin{\mathop{\otimes}\displaylimits_{v \notin S}} {\mathcal H}^{^{\circ}\!M_{v}}_{^{\circ}C_{v}}$ is a central subalgebra of 
${\mathcal H}^{^{\circ}\!M}_{^{\circ}C_{f}}.$   
We consider inner cohomology as a $\HH^{^{\circ}\!M}_{^{\circ}C_{f}}$-module. It is a general fact that inner cohomology is semi-simple as a Hecke module. 
Define 
${\rm Coh}_{!}(^{\circ}\!M, \mu/E, {}^{\circ}C_{f})$ as the isomorphism classes of absolutely simple modules $\sigma_f$ over $E$ (these are necessarily finite-dimensional 
$E$-vector spaces) which appear in $H^{\bullet}_{!} (\mathcal S^{^{\circ}\!M}_{^{\circ}C_{f}}, \widetilde{\mathcal M}_{\mu,E}),$ 
and inner cohomology is decomposed as the sum of such isotypic components: 
$$
  H^{\bullet}_{!} (\mathcal S^{{}^{\circ}\!M}_{{}^{\circ}C_{f}}, \widetilde{\mathcal M}_{\mu,E}) \ =  \ 
  \bigoplus_{\sigma_f \in {\rm Coh}_{!}(M, \mu/E, {}^{\circ}C_{f})} 
 H^{\bullet}_{!} (\mathcal S^{{}^{\circ}\!M}_{{}^{\circ}C_{f}}, \widetilde{\mathcal M}_{\mu,E}) (\sigma_f);
$$
by taking $E$ large enough we can guarantee that the full isotypic component of a particular $\sigma_f$ appears as a summand.  Each $\sigma_f$ is isotypic 
as a ${\mathcal H}^{{}^{\circ}\!M, S}$-module.

\smallskip

Towards defining the strongly inner spectrum ${\rm Coh}_{!!}(^{\circ}\!M, \mu/E, {}^{\circ}C_{f})$, 
we fix once and for all a Whittaker datum $\psi$ for ${}^{\circ}\!M$ which is a 
nontrivial additive character on the unipotent radical of the standard Borel subgroup which is nontrivial on all simple root spaces. We also recall Arthur's 
classification, as finessed by Atobe--Gan \cite{atobe-gan}, of the discrete (and hence cuspidal) spectrum of $^{\circ}\!M(\A) = \rO(2n)(\A_F)$ in terms of automorphic representations of $\GL_{2n}(\A_F)$; the Arthur parameter of an irreducible automorphic representation $\varsigma$ appearing in the discrete spectrum 
of $^{\circ}\!M(\A)$ will be denoted $\Psi(\varsigma).$ 
Let us  also recall some notation from Sect.\,\ref{sec:def-iota-lambda}: an embedding $\iota : E \to \C$ gives an identification 
$X^*({}^\circ T \times E) \to X^*({}^\circ T \times \C)$ that we denote $\mu \mapsto {}^\iota\mu.$ For $\mu = (\mu^\tau)_{\tau : F \to E},$ 
identifying $\Hom(F,\C)$ with $S_\infty,$ and since $\iota$ induces a bijection $\Hom(F,E) \to \Hom(F,\C) = S_\infty$ via composition,  
for $v \in S_\infty$, the $v$-th component ${}^\iota\mu_v$ is $\mu^{\iota^{-1}v}.$


\begin{definition}
\label{def:strongly-inner-spectrum}
Let $\mu \in X^*({}^\circ T \times E)$ be a dominant integral weight for ${}^\circ\!M.$ Assume that $\mu_{\rm min} \geq 1.$ 
For an open compact subgroup ${}^{\circ}C_f$, 
the strongly inner spectrum of ${}^{\circ}\!M$ for the weight $\mu$, denoted ${\rm Coh}_{!!}(^{\circ}\!M, \mu/E, {}^{\circ}C_f)$  consists of those 
$\sigma_f \in {\rm Coh}_{!}(^{\circ}\!M, \mu/E, {}^{\circ}C_f)$ for which  
there exists an $\iota: E \rightarrow \C$ and a cuspidal automorphic representation $\varsigma$ 
 of  ${}^{\circ}\!M(\A) = \rO(2n)(\A_F)$ such that 
\smallskip 
\begin{enumerate}
\item[(i)] $\varsigma$ is globally generic with respect to $\psi$; 
\smallskip
\item[(ii)] the Arthur parameter $\Psi(\varsigma)$ of $\varsigma,$ an automorphic representation of $\GL_{2n}(\A_F),$ is cuspidal; 
\smallskip

\item[(iii)] for all $v \in S_{\infty},$ the $v$-th component $\varsigma_v$ is $\mathbb D^{\dagger}_{{}^\iota\!\mu_v}$, the 
$\psi_v$-generic discrete series representation of $\rO(n,n)(\R)$ determined by ${}^\iota\!\mu_v$
(see Sect.\,\ref{subsec:choice-generic-member-L-packet}); and  

\smallskip
\item[(iv)] $\varsigma_f^{{}^{\circ}C_f} \cong \sigma_f \otimes_{E, \iota} \C $ as ${\mathcal H}^{^{\circ}\!M}_{^{\circ}C_{f}}$-modules over $\C.$ 
\end{enumerate}
\end{definition}

\smallskip

Since the Arthur parameter $\Psi(\varsigma)$ is cuspidal, we know from \cite[Prop.\,7.2]{atobe-gan} that $\varsigma$ appears in the cuspidal spectrum of ${}^\circ\!M$ with multiplicity one, and we let $V_\varsigma$ denote the representation
space of $\varsigma$ inside of the space of cusp forms of ${}^\circ\!M$.  In the above definition, condition (i) pins down every local component $\varsigma_v$ in its local $A$-packet, however, there can be other cuspidal automorphic representations in the same global $A$-packet as $\varsigma$ and satisfying (iv). This is further elaborated upon in Sect.\,\ref{sec:isotypic-component-strongly-inner-cohomology} below. 

  \smallskip 

Define strongly inner cohomology as the space spanned by the isotypic components of strongly-inner Hecke-modules: 
$$ 
H^{\bullet}_{!!}(\mathcal S^{^{\circ}\!M}_{^\circ C_{f}}, \widetilde \cM_{\mu,E}) \ := \ 
\bigoplus_{\sigma_f \in {\rm Coh}_{!!}(^{\circ} \!M, \mu, {}^\circ C_{f})}  H^{\bullet}_{!}(\mathcal S^{^{\circ}\!M}_{^\circ C_{f}}, \widetilde \cM_{\mu,E}) (\sigma_f). 
 $$

 For brevity, denote 
${\rm Coh}_{*}(^{\circ}\!M, \mu/E, {}^{\circ}C_{f})$ by ${\rm Coh}_{*}(^{\circ}\!M, \mu),$ for $* \in \{!, !!\}$, 
with the level structure ${}^{\circ}C_{f}$ and the field of coefficients $E$ being understood.

\smallskip

The definition of strongly inner spectrum (and so also strongly inner cohomology) seemingly needs an embedding $\iota : E \to \C$ to invoke concepts of  
automorphic representation theory, however, we show independence of $\iota$ in the next subsection.

\medskip 
\subsection{Galois equivariance} 
\label{sec:galois-equivariance}

Let $\mu \in X^*({}^\circ T \times E)$ with $\mu_{\rm min} \geq 1,$  
and $\sigma_f \in {\rm Coh}_{!!}(^{\circ}\!M, \mu).$ 
For $\iota : E \to \C$, let $\varsigma$ be the cuspidal automorphic representation of $^\circ\!M(\A) = \rO(2n)(\A_F)$ as in 
Def.\,\ref{def:strongly-inner-spectrum} above. 
The image of an $\iota : E \to \C$ lands in $\bar{\Q}$, and we may further compose this with $\eta \in {\rm Gal}(\bar{\Q}/ \Q).$ We wish to understand what happens when we replace $\iota$ by $\eta \circ \iota$; for this we need to be able to consider the Galois conjugate of a cuspidal representation 
of $\rO(2n)(\A_F)$ satisfying $(i)-(iii)$ of Def.\,\ref{def:strongly-inner-spectrum}. We show in this subsection that there is a cuspidal automorphic representation 
${}^\eta\varsigma$ of $^\circ\!M(\A) = \rO(2n)(\A_F)$ which is globally generic with respect to the same Whittaker datum $\psi,$ and 
$({}^\eta\varsigma)_v = \varsigma_{\eta^{-1}\circ v}$ for $v \in S_{\infty}$, and whose Arthur parameter is: 
$\Psi({}^\eta\varsigma) = {}^\eta(\Psi(\varsigma) \otimes |\ |^{-1/2}) \otimes |\ |^{1/2}.$

\medskip

Since $\varsigma$ is cuspidal and for $v \in S_\infty$ the local representation $\varsigma_v$ being a discrete series representation of $\rO(n,n)(\R)$, it follows that $\varsigma_f$ contributes to cuspidal cohomology, 
i.e., $\varsigma_f$ appears as a Hecke-summand of $H^\bullet_{\rm cusp}(\mathcal S^{^{\circ}\!M}_{{}^\circ C_{f}}, \widetilde \cM_{{}^\iota\mu,\C}).$
(The reader is referred to 
\cite[Sect.\,3.2]{harder-raghuram} for basics on cuspidal and square-integrable cohomology; although the reference is for general linear groups, the discussion in there applies {\it mutatis mutandis} to orthogonal groups.)
But, cuspidal cohomology is contained inside inner cohomology which admits a natural rational structure since it is sheaf-theoretically defined, and so we may consider the effect of a Galois element on the latter. Furthermore, 
inner cohomology is contained in square-integrable cohomology which is captured by the discrete spectrum of $\rO(2n)/F$.  This is summarised in the diagram:
$$
\xymatrix{
\varsigma_f \in H^\bullet_{\rm cusp}(\mathcal S^{^{\circ}\!M}_{{}^\circ C_{f}}, \widetilde \cM_{{}^\iota\mu,\C}) \ \ar@{^{(}->}[r] \ & 
H^\bullet_!(\mathcal S^{^{\circ}\!M}_{{}^\circ C_{f}}, \widetilde \cM_{{}^\iota\mu,\C}) \ar[d]^{\eta^\bullet} 
\\
 H^\bullet_{(2)}(\mathcal S^{^{\circ}\!M}_{{}^\circ C_{f}}, \widetilde \cM_{{}^{\eta\circ\iota}\mu,\C})\  & \ 
H^\bullet_!(\mathcal S^{^{\circ}\!M}_{{}^\circ C_{f}}, \widetilde \cM_{{}^{\eta\circ\iota}\mu,\C}) \ \ar@{_{(}->}[l] \ni {}^\eta\varsigma_f 
}
$$
i.e., the $\eta$-conjugate of $\varsigma_f,$ denoted ${}^\eta\varsigma_f,$ appears in inner-cohomology for the conjugated sheaf, and so also in square-integrable cohomology. Hence, there exists an automorphic representation $\delta$ that contributes to the global discrete spectrum of $\rO(2n)/F$ 
such that 
$$
\delta_f \ = \ {}^\eta\varsigma_f, 
$$
and furthermore, $\delta_\infty$ has nonvanishing relative Lie algebra cohomology with respect to the coefficient system $\M_{{}^{\eta\circ\iota}\mu,\C},$ 
giving us the following equality of infinitesimal characters: 
$$
\chi_{\delta_v} \ = \ \chi_{\M_{({}^{\eta\circ\iota}\mu)_v,\C}}  \ = \ \chi_{\M_{({}^{\iota}\mu)_{\eta^{-1}v,\C}}} = \chi_{\varsigma_{\eta^{-1}v}}.
$$

\smallskip

 Since $\delta$ is in the discrete spectrum, we have its Arthur parameter
$\Psi(\delta).$  We claim that $\Psi(\delta)$ is a cuspidal representation of $\GL_{2n}(\A_F).$
Now, $\Psi(\varsigma)$ is not of cohomological type on $\GL_{2n}(\A_F)$, however, as remarked in Sect.\,\ref{sec:dis-ser-param}, 
after a half-integral Tate-twist it becomes cohomological, and we know from Clozel \cite{clozel} that there is a Galois action on cohomological cuspidal representations of $\GL_{2n}(\A_F),$ i.e., 
we have the cuspidal representation ${}^\eta(\Psi(\varsigma) \otimes |\ |^{-1/2}).$ 
For almost all $v$ (where the local representations are unramified) we have
\begin{multline*}
\Psi( \delta)_v \  = \  \Psi(\delta_v) \ = \ \Psi({}^\eta\varsigma_v) \\ 
= \ \left({}^\eta(\Psi(\varsigma)_v \otimes |\ |^{-1/2})_v \otimes |\ |_v^{1/2}\right) \  = \
\left({}^\eta(\Psi(\varsigma) \otimes |\ |^{-1/2}) \otimes |\ |^{1/2}\right)_v, 
\end{multline*}
where the third equality is an easy calculation with unramified principal series representations of $\GL_{2n}(F_v)$. 
Hence we have two Arthur parameters $\Psi_1 := \Psi( \delta)$ and 
$\Psi_2 := {}^\eta(\Psi(\varsigma) \otimes |\ |^{-1/2}) \otimes |\ |^{1/2}$ such that $\Psi_2$ cuspidal and $\Psi_{1,v} = \Psi_{2,v}$ for almost all $v$. 
Hence $\Psi_1 = \Psi_2,$ and in particular $\Psi_1$ is also cuspidal. Let us  record this as: 
\begin{equation}
\label{eqn:psi1-psi2-1}
\Psi( \delta) \ = \ {}^\eta(\Psi(\varsigma) \otimes |\ |^{-1/2}) \otimes |\ |^{1/2}.
\end{equation}

\smallskip

The Arthur parameter $\Psi( \delta)$ for $\rO(2n)$ is cuspidal, and so a partial symmetric square $L$-function, 
$L^S(s, {\rm Sym}^2, \Psi( \delta)),$ has a pole at $s=1$, and by the backward lifting (see Soudry\,\cite{soudry} and the references therein) 
there exists 
a {\it unique} globally $\psi$-generic cuspidal representation $ \delta'$ of $\rO(2n)$ whose local components transfer almost everywhere to the 
local components of $\Psi( \delta).$ Let $\Psi( \delta')$ be the Arthur parameter of $ \delta'$. Once again we have two Arthur parameters, 
$\Psi( \delta)$ and $\Psi( \delta'),$ 
that agree almost everywhere, and one of them is cuspidal, hence so is the other, and they are equal, i.e., we have 
\begin{equation}
\label{eqn:psi1-psi2-2}
\Psi( \delta) \ = \ \Psi( \delta').
\end{equation}

Fix $v \in S_\infty$. Consider the local component $\delta'_v$. We have
$$ 
\Psi(\delta'_v) = \Psi(\delta')_v = 
\left(^{\eta}\!(\Psi(\varsigma) \otimes |\ |^{-1/2}) \otimes |\ |^{1/2}\right)_{v} = \Psi(\varsigma)_{\eta^{-1}v} = \Psi(\varsigma_{\eta^{-1}v}). 
$$

Hence, $\delta'_v$ and $\varsigma_{\eta^{-1}v} = \mathbb{D}^{\dagger}_{{}^\iota\!\mu_{\eta^{-1}v}}$ are in the same $A$-packet and indeed in the same $L$-packet; 
but $\delta'_v$ is also locally generic with respect to $\psi_v$, hence 
$\delta'_v = \mathbb{D}^{\dagger}_{{}^\iota\!\mu_{\eta^{-1}v}} = \mathbb{D}^{\dagger}_{{}^{\eta\circ\iota}\!\mu_{v}}.$

\smallskip

\begin{definition}
\label{def:galois-conjugate}
Let $\varsigma$ be a cuspidal automorphic representation of $\rO(2n)(\A_F)$ that satisfies $(i)-(iii)$ of Def.\,\ref{def:strongly-inner-spectrum}. 
For $\eta \in {\rm Gal}(\bar{\Q}/ \Q),$ define the $\eta$ conjugate of $\varsigma$, denoted ${}^\eta\varsigma$, 
as the cuspidal representation $\delta'$ as above.  
Then ${}^\eta\varsigma$ satisfies $(i)-(iii)$ of Def.\,\ref{def:strongly-inner-spectrum}, where the highest weight for ${}^\eta\varsigma$ in $(iii)$ is the 
$\eta$-conjugate of 
the highest weight for $\varsigma.$ The Arthur parameters are related by: 
$$
\Psi({}^\eta\varsigma) \ = \ {}^\eta(\Psi(\varsigma) \otimes |\ |^{-1/2}) \otimes |\ |^{1/2}.
$$
\end{definition}

\medskip

Now we can see the independence of $\iota$ in Def.\,\ref{def:strongly-inner-spectrum}, in that if the conditions $(i)-(iv)$ hold for one embedding $\iota : E \to \C$, then they hold for any embedding $\iota : E \to \C.$ Indeed, given some other $\iota' : E \to \C$, there is an $\eta \in {\rm Gal}(\bar{\Q}/ \Q),$ such that $\iota' = \eta \circ \iota.$ 
Given a cuspidal representation $\varsigma$ attached to the data $\{\sigma_f, \iota\}$ then the cuspidal representation ${}^\eta\varsigma$ is attached to the data 
$\{\sigma_f, \iota' = \eta\circ\iota\}.$

\medskip
\subsection{The inner structure of an isotypic component of strongly inner cohomology} 
\label{sec:isotypic-component-strongly-inner-cohomology}

For $\sigma_f \in  {\rm Coh}_{!!}(^\circ M, \mu)$ and $\iota : E \rightarrow \C$, we wish to describe 
$$
H^{\bullet}_{!!}(\mathcal S^{^{\circ}\!M}_{^\circ C_{f}}, \widetilde \cM_{\mu,E}) (\sigma_f)\otimes_{E, \iota} \C.
$$

Let $\cA_2({}^\circ M) = \cA_2\left( {}^{\circ} M(\Q) \backslash {}^{\circ} M(\mathbb A) \right)$ denote the automorphic discrete spectrum of ${}^\circ M;$ 
see \cite[Sec.\,6.7]{atobe-gan} for the definition of $\cA_2({}^\circ M)$. Inside $\cA_2({}^\circ M)$ is the cuspidal spectrum which we denote 
$\cA_{\rm cusp}({}^\circ M) = \cA_{\rm cusp} \left( {}^{\circ} M(\Q) \backslash {}^{\circ} M(\mathbb A) \right).$ 
Consider the representation ${}^{\iota} \sigma : = \varsigma $ in 
Def.\,\ref{def:strongly-inner-spectrum}. There exist finitely many inequivalent cuspidal automorphic representations 
${}^\iota \sigma = {}^\iota \sigma_{1},\ldots, {}^\iota \sigma_{l}$ of ${}^\circ \!M(\A)$ 
all having 
the same Arthur parameter as ${}^\iota \sigma,$ and so appearing with multiplicity one in $\cA_2({}^\circ M)$ (\cite[Prop.\,7.2]{atobe-gan}), indeed, appearing 
with multiplicity one in $\cA_{\rm cusp}$, 
and whose finite-parts are uniquely determined since 
${}^\iota \sigma_{j,f}^{{}^\circ C_{f}} \cong \sigma_f \otimes_{E,\iota} \C,$ however, for the archimedean components, i.e.,  
for $v \in S_\infty$, ${}^\iota \sigma_{j, v}$ can vary in the $A$-packet, or in this case, $L$-packet of $\varsigma_v$ which consists of discrete series 
representations determined by ${}^\iota\mu_v.$ 
For each $1 \leq j \leq l$, we fix 
\begin{equation}\label{eqn:Phi_j}
\Phi_j \in  {\rm Hom}_{(^\circ \mathfrak m(\R), K^\circ_{{}^\circ\!M(\R)}) \times ^\circ M(\mathbb A_f)}  
\left( ^\iota \sigma_j, \ \cA^2_{\rm cusp} \left( {}^{\circ} M(\Q) \backslash ^{\circ} M(\mathbb A) \right) \right),
\end{equation}
which is unique up to scalars by multiplicity-one; denote $\Phi_j({}^\iota \sigma_j) = V_{{}^\iota \sigma_j}.$ 
The base-change to $\C$ via $\iota$ of the strongly-inner 
isotypic component of $\sigma_f$ takes the shape: 
\begin{equation}
\label{eqn:isotypic-comp-after-base-change-1}
H^{\bullet}_{!!}(\mathcal S^{^{\circ}\!M}_{^\circ C_{f}}, \widetilde \cM_{\mu,E}) (\sigma_f)\otimes_{E, \iota} \C 
\ \ \cong \ \ 
\bigoplus_{j=1}^l
H^\bullet({}^\circ \mathfrak m(\R), K^\circ_{{}^\circ\!M(\R)}; \, V_{{}^\iota \sigma_j}^{{}^\circ C_{f}} \otimes \M_{{}^\iota\mu, \C}).
\end{equation}
Using the inverse of $\Phi_j$ in the $j$-th summand, we also have: 
\begin{equation}
\label{eqn:isotypic-comp-after-base-change-2}
H^{\bullet}_{!!}(\mathcal S^{^{\circ}\!M}_{^\circ C_{f}}, \widetilde \cM_{\mu,E}) (\sigma_f)\otimes_{E, \iota} \C 
\ \cong \ 
\bigoplus_{j=1}^l
H^\bullet({}^\circ \mathfrak m(\R), K^\circ_{{}^\circ\!M(\R)}; \, 
{}^\iota \sigma_{j, \infty} \otimes \M_{{}^\iota\mu, \C}) \otimes {}^\iota \sigma_j^{{}^\circ C_{f}}.
\end{equation}

\medskip
Note that for all $v \in S_{\infty}$ and $1 \leq j \leq l$, 
we have $\Psi(^\iota \sigma_{j,v}) = \Psi({\mathbb D}^{\dagger}_{\mu^{v}})$. From \eqref{eqn:isotypic-comp-after-base-change-2} it is clear that an absolutely simple Hecke-module represented 
by $\sigma_f$, after base-change to $\C$ via $\iota$, appears with finite-multiplicity in cuspidal cohomology. 
Later, in the context of the standard intertwining operator, we will simultaneously consider the cuspidal representations
$^{\iota} \sigma_j$ and $^{\kappa_{2n}} { ^{\iota} \sigma_j}$. Since $\kappa_{2n} \in {}^\circ M,$ we have $^{\iota} \sigma_j \cong {}^{\kappa_{2n}} { ^{\iota} \sigma_j}$
and their representation spaces in the space of cusp forms are the same: $\Phi_j({}^{\iota} \sigma_j) = \Phi_j(^{\kappa_{2n}} { ^{\iota} \sigma_j}),$ 
or $V_{{}^{\iota} \sigma_j} = V_{{}^{\kappa_{2n}\iota} \sigma_j}.$ 
Also, for $1 \leq j,j' \leq l$, if $j \neq j'$, then the representations ${}^\iota \sigma_{j}$ and ${}^\iota \sigma_{j'}$ being inequivalent, implies 
that there is some $v \in S_\infty$ such that the local representations ${}^\iota \sigma_{j, v}$ and ${}^\iota \sigma_{j', v}$ are inequivalent. In particular, 
${}^{\kappa_{2n} \iota} \sigma_j$ is not equivalent to ${}^{\iota} \sigma_{j'}.$

\medskip
\section{Critical set for $L$-functions for orthogonal groups}

\medskip
\subsection{$L$-functions for orthogonal groups}
\label{sec:L-fn-orth-group}

Let $\mu \in X^*({}^\circ T \times E)$ be a dominant integral weight for ${}^\circ\!M$ with $\mu_{\rm min} \geq 1,$ and $\sigma_f \in {\rm Coh}_{!!}(^{\circ}M, \mu).$ 
For $\iota : E \to \C$, the cuspidal representation $\varsigma$ of Def.\,\ref{def:strongly-inner-spectrum} will be henceforth denoted 
${}^\iota\sigma.$ Let $\chi^{\circ}$ be a finite order Hecke character of $F$ taking values in $E,$ and $d$ be an integer.  
Let $\chi$ stand for the algebraic Hecke character of $F$ with values in $E$ which gives a continuous homomorphism 
${}^\iota\chi : F^\times\backslash \A_F^\times \to \C^\times$ of the form ${}^\iota\chi = {}^\iota\chi^\circ \otimes |\ |^{-d}.$  The purpose of 
this section is to understand the critical set for the degree-$2n$ Langlands $L$-function $L(s, {}^\iota\chi \times {}^\iota\sigma).$

\smallskip

For a place $v$ of $F$, the local representation ${}^\iota\sigma_v$ is a representation of $\rO(2n)(F_v)$, and its 
restriction to $\SO(2n)(F_v)$ can be irreducible or a sum of two irreducible constituents that are mutually $\kappa_{2n}$-conjugates
since $\kappa_{2n} \in \rO(2n)(F_v) \setminus \SO(2n)(F_v)$. From Theorems 3.5, 3.6, 3.9, and 3.10 of Atobe--Gan \cite{atobe-gan}, 
we get an $L$-parameter for $\rO(2n)(F_v)$. Then we consider the degree-$2n$ local Langlands--Shahidi $L$-function 
$L_v(s, {}^\iota\chi_v \times {}^\iota\sigma_v)$ attached to the data ${}^\iota\chi_v$ and this L-parameter  (corresponding to case ($D_{n,i}$) of  
\cite[App.\,A]{shahidi-book}). The global $L$-function is defined as an Euler product: 
$L(s, {}^\iota\chi \times {}^\iota\sigma) = \prod_v L_v(s, {}^\iota\chi_v \times {}^\iota\sigma_v).$ The analytic properties of this $L$-function is 
well-understood by the Langlands--Shahidi machinery for which we refer the reader to Shahidi's book \cite{shahidi-book}.

\medskip
\subsection{Critical points for Rankin--Selberg L-functions for orthogonal groups} 
\label{sec:critical-set}

The infinite part of $L(s, {}^\iota\chi \times {}^\iota\sigma)$ is a product over archimedean places: 
\[ 
L_\infty(s, {}^\iota\chi \times {}^\iota\sigma) \ := \ \prod \limits_{v \in S_\infty} L(s, {}^\iota\chi_{v} \times {}^\iota\sigma_v).
\]
An integer $m$ is said to be a {\it critical} for $L(s, {}^\iota\chi \times {}^\iota\sigma)$ if both 
$$
L_\infty(s, {}^\iota\chi \times {}^\iota\sigma) \ \ {\rm and}\ \ L_\infty(1-s, {}^\iota\chi^{-1} \times {}^\iota\sigma^{\sf v})
$$ 
are holomorphic at $s=m$, where ${}^\iota\sigma^{\sf v}$ is the contragredient of ${}^\iota\sigma.$ The purpose of this subsection is to determine all such critical integers. 

\smallskip

\smallskip

\begin{proposition} 
\label{prop:critical-points}  
Let the notations be as in the first paragraph of Sect.\,\ref{sec:L-fn-orth-group}. 
The critical set for the degree-$2n$ $L$-function $L(s, {}^\iota\chi \times {}^\iota\sigma)$ is the contiguous set of $2\mu_{\min}$ integers:  
$$
\left\{  1+d-\mu_{\rm min}, ~ 2+d-\mu_{\rm min}, \ \dots \ , \ d+\mu_{\rm min}\right\}.
$$ 
Since $\mu_{\rm min} \geq 1$, there are at least two successive critical integers. The critical set is independent of $\iota$. 
\end{proposition}

\begin{proof}  
For $v \in S_\infty$, the local representation ${}^\iota\sigma_v$ is the discrete series representation $\D_{{}^\iota\!\mu^{v}}$ of $\rO(n,n)(\R).$ 
As before, identifying $S_\infty$ with $\Hom(F,\C)$, we have ${}^\iota\!\mu^{v} = \mu^{\iota^{-1}v}.$ For given $v$ and $\iota,$ let's denote 
$\iota^{-1}v$ by $\tau.$ Recall that $\mu^\tau_1 \geq \cdots \geq \mu^\tau_{n-1} \geq |\mu^\tau_n| \geq 1.$
We know the Langlands parameter of ${}^\iota\sigma_v$ from \eqref{eqn:L-paramater-D_mu}. Also, the local 
component at $v$ of the character $\chi$ may be expressed as: $\chi_v =  |\ |^{-d} {\rm sgn}^{\epsilon_v}$ with $\epsilon_v \in \{0,1\}.$ 
Recall that the $2$-dimensional irreducible representation $I(\ell)$ of $W_\R$ is invariant under twisting by the sign character, i.e., 
$I(\ell) \otimes {\rm sgn} = I(\ell).$ Recall also (\cite{knapp}) that the local $L$-factor attached to $I(\ell)$ is $2(2\bpi)^{-(s+|\ell|/2)}\Gamma(s+|\ell|/2).$ Putting these together, 
and recalling $\ell^\tau_j$ in terms of $\mu^\tau_j$ from \eqref{eqn:ell-and-mu}, 
we get the local $L$-factor at $v$  as: 
\begin{multline}
\label{eqn:L-fact-at-v}
L(s, {}^\iota\chi_{v} \times {}^\iota\sigma_v) \ = \ 
L(s-d, {}^\iota\chi^\circ_{v} \times {}^\iota\sigma_v) \ = \\
2^{n}  \prod \limits_{j=1}^{n}  (2 \bpi)^{-(s-d+ |\mu_j^\tau| + n - j)} ~\Gamma(s-d+ |\mu_j^\tau| + n - j).   
\end{multline}
To compute the critical integers, powers of $2$ and exponential terms -- since they are holomorphic and nonvanishing everywhere -- are not relevant. Ignoring such terms, we write 
$L(s, {}^\iota\chi_{v} \times {}^\iota\sigma_v) \approx \prod_{j=1}^{n}  \Gamma(s-d+ |\mu_j^\tau| + n - j).$ Taking the product over $v \in S_\infty$, which on the 
right hand side will be taking the product over all $\tau : F \to E$, we get
$$
L_\infty(s, {}^\iota\chi \times {}^\iota\sigma) \ \approx \ \prod_{\tau : F \to E} \, \prod_{j=1}^{n} \, \Gamma(s-d+ |\mu_j^\tau| + n - j).
$$
Similarly, on the dual side we get: 
$$
L_\infty(1-s, {}^\iota\chi^{-1} \times {}^\iota\sigma^{\sf v}) \ \approx \ \prod_{\tau : F \to E} \, \prod_{j=1}^{n} \, \Gamma(1-s+d+ |\mu_j^\tau| + n - j). 
$$
We see that $m$ is critical if and only if for all $\tau : F \to E$ and all $1 \leq j \leq n$ both the conditions are satisfied: 
\begin{itemize} 
\item $m-d+ | \mu_j^\tau|+ n - j \geq 1$;
\item $ 1 - m + d +  |\mu_j^\tau| + n - j \geq 1$.
\end{itemize} 
Since $
\mu_1^\tau +n-1 \ > \ \mu_2^\tau +n- 2 \ > \ \cdots \ > \ |\mu_n^\tau| \ \geq \ \mu_{\rm min},$
we see that the conditions are equivalent to
$1 - \mu_{\rm min} \ \leq \ m-d \ \leq \ \mu_{\rm min}.$ 
\end{proof}

\medskip
We draw the following two easy corollaries of Prop.\,\ref{prop:critical-points}: 

\medskip
\begin{corollary}
\label{cor:-n-and-1-n}
The integers $-n$ and $1-n$ are critical for $L(s, {}^\iota\chi \times {}^\iota\sigma)$ if and only if 
 \begin{equation*}
1 - \mu_{\rm min} \ \leq \ -(d+n) \ \leq \ \mu_{\rm min} - 1.
\end{equation*}
\end{corollary}

\medskip
\begin{corollary}
\label{cor:arch-ratio-L-values}
If the inequalities in Cor.\,\ref{cor:-n-and-1-n} hold then 
 $$
\frac{L_\infty(-n, {}^\iota\chi \times {}^\iota\sigma)}{L_\infty(1-n, {}^\iota\chi \times {}^\iota\sigma)}
 \ \approx_{\Q^{\times}} \  \bpi^{n \sf r_F}, 
$$
where $\approx_{\Q^{\times}}$ means equality up to a nonzero rational number. 
\end{corollary}

\begin{proof}
Use the functional equation $\Gamma(s+1)  = s \Gamma(s)$ and \eqref{eqn:L-fact-at-v}. 
\end{proof}

\medskip
\section{Kostant representatives and a combinatorial lemma}
\label{sec:kostant-comb-lem}

As in \cite{harder-raghuram}, we have a philosophically meaningful combinatorial lemma which says that 
the combinatorial condition in Cor.\,\ref{cor:-n-and-1-n} that guarantees the successive integers $-n$ and $1-n$ to be critical for $L(s, {}^\iota\chi \times {}^\iota\sigma)$ is equivalent 
to the existence of some very special elements in the Weyl group of $G = \Res_{F/\Q}(\rO(2n+2)/ F).$ This lemma will ultimately allow us to consider 
a particular induced representation in the cohomology of the Borel--Serre boundary of a locally symmetric space for $G$, thus permitting us to 
invoke the machinery of Eisenstein cohomology. Interestingly, the proof of the combinatorial lemma, which was challenging for general linear groups, is a relatively easy in the situation of our orthogonal group $G$ and the particular parabolic subgroup $P$ we are interested in.

\medskip
\subsection{Weyl group}
\label{sec:weyl-group-new}

Recall our notation that $n = 2r$, and $G_0 = \rO(2n+2)/ F$. The Weyl group $W_{G_0}$ is isomorphic to the subgroup of the group 
$\mathfrak{S}_{n+1} \rtimes (\Z/2\Z)^{n+1}$ of signed permutations consisting of all signed permutations of $n+1$ symbols with an even number of sign changes.  An element $w \in W_{G_0}$ permutes the $e_j$'s according to the corresponding signed permutation. 

\smallskip
 
Recall the maximal parabolic subgroup $P_0$ of $G_0$ from Sect.\,\ref{sec:root-system} corresponding to deleting the simple root $\alpha_0.$ 
Let $w_{P_{0}} $ be the unique element of $W_{G_0}$ such that $w_{{P_{0}}}( \Pi_{G_{0}} \setminus \{ \alpha_0\}) \subset  \Pi_{G_{0}}$ and $w_{P_0}( \alpha_{0}) < 0$. Observe that ${P_{0}}$ is self associate and we have 
$$
w_{{P_{0}}}(\Pi_{G_{0}} \setminus \{ \alpha_0 \}) =  \Pi_{G_{0}} \setminus \{ \alpha_0\} ~ \text{and}~w_{P_{0}}(\Delta_{U{_{0}}}) = - \Delta_{U_{0}},
$$ 
where $\Delta_{U_{0}}$ is the set of positive roots with root spaces in $\mathfrak{u}_{P_0} = {\rm Lie}(U_{P_0}).$ The element $w_{P_0}$ is the signed permutation: 
$$
e_0 \mapsto -e_0, \quad e_i \mapsto e_i \ \ {\rm for} \ \ 1 \leq i \leq n-1, \quad e_n \mapsto -e_n. 
$$
The element $w_{P_0}$ can be factorised as a product in terms of the simple reflections $\ss_{\alpha_{j}}$ as:
 \begin{equation}
 \label{eqn:factor-w-P}
w_{P _{0}} \ = \ 
\ss_{\alpha_{0}} ~\ss_{\alpha_{1}} ~\cdots ~
\ss_{\alpha_{n-3}} ~ \ss_{\alpha_{n-2}} ~\ss_{\alpha_{n}} ~\ss_{\alpha_{n-1}} \ss_{\alpha_{n-2}}  ~\cdots~
\ss_{\alpha_{1}} ~\ss_{\alpha_{0}}. 
\end{equation}
The length of $w_{P _{0}}$ is $l(w_{P _{0}}) =  2n = \dim(U_{P_0}).$

\medskip
   
The Weyl group of $G \times E$ factors as a product 
$$
W_{G \times E} \ = \ \prod \limits_{\tau: F \rightarrow E} \, W_{G_0 \times_{F, \tau} E}, 
$$ 
where each $W_{G_0 \times_{F, \tau} E}$ is isomorphic to  $W_{G_{0}}$. Sometimes, we abbreviate $W_{G_0 \times_{F, \tau} E}$ as $W^\tau.$ 
The element $w_P \in W_{G \times E}$ is $(w_{P_0}^\tau)_{\tau : F \to E}$ where each $w_{P_0}^\tau$ is the copy of $w_{P_0}$ in 
$W_{G_0 \times_{F, \tau} E}.$ The length of $w_{P}$ is $l(w_{P}) =  2n \sr_F = \dim(U_P).$ 

\medskip

The {\it twisted} action of $w \in W_{G \times E}$ on $\lambda \in X^{\ast}(T \times E)$ is given by 
$w \cdot \lambda: = (w^{\tau} \cdot \lambda^\tau)_{\tau : F \rightarrow E},$ 
where
$$
w^{\tau} \cdot \lambda^{\tau} \ := \ 
w^{\tau} (\lambda^{\tau} + \brho_{G^{\tau}_{0}}) - \brho_{G^{\tau}_0}.
$$

\medskip
\subsection{\bf Kostant representatives}
\label{sec:kostant-reps}

Let $\Pi_{M_{P_{0}}}$ denote the set of simple roots of the Levi quotient $M_{P_0}$ of $P_0.$ 
The set of Kostant representatives in $W_{G_0}$ corresponding to the parabolic subgroup $P_0$ is defined as: 
$$
W^{P_{0}}  \ := \ 
\{ w \in W_{G_0} \ : \ 
w^{-1} \alpha > 0 \ \ \forall \alpha \in \Pi_{M_{P_{0}}} \} .
$$
Now define the set of Kostant representatives in $W_{G \times E}$ corresponding to the parabolic subgroup $P$ as: 
$$
W^{P} \ = \ \left\{ w = (w^{\tau}) \ : \ w^{\tau} \in {W^{\tau}}^{P_{0}^{\tau}} \right\}.
$$

\medskip
If $w_0 \in W^{P_{0}}$ then for its length we have $0 \leq \ell(w_0) \leq \dim(U_{P_0})$, and hence for $w \in W^P$ we have 
$0 \leq \ell(w) \leq \dim(U_{P}).$ Recall that $\dim(U_{P_0}) = 2n$ is even.

\begin{definition}
\label{def:balanced-elts}
We say $w_0 \in W^{P_{0}}$ is balanced if $\ell(w_0) = \tfrac{1}{2} \dim(U_{P_0}),$ and 
we say $w = (w^{\tau}) \in  W^P$ is balanced if each $w^\tau$ is balanced, i.e., 
$\ell(w^{\tau}) = \tfrac{1}{2} \dim(U_{P_{0}^{\tau}})$ for each $\tau : F \to E.$ 
\end{definition}

\medskip
In the self-associate case of our parabolic subgroup $P$, a couple of self-bijections of $W^P$ with complementary lengths will be important. We record these 
in the following two propositions. (See Lem.\,5.6 and Lem.\,5.7 of \cite{harder-raghuram}, the proofs of which are totally general.)

\begin{proposition}
\label{prop:w'}
The map $w \mapsto w': = w_P\,w$ defines a bijection $W^P \rightarrow W^P$ such that $\ell(w) + \ell(w') = {\rm dim}(U_{P}).$
Hence, $w$ is balanced if and only if $w'$ is balanced.
\end{proposition}

Let $w_{G_0}$ be the element of $W_{G_0}$ with longest length. As a signed permutation it is 
$$
e_i \mapsto -e_i \ \ {\rm for} \ \ 0 \leq i \leq n-1, \ \ {\rm and} \quad e_n \mapsto e_n.
$$
(Since $n$ is even, we have an even number of sign changes.) The element $w_G \in W_{G \times E}$ is $(w_{G_0}^\tau)_{\tau : F \to E}$ where each $w_{G_0}^\tau$ is the copy of $w_{G_0}$ in 
$W_{G_0 \times_{F, \tau} E}.$ Similarly, let $w_{M_0}$ be the element of $W_{M_{P_0}}$ of longest length which is given by
$$
e_i \mapsto -e_i \quad {\rm for} \ 1 \leq i \leq n.
$$
We have $w_{M} = (w_{M_0}^\tau)_{\tau : F \to E},$  where $w_{M_0}^\tau$ is the copy of $w_{M_0}$ in 
$W_{G_0 \times_{F, \tau} E}.$

\begin{proposition}
\label{prop:w-vee}
The map $W^P \rightarrow W^P$ defined by $w \mapsto w^\sv: = w_{M} \, w \, w_G$ is a bijection such that 
$\ell(w) + \ell(w^\sv) = \dim(U_P).$ Hence, $w$ is balanced if and only if $w^\sv$ is balanced. 
\end{proposition}

\medskip

In the set of Kostant representatives $W^{P_0}$ one can check that there are only two balanced elements $w^+$ and $w^{-}$ which are described below:
\smallskip
\begin{equation*}
\begin{split}
&w^+: \quad e_i \mapsto e_{i+1} \quad {\rm for} \ 0 \leq i \leq n-1, \quad e_n \mapsto e_0; \\ 
&w^-: \quad e_i \mapsto e_{i+1} \quad {\rm for} \ 0 \leq i \leq n-2, \quad e_{n-1} \mapsto -e_n, \quad e_n \mapsto -e_0. 
\end{split}
\end{equation*}
It follows using the definitions of the self-bijections in the propositions above that: 
\smallskip
\begin{equation*}
\begin{split}
& w^+{}' = w^-, \quad w^-{}' = w^+ ; \\
& w^+{}^{\sv} = w^+, \quad w^-{}^{\sv} = w^-. 
\end{split}
\end{equation*}

\medskip
\subsection{A Combinatorial lemma} 
\label{sec:combinatorial-lemma}

Let the notations be as in the first paragraph of Sect.\,\ref{sec:L-fn-orth-group}. The weight $\mu = (\mu^\tau)_{\tau : F \to E} \in X^*_+({}^\circ\!T \times E)$ and the integer $d$ give a weight $de_0 + \mu \in X^*(T\times E)$ whose $\tau$-component is $de_0 + \mu^\tau.$ Note that $de_0 + \mu$ need not 
be dominant as a weight for $G.$

\begin{lemma}[Combinatorial lemma] 
\label{lem:comb-lemma}
With notations as above, the following are equivalent: 

\smallskip
\begin{enumerate}
\item[(i)] $-n$ and $1-n$ are critical points for the L-function $L(s, {}^\iota\chi \times {}^\iota\sigma)$;

\smallskip
\item[(ii)] $1 - \mu_{{\rm min}} \ \leq \ -(n+d) \ \leq \ \mu_{{\rm min}} - 1;$

\smallskip
\item[(iii)] There exists a balanced Kostant representative $w \in W^P$ such that $w^{-1} \cdot (de_0 + \mu)$ is a dominant integral weight for $G.$ 
\end{enumerate}
\end{lemma}

\begin{proof} The equivalence $(i) \iff (ii)$ is exactly Cor.\,\ref{cor:-n-and-1-n}. It suffices to prove $(ii) \iff (iii)$. 

Suppose $(ii)$ holds. A calculation show that
\begin{align*}
\begin{split}
(w^{+})^{-1} \cdot (de_0 + \mu^{\tau}) = (\mu_1^{\tau} -1)~ e_0 + \ldots + (\mu_n ^{\tau}- 1)~ e_{n-1} +  (d+n)~ e_n, \\
(w^{-})^{-1} \cdot (de_0 + \mu^{\tau}) = (\mu^{\tau}_1 -1)~ e_0 + \ldots + (-\mu^{\tau}_n - 1)~ e_{n-1}  - (d+n)~ e_n. 
\end{split}
\end{align*}
For $\tau : F \to E$, if $\mu^{\tau}_n \geq 1$ then take $w^{\tau} = w^+$ and if $\mu^\tau_n \leq -1$ then take $w^{\tau} = w^-$. 
Now put $w = (w^\tau)_{\tau : F \to E}$ which is a balanced element of $W^P.$ For each $\tau$, 
observe that $(w^\tau)^{-1} \cdot (de_0 + \mu^\tau)$ is dominant; indeed, if $\mu^{\tau}_n \geq 1$ then $(ii)$ implies that 
$\mu_n ^{\tau}- 1 \geq |d+n|$ rendering $(w^{+})^{-1} \cdot (de_0 + \mu^{\tau})$ dominant, 
and if $\mu^{\tau}_n \leq -1$ then $(ii)$ implies that $-\mu_n ^{\tau}- 1 \geq |d+n|$ rendering $(w^{-})^{-1} \cdot (de_0 + \mu^{\tau})$ dominant. 
Hence $(iii)$. 

Conversely, if $(iii)$ holds for some balanced element $w$ of $W^P,$ then for a given $\tau,$ the above formulae tell us that the sign of $\mu_n^\tau$ determines 
if $w^\tau$ is $w^+$ or $w^-,$ and furthermore, that $1- |\mu^{\tau}_n|  \leq -(n+d) \leq  |\mu^{\tau}_n|-1.$ But this is true for all $\tau,$ 
hence $(ii)$.
\end{proof}

\medskip
\subsection{Consequences of the combinatorial lemma}
\label{sec:w-on-other-weight}

Under the notations and conditions imposed by the Combinatorial Lemma (Lem.\,\ref{lem:comb-lemma}), let's record some consequences on the effect of 
Kostant representatives related to $w,$ such as $w'$, $w^{\sv}$ and $w^{\sv}{}'$ (see Prop.\,\ref{prop:w'} and Prop.\,\ref{prop:w-vee}), 
on some weights related to $\lambda: = w^{-1} \cdot (de_0 + \mu).$ 
For $\mu \in X^*({}^\circ\!T \times E)$ define $\mu^\sv := -w_{M_{P}} \mu.$ Observe that $\mu^\sv = \mu,$ reflecting the fact that the representation $\M_{\mu,E}$ is self-dual. Similarly, for $\lambda \in X^*(T \times E)$ define $\lambda^\sv : = -w_{G} \lambda.$ 
When $\lambda$ is dominant-integral then $\lambda^\sv$ is the the highest weight of the dual of $\M_{\lambda,E}$. We omit the proof of the following proposition which is a straightforward calculation. The information in this proposition will be relevant when we deal with Poincar\'e duality for the cohomology of the Borel--Serre boundary.

\medskip
\begin{proposition}
\label{prop:w-w'-etc} 
Under the notations and conditions of Lem.\,\ref{lem:comb-lemma}, for 
$\lambda: = w^{-1} \cdot (de_0 + \mu),$ we have
\smallskip
\begin{enumerate}
\item $w' \cdot \lambda \ = \ (-d -2n) e_0 + ~^{\kappa_{2n}} \mu$. 

\smallskip
\item  $ w^\sv \cdot \lambda^\sv  \ = \ (-d-2n) e_0 + \mu^\sv \ = \ (-d-2n) e_0 + \mu$. 

\smallskip
\item $w^\sv{}' \cdot \lambda^\sv \ = \ d e_0 + {}^{\kappa_{2n}}\!\mu^\sv \ =  \ d e_0 + {}^{\kappa_{2n}}\!\mu$.
\end{enumerate}
\end{proposition}


\medskip
\section{Arithmetic of intertwining operators}
\label{sec:arith-int-op}

We continue with our global notations as in the first paragraph of Sect.\,\ref{sec:L-fn-orth-group}. Since the embedding $\iota : E \to \C$ will be fixed throughout 
this section, we will allow ourselves, only for this section, an abuse of notation that will not cause any confusion: denote the 
cuspidal automorphic representation ${}^\iota\sigma$ of ${}^\circ\!M(\A) = {}^\circ\!M_0(\A_F) = \rO(2n)(\A_F)$ simply by $\sigma$, and similarly, 
denote the Hecke character ${}^\iota\chi$ of $F$ simply as $\chi$ which is of the form $ \chi^\circ \otimes |\ |^{-d}$ for a finite-order
character $\chi^\circ.$

\medskip
\subsection{Induced representations}
\label{sec:ind-rep}

We begin by recalling the induced representations as set-up in the Langlands--Shahidi machinery \cite{shahidi-book}. 
Recall that the Levi-quotient $M_0 = M_{P_0}$ of the maximal parabolic subgroup $P_0$ of $G_0$ corresponding to deleting the simple 
root $\alpha_0 = e_0 - e_1$ is explicitly given by: 
$$ 
M_0 = \left\{  m_{x,g} = \left(\begin{matrix} x & & \\ & g & \\ & & x^{-1}\end{matrix}\right) \ : \ x \in A_0 = \GL_1/F, \ \ g \in {}^{\circ}\!M_0 = \rO(2n)/F \right\}. 
$$ 
Of course, $M_0 = A_0 \times {}^\circ\!M_0.$ 
We have the dual of the real Lie algebra $\fa^* = X^*(M_0) \otimes \R = X^*(A_0) \otimes \R = \R \, e_0.$ The set of roots whose root
spaces appear in $U_{P_0}$ is $\{e_0 \pm e_i : 1 \leq i \leq n\}$, and half their sum is 
$\brho_{P_0} = n\, e_0.$ We have 
$\langle \brho_{P_0}, \alpha_0 \rangle = 2 (\brho_{P_0}, \alpha_0)/(\alpha_0, \alpha_0) = n.$ The fundamental weight $\tilde{\alpha}$ corresponding 
to $\alpha_0$ is 
$$
\tilde{\alpha} \ = \ \langle \brho_{P_0}, \alpha_0 \rangle^{-1} \brho_{P_0} \ = \ e_0.
$$
For any place $v$ of $F$, the modular character $\delta_{P_0}$ at $v$ given by the adjoint action of $M_{P_0}(F_v)$ on $U_{P_0}(F_v)$ is 
$\delta_{P_0, v}(m_{x,g}) =  |x|_v^{2n}.$ At an ad\`elic level, it has the same form: $\delta_{P_0}(m_{x,g}) =  |x|^{2n}.$ Hence, 
$\delta_{P_0}^{1/2} = |\brho_{P_0}|.$ For $s \in \C$, in 
the Langlands--Shahidi method we start with the induced representation $I(s\tilde{\alpha}, \chi \otimes \sigma),$ which we will denote simply as
$I(s, \chi \times \sigma),$ consisting of all smooth functions 
$f: G_0(\A_F) \longrightarrow  V_{\sigma}$ such that 
$$
f(m_{x,g}uh ) \ = \ |x|^{n+s} \, \chi(x) \, \sigma(g) \, f(h), 
$$
for all $m_{x,g} \in M_{P_0}(\A_F), \ u \in U_{P_0}(\A_F), \ h \in G_0(\A_F).$
In terms of normalised parabolically induced representations we have
$$
I(s, \chi \times \sigma) \ = \ \Ind_{P(\A)}^{G(\A)}  \left(\chi[s] \times \sigma \right), 
$$
where $\chi[s] = \chi \otimes |\ |^s.$ 
In terms of algebraically (un-normalised) parabolically induced representations we have
$$ 
I(s, \chi \times \sigma) \ = \ \aInd_{P(\A)}^{G(\A)}  \left(\chi[s+n] \times \sigma \right).
$$ 
In particular, observe that 
$$
I(s, \chi \times \sigma)|_{s = -n} \ = \ \aInd_{P(\A)}^{G(\A)}  \left(\chi \times \sigma \right).
$$ 
We designate 
$s = s_0 = -n$ as our {\it point of evaluation}. The algebraic induction works well for cohomological purposes, whereas the 
analytic theory of Langlands--Shahidi $L$-functions works well with normalised induction. In the arithmetic theory of $L$-functions, 
it's a necessary evil to go back and forth between these two forms of induction.

\medskip
\subsection{The standard intertwining operator}
\label{sec:T-st}
The weyl group element $w_{P_{0}}$ defined in  Sect.\,\ref{sec:weyl-group-new} is realised by conjugation by the matrix in $\rO(2n+2)$: 
$$
\tilde{w}_{P_0} \ = \ 
\begin{pmatrix} 
& & & & 1 \\
& I_{n-1} & & & \\
& & J_2 & & \\
& & & I_{n-1} & \\
1 & & & & 
\end{pmatrix}.
$$
We will denote $\tilde{w}_{P_0} $ simply as $w_{P_0}$ itself. The conjugation action of $w_{P_0}$ on the inducing representation is given by
$w_{P_0}(\chi[s] \times \sigma) = \chi^{-1}[-s] \times {}^{\kappa_{2n}} \sigma,$
where the element $\kappa_{2n}$ was defined in \eqref{eq:kappa-2n}. The standard global intertwining operator 
$$
T_{\rm st}(s, w_{P_0}, \chi \times \sigma) \ : \ I(s, \chi \times \sigma) \ \longrightarrow \ I(-s, \chi^{-1} \times {}^{\kappa_{2n}} \sigma) 
$$ 
is defined by the integral:
$$
T_{\rm st}(s, w_{P_0}, \chi \times \sigma)(f)(g) \ = \ \int \limits_{U_P(\A)} f(w_{P_0}^{-1} \, u\, g) \, du. 
$$ 
We will denote $T_{\rm st}(s, w_{P_0}, \chi \times \sigma)$ simply by $T_{\rm st}(s, w_{P_0})$ or even just $T_{\rm st}(s).$ The intertwining operator $T_{\rm st}(s)$ is a product of local intertwining operators 
$$
T_{\rm st}(s) \ = \ \otimes'_v \, T_{{\rm st}, v}(s), 
$$ 
where the local operators are defined by analogous local integrals for each place $v $ of $F$. The global measure $du$ on $U_P(\A) = U_{P_0}(\A_F)$ 
is a product over the coordinates of $U_{P_0}$ of the additive measure $dx$ on $\A_F$, which in turn is a product $\prod_v dx_v$ of local additive measures $dx_v$ on $F_v^+$; for a finite place $v$ we normalise $dx_v$ by ${\rm vol}(\O_v) = 1$, where $\O_v$ is the ring of integers of $F_v$, 
and for an archimedean $v$ we take $dx_v$ as the Lebesgue measure on $\R.$

\medskip
\subsection{Local intertwining operators at a non-archimedean place: unramified case}
\label{sec:p-adic-unramified}
Since we are dealing with an algebraically disconnected group $\rO(2n)$ and not 
the connected algebraic group $\SO(2n)$ some brief comments are in order on local unramified representations. 
The reader is referred to \cite[Sect.\,2.3]{atobe-gan} for further details. 
For a finite place $v$ of $F$, let $\O_v$ be the ring of integers of $F_v$, 
$\varpi_v$ a uniformizer, and $q_v$ the cardinality of the residue field $\O_v/\varpi_v\O_v$. 
Assume that $v$ is such that $\chi_v$ and $\sigma_v$ are unramified. 
The orthogonal space we consider is $F^{2n}$ with a basis 
$\{w_1,\dots,w_n, w_n^*,\dots,w_1^*\}$ and the non-degenerate symmetric bilinear form on $F^{2n}$ is represented by the matrix $J_{2n}$ with respect to 
this basis. Consider the $\O_v$-lattice $L_v$ in $F^{2n} \otimes F_v = F_v^{2n}$ spanned by this basis: 
$$
L_v \ = \ \O_v w_1 \oplus \cdots \oplus \O_v w_n \oplus \O_v w_n^* \oplus \cdots \oplus \O_v w_1^*. 
$$ 
Let $K_v$ be the subgroup of $\rO(2n)(F_v)$ that stabilizes the lattice $L_v$. Then $K_v$ is a maximal compact subgroup of $\rO(2n)(F_v)$.
We have $K_v = K_v^1 \ltimes \langle \kappa_{2n} \rangle$ where 
$K_v^1 = K_v \cap \SO(2n)(F_v)$ is a maximal compact subgroup of $\SO(2n)(F_v)$. We say an irreducible admissible representation 
$\pi$ (resp., $\pi_1$) of $\rO(2n)(F_v)$ (resp., $\SO(2n)(F_v)$) is unramified if $\pi$ (resp., $\pi_1$) has a nonzero vector fixed by $K_v$ 
(resp., $K_v^1$). It is known in this case that $\dim(\pi^{K_v}) = \dim(\pi_1^{K_v^1}) = 1.$ The nonzero vector fixed by $K_v$ (or $K_v^1$) will be called
a spherical vector. 
(The reader specifically interested in comparing our notations with \cite{atobe-gan} will note
that our $2n$-dimensional orthogonal space is the $V_{2n}$ on p.\,358 of {\it loc.\,cit.}\ with $c = d = 1.$)  

\medskip

The relation between the local intertwining operator at a finite unramified place and local $L$-functions is a famous calculation of Langlands that generalises
the classical Gindikin--Karpelevich formula. This is a well-known part of the Langlands--Shahidi machinery; we will briefly review the results as applied to our situation, and refer the reader to Shahidi's book \cite{shahidi-book} for more details. Let $f_v^\circ$ denote the normalised spherical vector in the local induced representation 
$I_v(s, \chi_v \times \sigma_v)$--normalised by taking the value $1$ on the identity element. Similarly, let $\tilde{f}_v^\circ$ denote the normalised spherical vector in $I_v(-s, \chi_v^{-1} \times {}^{\kappa_{2n}}\sigma_v)$. Then: 
\begin{equation}
\label{eqn:L-G-K}
T_{{\rm st}, v}(s)(f_v^\circ) \ \ = \ \ 
\frac{L_v(s, \chi_v \times \sigma_v)}{L_v(1+s, \chi_v \times \sigma_v)} \ 
\tilde{f}_v^\circ.
\end{equation}
Let us  recall the local $L$-factor $L_v(s, \chi_v \times \sigma_v)$. Suppose the Satake parameter of $\sigma_v$ is 
${\rm diag}(\vartheta_{v,1},\dots,\vartheta_{v,n}, \vartheta_{v,n}^{-1}, \dots, \vartheta_{v,1}^{-1}) \in \rO(2n)(\C)$, 
and $\vartheta_v = \chi_v(\varpi_v)$, then we have
$$
L_v(s, \chi_v \times \sigma_v) \ = \ 
\prod_{j=1}^n \left((1- \vartheta_v \, \vartheta_{v,j}\, q_v^{-s}) (1- \vartheta_v\, \vartheta_{v,j}^{-1}\, q_v^{-s}) \right)^{-1}. 
$$
A subtle point to appreciate is the manner in which the complex variable $s$ is inserted in the induced representation, via the $s\tilde{\alpha}$, which has a 
bearing on the precise form of \eqref{eqn:L-G-K}.

\medskip
\subsection{Langlands's constant term theorem} 
\label{sec:Langlands}
So far in our discussion of intertwining operators we did not talk about convergence issues. As is well-known, convergence is guaranteed in some half-plane, and one has a meromorphic continuation to all $s \in \C.$ The holomorphy of the intertwining operator, Eisenstein series, and holomorphy of global (partial) $L$-functions are all related by Langlands's constant term theorem that we now recall. 
A section $f \in  I(s, \chi \times \sigma)$ can be identified with the function $\underline{g} \mapsto f(\underline{g})(\underline{1})$ on $G(\A)$, 
letting us embed $I(s, \chi \times \sigma) \xhookrightarrow{}  \mathcal C^{\infty} (P(\Q) \backslash G(\A)).$
The Eisenstein series associated to $f$ is defined via Eisenstein summation for $\Re(s) \gg 0$:  
$$
{\rm Eis}_{P} (s,f)(\underline{g}) \ = \ 
\sum \limits_{\gamma \in P(\Q) \backslash G(\Q)} f(\gamma \, \underline{g}), 
$$
giving us a map 
$
{\rm Eis}_{P}(s,-) : \mathcal{C}^{\infty}(P(\Q) \backslash G(\A))  \rightarrow \mathcal \mathcal{C}^{\infty}(G(\Q) \backslash G(\A)).
$ 
Next, we can take the constant term along $P$ for $\phi \in \mathcal{C}^{\infty}(G(\Q) \backslash G(\A))$ defined by:
$$ 
\mathcal F_P(\phi)(\underline{g}) = \int \limits_{U_P(\Q) \backslash U_P(\A)} \phi(\underline u \, \underline g) \, d\underline u.
$$
giving us a map 
$
\mathcal F_{P} : \mathcal{C}^{\infty}(G(\Q) \backslash G(\A)) \to \mathcal{C}^{\infty}(P(\Q) \backslash G(\A)).
$ 
Langlands's theorem gives a relation between the intertwining operator $T_{\rm st}(s)$, the Eisenstein map 
${\rm Eis}_{P}(s,-),$ and the constant term map $\mathcal{F}_{P}.$ The reader will note that the parabolic $P$ under consideration is self-associate.

\begin{theorem}[Langlands]
\label{thm:Langlands}
For $f \in  I(s, \chi \times \sigma)$ we have: 
$$ \mathcal F_{P} ( {\rm Eis}_{P}(s,f)) = f + T_{st}(s)(f).$$
\end{theorem}

\smallskip

The program enunciated by Harder \cite{harder-inventiones}, and continued in \cite{harder-raghuram}, on Eisenstein cohomology and special values of 
$L$-functions, hinges on giving a cohomological interpretation to Langlands's theorem.

\medskip
\subsection{Holomorphy of Eisenstein series at point of evaluation} 
\label{sec:holomorphy}
 
We will be specifically interested in Langlands's constant term theorem at our point of evaluation $s = -n.$ For this we will need the 
holomorphy of ${\rm Eis}_{P}(s,f)$ at $s = -n.$ Now, suppose $f$ is a pure tensor $f = \otimes'_v f_v$, where for a place $v$ outside a finite set $S$
of places, which includes all the archimedean places and the non-archimedean places where both $\chi$ and $\sigma$ are ramified, the local vector $f_v$ is the normalised spherical vector $f_v^\circ$. Then
using \eqref{eqn:L-G-K} for all $v \notin S$ we have
$$
T_{\rm st}(s)(f) \ = \ 
\frac{L^S(s, \chi \times \sigma)}{L^S(1+s, \chi \times \sigma)} 
\otimes_{v \in S} T_{{\rm st}, v}(s)(f_v) \otimes \otimes_{v \notin S} \tilde{f}_v^\circ, 
$$
where the partial $L$-function is defined by $L^S(s, \chi \times \sigma) = \prod_{v \notin S} L_v(s, \chi_v \times \sigma_v).$ 
Holomorphy of ${\rm Eis}_{P}(s,f)$ is governed by nonvanishing of $L^S(1+s, \chi \times \sigma)$ in the denominator, and 
the numerator $L^S(s, \chi \times \sigma)$ not having any poles. The definition below and the following theorem are 
motivated from \cite[Sect.\,6.3.6]{harder-raghuram}: 

\medskip
\begin{definition}
\label{def:right-of-unitary} 
For a dominant integral weight $\mu \in X^*({}^\circ T \times E)$ and an integer $d$, we say $(d, \mu)$ is on the right of the unitary axis if $-(n+d) \geq 0.$ 
\end{definition}

\medskip
\begin{theorem}
\label{thm:right-of-unitary} 
Suppose the notation is as in the beginning of Sect.\,\ref{sec:arith-int-op}. Assume that $(d, \mu)$ is on the right of the unitary axis.  Then 
for a section $f \in  I(s, \chi \times \sigma)$, the Eisenstein series ${\rm Eis}_{P}(s,f)$ is holomorphic at $s = -n.$ 
\end{theorem}

Recall that not only is $\sigma$ cuspidal, but its Arthur parameter $\Psi(\sigma)$ is cuspidal on $\GL_{2n}/F.$ The degree $2n$ $L$-function 
$L^S(s, \, \chi \times \sigma)$ is the same as the degree $2n$ standard $L$-function $L^S(s, \, \chi \times \Psi(\sigma))$ 
for $\GL_{2n} \times \GL_1$ which is entire. Furthermore, if $-(n+d) \geq 0,$ then of course $1-(n+d) \geq 1,$ and $L^S(1-n, \, \chi \times \sigma) = 
L^S(1-(n+d),\, \chi^{\circ} \times \Psi(\sigma)) \neq 0$ by a well-known theorem of Jacquet--Shalika \cite{jacquet-shalika-invent}. The reader 
is referred to Harish-Chandra \cite{harish-chandra} for general aspects of holomorphy of Eisenstein series when the data is to the right of the unitary axis.

\medskip
\subsection{Intertwining operators at an archimedean place}
\label{sec:archimedean}

Throughout this subsection we fix a place $v \in S_{\infty}$; and recall that we have suppressed the $\iota : E \to \C$ in this section. 

\medskip
\subsubsection{\bf Irreducibility of induced representations}
\label{sec:irreduciblity}
We let $G_v = \rO(n+1,n+1)(\R),$ and $P_v = P_0(F_v) = P_0(\R)$. 
Let $H_v:= \SO(n+1,n+1)(\R)$ and $P_{H_{v}} = P_v \cap H_v$ be the corresponding maximal parabolic subgroup of $H_v$, and 
$P_{H_{v}}= M_{H_{v}} U_{P_v}$ be its Levi decomposition; $M_{H_{v}} = \GL_{1}(\R) \times \SO(n,n)(\R).$ We have the data $(\mu, d),$ 
with $\mu_{\rm min} \geq1$ and satisfying the 
inequalities of the combinatorial data: $1 - \mu_{\rm min} \leq -(n+d) \leq  \mu_{\rm min}-1.$ The local data consists of 
$\chi_v = \chi^\circ_v[-d]$ (with $\chi^\circ_v = {\rm sgn}^{\epsilon_v}$ a quadratic character of $\GL_{1}(\R)$, where $\epsilon_v \in \{0,1\}$), 
and $\sigma_v$ which is the discrete series representation $\D_{\mu^{v}}$ of $^{\circ}M_{v} = \rO(n,n)(\R)$. 

\begin{proposition} 
\label{prop:irreducible-mods-at-infty}
Assume we are on the right of unitary axis, i.e., $-(n+d) \geq 0$. The induced representations 
$$ 
\Ind _{P_{v}} ^{G_{v}}( \chi_{v}[-n] \times \sigma_v) \quad {\rm and} \quad \Ind _{P_{v}}^{G_{v}}( \chi_{v}^{-1}[n] \times {}^{\kappa_{2n}}\sigma_v)
$$ 
are irreducible. The standard intertwining operator $T_{{\rm st},v}(s)|_{s = -n}$ is an isomorphism between these two induced modules.  
\end{proposition}

\begin{proof} 
The proof involves reducing to a well-known criterion of irreducibility of Casselman and Shahidi \cite{casselman-shahidi} which is applicable to $H_v$ (and not $G_v$). Recall that, $\Ind _{P_{v}} ^{G_{v}}( \chi_{v}[-n] \times \sigma_v) = \aInd_{P_{v}} ^{G_{v}}( \chi_{v} \times \sigma_v),$ and since $G_v = H_v P_v$, we get
$$
{\rm Res}_{H_{v}} \aInd_{P_{v}} ^{G_{v}}~( \chi_{v} \times \sigma_v) = ~^{a} {\rm Ind}_{P_{H_{v}}} ^{H_{v}}( \chi_{v} \times { \sigma_v}|_{H_v}). 
$$
We know that  ${ \sigma_v}|_{ \SO(n,n)(\R)} = \D_{\mu_v}|_{ \SO(n,n)(\R)} = \mathcal D_{\mu_v} \oplus~ \!^ {\kappa_{n}}\mathcal D_{\mu_v},$ where 
$\mathcal D_{\mu_v}$ and ${}^{\kappa_{2n}}\mathcal D_{\mu_v}$ are inequivalent irreducible representations of $\SO(n,n)(\R)$. Thus 
$$
{\rm Res}_{H_{v}}\aInd_{P_{v}} ^{G_{v}}( \chi_{v} \times \sigma_v) \ = \ 
\aInd_{P_{H_{v}}} ^{H_{v}}( \chi_{v} \times  \mathcal D_{\mu_v}) \ \oplus \ \aInd_{P_{H_{v}}} ^{H_{v}}( \chi_{v} \times {}^{\kappa_{2n}}\mathcal D_{\mu_v}). 
$$
We show that both the summands on the right are irreducible representations of $H_v.$ 

\smallskip

\noindent
{\bf Case 1)} Suppose $-(n + d) > 0$. 
The module $\aInd_{P_{H_{v}}} ^{H_{v}}( \chi_{v} \times  \mathcal D_{\mu_v})$ is the same as the induced module $I(s, \chi_{v}^\circ \times  \mathcal D_{\mu_v})|_{s = -(n+d)}$ as defined in Sect.\,\ref{sec:ind-rep}, but adapted to $H_v$. 
Now, \cite[Prop.\,5.3]{casselman-shahidi} says that this module is irreducible 
if and only if $L(1+n+d, \chi_{v}^\circ{}^{-1} \times  \mathcal D_{\mu_v})^{-1} \neq 0.$ This condition that the inverse of the local $L$-factor not being zero is 
the same as the local $L$-factor $L(1+n, \chi_{v}^{-1} \times  \mathcal D_{\mu_v})$ not being a pole, which is guaranteed by criticality of $s = -n.$

\smallskip

\noindent
{\bf Case 2)}  Suppose $-(n + d) = 0$. Since 
$s = -n$ is critical, we have 
$$
L(-n, \, \chi_v \times \sigma_v) \ = \ L(-n-d, \, \chi_v^\circ \times  \sigma_v) \ = \ L(0,\, \chi_v^\circ \times  \sigma_v)
$$ 
is finite; by the local functional equation we get $L(1, \chi_v ^\circ{}^{-1}\times  \sigma_v)$ is finite; 
which implies, as in the proof of Prop.\,5.3 in \cite{casselman-shahidi}, that $\Ind _{P_{H_{v}}} ^{H_{v}}  (\chi_{v}[-n] \times \mathcal D_{\mu_v})$ is irreducible. 
 
\smallskip 
Irreducibility of $\aInd_{P_{H_{v}}} ^{H_{v}}( \chi_{v} \times {}^{\kappa_{2n}}\mathcal D_{\mu_v})$ is proved analogously. 

\smallskip

We now show that the two summands $\aInd_{P_{H_{v}}} ^{H_{v}}( \chi_{v} \times  \mathcal D_{\mu_v})$ and 
$\aInd_{P_{H_{v}}} ^{H_{v}}( \chi_{v} \times {}^{\kappa_{2n}}\mathcal D_{\mu_v})$ are conjugate by ${\kappa_{2n+2}} \in G_v \setminus H_v$ but not equivalent. This will show that $ \aInd _{P_{v}} ^{G_{v}}( \chi_{v} \times \sigma_v)$ is irreducible. To see 
$$
{}^{\kappa_{2n+2}} (\aInd_{P_{H_{v}}} ^{H_{v}}( \chi_{v} \times  \mathcal D_{\mu_v})) 
\ \cong \ 
\aInd_{P_{H_{v}}} ^{H_{v}}( \chi_{v} \times {}^{\kappa_{2n}}\mathcal D_{\mu_v}). 
$$
observe that the map $f \mapsto \tilde{f}$, where $\tilde{f}(h) = f(\kappa_{2n+2}^{-1} \, h \, \kappa_{2n+2}),$ gives the required $H_v$-equivariant isomorphism. 
Suppose the induced representations 
$\aInd_{P_{H_{v}}}^{H_{v}}(\chi_{v} \times  \mathcal D_{\mu_v})$ and $\aInd_{P_{H_{v}}} ^{H_{v}}( \chi_{v} \times {}^{\kappa_{2n}}\mathcal D_{\mu_v})$ are  equivalent, then by Mackey theory, it means either $\mathcal D_{\mu_{v}} \cong ~^{\kappa_{2n}} \mathcal D_{\mu_{v}}$  
or $^{w_{P_0}} (\chi_v \times \mathcal D_{\mu_{v}}) = \chi_{v} \times {}^{\kappa_{2n}} \mathcal D_{\mu_{v}}.$
The first condition is impossible, since $\mu_{\rm min} \neq 0$ which gives $\mathcal D_{\mu_{v}} \ncong  {}^{\kappa_{2n}} \mathcal D_{\mu_{v}}.$ 
If the second condition is true, then since $^{w_{P_{0}}} (\chi_v \times \mathcal D_{\mu_{v}} )= \chi_v^{-1} \times~ ^{\kappa_n} \mathcal D_{\mu_{v}}$, we 
have $\chi_{v} =  \chi_v^{-1},$ hence $d =0$, which is impossible as $-(n+d) \geq 0$. 

\medskip

From Shahidi's result \cite{shahidi-duke85} on local factors using Whittaker functionals, exactly as in the proof of \cite[Prop.\,7.24]{harder-raghuram}, 
we conclude that $T_{\st, v}(s, w_{P_0})$ is holomorphic at $s = -n$ and gives an isomorphism between the two induced modules. 
\end{proof}

\medskip
\subsubsection{\bf Factorisation into rank-one intertwining operators}
\label{sec:factor-rank-one}
By well-known theorems of Harish-Chandra and Casselman (see, for example, \cite{wallach}), the representation $\D_{\mu^v}$ can be realised as a submodule of the principal series $$
\Ind_{B_{{}^{\circ}\!{M_{P}}_v}}^{{}^{\circ}\!{M_{P}}_v} ( \chi_{1,v} \times \cdots \times \chi_{n,v}), 
$$ 
where $\chi_{j,v}$ is the character of $\R^{\times}$ defined as
 $$ 
 \chi_{j,v} (t) \ = \ |t|^{\mu^v_{j} +n-j} {\rm sgn}(t)^{\mu^v_{j}}, \quad \forall  ~1 \leq j \leq n, 
 $$
 giving us the character $\chi_{1,v} \times \cdots \times \chi_{n,v}$ on the Borel subgroup $B_{{}^{\circ}\!{M_{P}}_v}$ of ${}^{\circ}\!{M_{P}}_v.$ 
 For later use, we will call $\mu^v_{j} +n-j$ the exponent of the character $\chi_{j,v}$. 
 Recall that $\chi_{v}$ is of the form: $\chi_{v}(t) = |t|^{-d} {\rm sgn}(t)^{\epsilon_{v}}$. The induced modules on either side of the standard intertwining 
 operator $T_{{\rm st},v}$ can be embedded into principal series representations as: 
 $$ 
 \Ind _{P_v}^{G_v}(\chi_{v}[s] \times \D_{\mu^v}) \ \hookrightarrow \ 
 \Ind_{B_v}^{G_v} ( \chi_{v}[s] \times \chi_{1,v} \times \cdots \chi_{n-1,v} \times \chi_{n,v});
 $$
 note the use of normalised parabolic induction; and similarly, 
 $$ 
 \Ind _{P_v}^{G_v}(\chi_{v}^{-1}[-s] \times {}^{\kappa_{2n}} \D_{\mu^v}) 
 \ \hookrightarrow \ 
 \Ind_{B_v}^{G_v}( \chi_{v}^{-1}[-s] \times \chi_{1,v} \times \cdots \times \chi_{n-1,v} \times \chi_{n,v}^{-1}).
 $$ 
 For brevity, let's denote 
 \begin{equation}
 \label{eqn:varkappa}
 \varkappa_v \ := \ 
 \chi_{v}[s] \times \chi_{1,v} \times \cdots \chi_{n-1,v} \times \chi_{n,v},
 \end{equation}
 hence,  $w_{P_0}(\varkappa_v) = \chi_{v}^{-1}[-s] \times \chi_{1,v} \times \cdots \times \chi_{n-1,v} \times \chi_{n,v}^{-1}.$ 
 The local intertwining operator $T_{{\rm st},v}(s, w_{P_0})$ can be extended to an intertwining operator between the principal series representations: 
 $$
 A_v(w_{P_0}, \varkappa_v)\ : \ 
 \Ind _{B_v}^{G_v}(\varkappa_v) 
 \longrightarrow \ 
 \Ind_{B_v}^{G_v}(w_{P_0}(\varkappa_v)); 
$$
where, for any Weyl group element $w \in W_{G_0},$ we define $A_v(w, \varkappa_v)$ by the integral
$$
A_v(w, \varkappa_v)(f)(g) \ = \ \int \limits_{U_{w}} f(w^{-1} \, u\, g) \, du, \quad U_{w} := U \cap w U^- w^{-1}.
$$ 
Note that $U_w$ is the product of the root groups $U_\alpha$ corresponding to 
$\{\alpha > 0 : w^{-1}\alpha < 0\}.$ In particular, $U_{w_{P_0}} = U_{P_0}.$ Now, recall the factorisation in \eqref{eqn:factor-w-P}, 
which we rewrite as 
$$
w_{P _{0}} =   
\ss_{\alpha_{0}} ~\ss_{\alpha_{1}} ~\cdots ~
\ss_{\alpha_{n-3}} ~ \ss_{\alpha_{n-2}} ~\ss_{\alpha_{n}} ~\ss_{\alpha_{n-1}} \ss_{\alpha_{n-2}}  ~\cdots~
\ss_{\alpha_{1}} ~\ss_{\alpha_{0}} =:
\ss_{2n} ~\ss_{2n-1} ~ \cdots ~ \ss_2 ~\ss_1, 
$$
i.e., $\ss_j$ denotes the simple reflection appearing in $j$-th position from right in \eqref{eqn:factor-w-P}.
A well-known lemma due to Langlands (see Shahidi \cite[Lem.\,4.2.1]{shahidi-book}) uses the above factorisation of $w_{P_0}$ to give a factorisation of the operator $A_v(w, \varkappa_v)$ into rank-one operators: 
\begin{multline}
\label{eqn:factor-int-op-p-s-rep}
A_v(w, \varkappa_v) \ = \\ 
A_v(\ss_{2n}, \, \ss_{2n-1}\ss_{2n-2}\dots \ss_1\varkappa_v) \, \circ \, 
\cdots
\, \circ \, A_v(\ss_3, \,  \ss_2\ss_1\varkappa_v)
\, \circ \, A_v(\ss_2, \, \ss_1\varkappa_v)
\, \circ \, A_v(\ss_1, \, \varkappa_v).
\end{multline}

\smallskip
A few words of explanation can be helpful. (For more details the reader is referred to Kim's expository article \cite[Thm.\,6.2]{kim-notes}.)
If $w = w_2w_1$ with $l(w) = l(w_2) + l(w_1),$ then the operator $A_v(w, \varkappa_v)$ factors as a cocyle: 
$$
A_v(w, \varkappa_v) = A_v(w_2, w_1\varkappa_v) \circ A_v(w_1, \varkappa_v),
$$ 
and furthermore, 
$$
\{\alpha > 0 : w^{-1}\alpha < 0\} = \{\alpha > 0 : w_1^{-1}\alpha < 0\} \cup w_1\{\alpha > 0 : w_2^{-1}\alpha < 0\};
$$ 
the union is disjoint.
Let us enumerate the set $\Delta_{U_{0}}$ as follows:  
  \begin{multline}
  \label{eqn:betas}
  \beta_1 = e_0 - e_1, ~\beta_2 = e_0 - e_2, ~\ldots,~  \beta_n = e_0 - e_n, \\ 
  \beta_{n+1} = e_0 + e_n, ~\beta_{n+2} = e_0+e_{n-1},~\ldots,~ \beta_{2n} = e_0 + e_1.
  \end{multline}
In \eqref{eqn:factor-int-op-p-s-rep}, the $j$-th integral is happening over the root-space corresponding to $\beta_j.$ 

\smallskip
The simplest nontrivial example can help to visualise the sequence of rank-one operators: take $n=2$, and suppose $\chi_{v}[s] \times \chi_{1,v} \times \chi_{2,v},$ is a short-form notation for $\Ind _{B_v}^{G_v}(\varkappa_v)$, then \eqref{eqn:factor-int-op-p-s-rep} is the composition of the $4$ rank-one intertwining operators:
\begin{multline*}
\chi_{v}[s] \times \chi_{1,v} \times \chi_{2,v} \ \stackrel{\ss_{\alpha_0}}{\longrightarrow} \ 
\chi_{1,v} \times \chi_{v}[s]\times \chi_{2,v} \ \stackrel{\ss_{\alpha_1}}{\longrightarrow} \ 
\chi_{1,v} \times \chi_{2,v} \times \chi_{v}[s]  \ \\ 
\stackrel{\ss_{\alpha_2}}{\longrightarrow} \ 
\chi_{1,v} \times \chi_{v}^{-1}[-s] \times  \chi_{2,v}^{-1} \ \stackrel{\ss_{\alpha_0}}{\longrightarrow} \ 
\chi_{v}^{-1}[-s] \times \chi_{1,v} \times \chi_{2,v}^{-1}. 
 \end{multline*}

\medskip 

Each operator on the right hand side of \eqref{eqn:factor-int-op-p-s-rep} is the induction to $G_v$ of an $\SL(2)$ intertwining operator; hence the adjective rank-one. We now collect some well-known facts about the $\SL(2)$ situation.

\medskip
 \subsubsection{\bf The ${\rm SL}(2)$ calculation}
 \label{sec:SL2-calculation}
Just for this paragraph, we let $G={\rm SL}_2(\R)$ and $B$ the standard Borel subgroup of all upper triangular matrices in $G$.
 For $z \in \C$ and $\epsilon \in \{ 0, 1\}$, let $\chi = \chi_{z,\epsilon}$ be the character of the standard diagonal torus of $G$ defined by 
 $$
 \chi(\begin{smallmatrix}  t & 0 \\ 0  & t^{-1} \end{smallmatrix}) \ = \ |t|^{z} {\rm sgn}(t)^{\epsilon} \quad \forall~t \in \R^{\times}. 
 $$
 If $\Ind_B^G(\chi)$ denotes the normalised induction from $B$ to $G$ of the character $\chi$, then the standard intertwining operator $T_\st$ is defined between the spaces $\Ind_B^G(\chi) \rightarrow \Ind_B^G (w(\chi))$, where $w = \left(\begin{smallmatrix}  & -1 \\ 1 &\end{smallmatrix}\right)$ is the non-trivial element of Weyl group  and $w(\chi) = \chi_{-z,\epsilon}.$
For $m \in \Z$, let $\vartheta_m$ be the character of  $\SO(2)$ defined by 
$\left(\begin{smallmatrix} \cos\theta & -\sin\theta \\ \sin\theta  & \cos\theta  \end{smallmatrix} \right) \mapsto e^{{\boldsymbol i}m\theta},$ and suppose 
$V_m$ denotes the representation space of $\vartheta_m.$ The $\SO(2)$-type $(\vartheta_m, V_m)$ appears in $\Ind_B^G(\chi)$ if and only if 
$m \equiv \epsilon \pmod{2}$, and in which case it appears with multiplicity one; same comment applies to $\Ind_B^G(w(\chi))$. 
Let $f_{\chi, \epsilon}$ (resp., $f_{w(\chi), \epsilon}$) 
be the normalised (i.e., taking the value $1$ on the identity element $I_2$ of $G$) highest weight vector in the minimal $\SO(2)$-type $V_{\epsilon}$ of $\Ind_B^G(\chi)$ 
(resp., $\Ind_B^G(w(\chi))$). Then $T_{\st}$ maps $f_{\chi, \epsilon}$ to a scalar multiple of $f_{w(\chi), \epsilon}$. The scalar is computed by the integral: 
 $$
 T_{\st}f_{\chi, \epsilon}\left(\begin{pmatrix} 1 & 0 \\ 0 & 1 \end{pmatrix}\right) \ = \ 
 \int \limits_{\R} f_{\chi, \epsilon} \left( \begin{pmatrix}  & -1 \\ 1 &\end{pmatrix}^{-1} \begin{pmatrix} 1 & u \\ 0 & 1 \end{pmatrix}\right) ~du.
 $$ 
A standard calculation (see, for example, \cite[Sect.\,9.2.3, 9.2.4]{harder-raghuram}) gives: 
$$ 
 T_\st
 f_{\chi, \epsilon} \ = \ 
 (-{\boldsymbol i})^{\epsilon}
 \frac{\Gamma(\frac{z + \epsilon}{2}) \Gamma(\frac{1}{2})} {\Gamma(\frac{z + \epsilon + 1}{2})}  \, f_{w(\chi), \epsilon}. 
 $$ 
 Similarly, for $m \equiv \epsilon \pmod{2},$ and $f_{\chi, m}$ (resp., $f_{w(\chi), m}$) the normalised weight vectors in the $\SO(2)$-type $\vartheta_m$ of 
 $\Ind_B^G(\chi)$ (resp., $\Ind_B^G(w( \chi))$), we get: 
 \begin{equation} 
 \label{T-st} 
 T_{st}(f_{\chi, m}) \ = \ 
   (-{\boldsymbol i})^{m} 
   \frac{\Gamma(\frac{1}{2}) ~\Gamma(\frac{z}{2} )~\Gamma(\frac{z+1}{2})} 
   {\Gamma \left (\frac{z + 1 + m}{2} \right) \Gamma \left(\frac{z + 1 - m}{2} \right)} \, 
   f_{w(\chi), m}.
  \end{equation}
Let $m = 2k + \epsilon$. Consider a rational function of the variable $z$,  
defined by
  \[
    M(z)_{k} \ =\ 
     \begin{cases}
\dfrac{(z - \epsilon - 1)~   (z - \epsilon - 3)~ \cdots~ (z - \epsilon -(2k-1) )     }
{(z + \epsilon + 1)~   (z + \epsilon + 3)~ \cdots~ (z + \epsilon + (2k-1) )     } &~\text{if}~k > 0, 
\\ \smallskip
1 &~\text{if}~k =0, \\ \smallskip
\dfrac{(z + \epsilon - 1)~   (z + \epsilon - 3)~ \cdots~ (z + \epsilon +(2k + 1) )     }
{(z - \epsilon + 1)~   (z - \epsilon + 3)~ \cdots~ (z - \epsilon - (2k +1) )     } &~\text{if}~k < 0.
   \end{cases} \smallskip
   \]
We can rewrite \eqref{T-st}  as
\begin{equation}
\label{T-st-on-K-types}
 T_{st}(f_{\chi, m}) \ = \ 
 (-{\boldsymbol i})^{2k + \epsilon} ~ \frac{\Gamma(\frac{1}{2} )~\Gamma(\frac{z + \epsilon}{2})}{\Gamma(\frac{z + \epsilon + 1}{2})} ~M(z)_{k}~f_{w(\chi), m}.
\end{equation}

\medskip
\subsubsection{\bf The intertwining operator on a highest weight vector of a lowest $K$-type}
\label{sec:Factor-Int-op}

Recall from Prop.\,\ref{prop:irreducible-mods-at-infty} that the induced representations 
$$
\sI_v \ := \ \Ind _{P_{v}}^{G_{v}}(\chi_{v}[-n] \times \D_{\mu^v}) 
\quad {\rm and} \quad 
\tilde\sI_v \ := \ \Ind _{P_{v}}^{G_{v}}(\chi_{v}^{-1}[n] \times {}^{\kappa_{2n}} \D_{\mu^v}) 
$$ 
are irreducible, and $T_{\st, v}(-n, w_{P_0})$ is an isomorphism between them. Fix a highest weight vector ${\bf f}_0$ of the lowest $K_v^{\circ}$-type in $\sI_v$ normalised such that ${\bf f}_0(1) = 1$; this normalisation is possible as $\sI_v $ is subrepresentation of a principal series representation. Similarly, 
fix $\tilde{\bf f}_0$ in $\tilde\sI_v.$ There exists $c_v(  \chi_v, \sigma_v) \in \C^{\times}$ such that 
$$
T_{\st,v}(-n, w_{P_0})({\bf f}_0) \ = \ c_v(\chi_v, \sigma_v) \, \tilde{\bf f}_0.
$$
Now we compute this scalar $c_v(  \chi_v, \sigma_v)$ and relate it to a ratio of archimedean $L$-factors which, up to nonzero rational numbers, has the same form as \eqref{eqn:L-G-K}; see Prop.\,\ref{prop:c-chi-sigma-archimedean} below. 

\smallskip

Recall our notation: $\varkappa_v$ in \eqref{eqn:varkappa}, and $\{\beta_1,\dots, \beta_{2n}\}$ the roots 
defined in \eqref{eqn:betas} whose root spaces generate $U_{P_0}.$  Fix $1 \leq j \leq 2n$.  
Suppose $\beta_j^\sv$ is the coroot corresponding to $\beta_j$, then $\langle \beta_{j}^\sv, \varkappa_v \rangle$ denotes the exponent of the character 
$t \mapsto \varkappa_v(\beta_{j}^\sv(t)).$ We have: 
$$
\langle \beta_{j}^\sv , \varkappa_v \rangle = 
\begin{cases} 
-d -(\mu^v_j + n -j), & 1 \leq j \leq n, \\
-d + (\mu^v_{2n+1-j} -n-1+j), & n+1 \leq j \leq 2n. 
\end{cases}
$$
Considering the action on $\SO(2)$ types as in \eqref{T-st-on-K-types}, the $j$-th operator 
from the right in \eqref{eqn:factor-int-op-p-s-rep} can be written as: 
\begin{equation} 
\label{standard-intertwining} 
A_v(\ss_j, \,  \ss_{j-1} \dots \ss_2\ss_1\varkappa_v) \ = \ 
\sqrt{\bpi} ~ \frac{ \Gamma \left( \dfrac{s+ \langle \beta_{j}^\sv, \varkappa_v \rangle + \epsilon_{j} } {2} \right) }{
\Gamma \left( \dfrac{s+ \langle \beta_{j}^\sv, \varkappa_v \rangle + \epsilon_{j} + 1} {2} \right)} 
\ M_{j}(s), 
\end{equation}
where, $\epsilon_j \in \{0, 1\}$ is defined as $\epsilon_j \equiv \epsilon_0 + \mu^v_j \pmod{2}$ for $1 \leq j \leq n$, 
and $\epsilon_j:  = \epsilon_{2n + 1- j}$ for $j \geq n+1$, 
and $M_{j}(s)$ is the diagonal matrix with rows and columns indexed by $k \in \Z$ defined by: 
\begin{align*}
\begin{split}
\left( M_{j}(s) \right)_{0,0} & 
=  (-{\boldsymbol i})^{\epsilon_j},\\
\left( M_{j}(s) \right)_{k,k} & 
= (-{\boldsymbol i})^{2k+ \epsilon_j} \prod \limits_{l=1}^{k}   
 \dfrac{ {\left( s +  \langle \beta_{j}^\sv, \varkappa_v \rangle  - \epsilon_j - (2l-1)  \right)} }{
  {\left( s +  \langle \beta_{j}^\sv, \varkappa_v \rangle  + \epsilon_j + (2l -1)  \right)}} \quad \mbox{if $k \geq 1$,}\\
\left( M_{j}(s) \right)_{k,k} & 
= (-{\boldsymbol i})^{2k+ \epsilon_j}  \prod \limits_{l=1}^{-k}   
 \dfrac{ {\left( s +  \langle \beta_{j}^\sv, \varkappa_v \rangle  + \epsilon_j - (2l-1)  \right)} }{
  {\left( s +  \langle \beta_{j}^\sv, \varkappa_v \rangle  - \epsilon_j + (2l -1)  \right)}} \quad \mbox{if $k \leq -1$},
  \end{split}
  \end{align*}
Note that up to the scaling factor $(-{\boldsymbol i})^{\epsilon_j}$, the entries of $M_{j}(s)$ are in $\Q(s).$ 

\smallskip

Consider the effect of the composition of the operators in \eqref{standard-intertwining}, for the factorisation in \eqref{eqn:factor-int-op-p-s-rep}, at our point of evaluation $s =-n$, on the highest weight vector of a lowest $K$-type. 
If $1 \leq j \leq n$, then
$$ 
-n + \langle \beta_{j}^\sv, \varkappa_v \rangle + \epsilon_{j} \ = \ 
-n +  (-d - (\mu_j + n - j) + \epsilon_{j}  \ \equiv \ d + j + \epsilon_{0} \pmod{2}, 
$$
and  hence $-n + \langle \beta_{j}^\sv, \varkappa_v \rangle + \epsilon_{j}$ is even for exactly $r$ values of $j$ (either for $j = 1, 3, \ldots, n-1$ or for $j = 2, 4, \ldots, n$) and it is odd for other $r$ values of $j$. A similar analysis for all $j \geq n+1$ gives us that $-n + \langle \beta_{j}^\sv, \varkappa_v \rangle + \epsilon_{j}$ is even for $r$ many values of $j$, and is odd for the remaining $r$ values of $j$.
In view of Prop.\,\ref{prop:critical-points}, we have for all $1 \leq j \leq n-1$: 
$$
1 - \mu^v_j  \ \leq \ 1 - |\mu^v_n| \ \leq \ \ -(d+n) \ \ \leq \ |\mu^v_n| - 1 \ \leq \ \mu^v_j - 1. 
$$
So we have 
\begin{align}\label{inequalities}
\begin{split}
&-n + \langle \beta_{j}^\sv, \varkappa_v \rangle + \epsilon_{j} \leq 0 \quad \text{for}~n~\text{values of}~j,~\text{and} \\
&-n + \langle \beta_{j}^\sv, \varkappa_v \rangle + \epsilon_{j} \geq 1\quad \text{for remaining}~n~\text{values of}~j.
\end{split}
\end{align}

\smallskip
Let us  recall basic properties of the Gamma function: For $m \in \Z$, $\Gamma(m) \in \Z_{>0}$ for $m \geq 1$, and 
and has a simple pole with rational residue for $m \leq 0$; these are all the poles; $\Gamma(z)$ is nonvanishing everywhere; and  
$\Gamma (m+ \frac{1}{2}) \in \sqrt{\bpi}~ \Q^{\times}.$ 
Consider
$$ 
\Phi(s) \ : =  \ 
\prod \limits_{j=1}^{2n}  \sqrt{\bpi} ~ 
\frac{ \Gamma \left( \dfrac{s+ \langle \beta_{j}^\sv, \varkappa_v \rangle + \epsilon_{j} } {2} \right)}
{\Gamma \left( \dfrac{s+ \langle \beta_{j}^\sv, \varkappa_v \rangle + \epsilon_{j} + 1} {2} \right) }. 
$$
Then, $\Phi(s)$ is holomorphic at $s = -n$ since exactly $r$ of the $\Gamma$-factors in the numerator have a pole at $s = -n,$ and 
exactly $r$ of the $\Gamma$-factors in the denominator have a pole at $s = -n$ (due to the parity conditions on 
$-n + \langle \beta_{j}^\sv, \varkappa_v \rangle + \epsilon_{j}$ and inequalities \eqref{inequalities}).
Note that the contribution from all factors of the type $(-{\boldsymbol i})^{2k+ \epsilon_j}$ is $\pm 1$ as $\epsilon_j = \epsilon_{2n + 1- j}$ for all $j$.\medskip

\medskip

The calculation in preceding paragraphs, together with \eqref{eqn:L-fact-at-v} as used in Cor.\,\ref{cor:arch-ratio-L-values}, 
proves the following proposition: 

\begin{proposition}
\label{prop:c-chi-sigma-archimedean}
For every $v \in S_{\infty}$, the nonzero complex number $c_v(  \chi_v, \sigma_v)$ defined by 
$T_{\st,v}(-n, w_{P})({\bf f}_0) = c_v(  \chi_v, \sigma_v) \, \tilde{\bf f}_0$ satisfies
$$
c_v(  \chi_v, \sigma_v) \ \approx_{\Q^{\times}} \ 
\bpi^{n} \ \approx_{\Q^{\times}} \ 
\frac{L(-n, \chi_{v} \times \sigma_v) }{L(1-n, \chi_{v} \times \sigma_v) }, 
$$
where, $\approx_{\Q^{\times}}$ means equality up to a nonzero rational number. 
\end{proposition}

\medskip

We now consider the map induced by $T_{\st,v}(-n, w_{P})$ in cohomology, for which we need to understand the relative Lie algebra cohomology groups 
of the induced representations $\sI_v$ and $\tilde\sI_v.$ Although this is a well-known calculation due to Delorme (see Borel and Wallach 
\cite[Thm.\,III.3.3]{borel-wallach}) for connected real reductive groups, 
in our context of disconnected groups it entails an arduous exercise in book-keeping, towards which we need to first discuss the cohomology of the inducing data.

\medskip
\subsubsection{\bf Cohomology of the inducing data $\chi \times \sigma$}
\label{sec:coh-inducing-data}

The Levi factor of the parabolic subgroup $P_v = P_0(F_v)$ of $G_v = G_0(F_v) = \rO(n+1,n+1)(\R),$ will be denoted variously as 
$M_v = M_{P_v} = M_{P_0}(F_v) = A_v \times {}^\circ \! M_v$, where $A_v = \R^\times$ and ${}^\circ \! M_v = \rO(n,n)(\R).$ 
Let us  denote the maximal 
compact subgroup of ${}^\circ \! M_v$ by $K_{{}^\circ \! M_v},$ and their respective Lie algebras by ${}^\circ\fm_v$ and $\fk_{{}^\circ \! M_v}.$ 
The $({}^\circ\fm_v, \fk_{{}^\circ \! M_v})$-cohomology of the discrete series representation $\D_{\mu^v}$ was discussed in Sect.\,\ref{sec:def-boldface-D-mu};  
let's recall from \eqref{eqn:m-k-coh-D-mu}, with notations adapted to the current notations: 
\begin{multline*}
H^{q_0}({}^\circ\fm_v, \fk_{{}^\circ \! M_v};\, \D_{\mu^v} \otimes \mathcal M_{\mu^v} ) \ = \\  
H^{q_0}({}^\circ\fm_v, \fk_{{}^\circ \! M_v};\, D_{\mu} \otimes  \mathcal N_{\mu} ) \ \oplus \ 
H^{q_0}({}^\circ\fm_v, \fk_{{}^\circ \! M_v};\, {}^{s_{2n}}\!D_{\mu}  \otimes \mathcal N_{\mu} ) \ \oplus \\ 
H^{q_0}({}^\circ\fm_v, \fk_{{}^\circ \! M_v};\,  D_{{}^{\kappa_{2n}} \mu}  \otimes \mathcal N_{^{\kappa_{2n}} \mu} ) \ \oplus \   
H^{q_0}({}^\circ\fm_v, \fk_{{}^\circ \! M_v};\,  {}^{s_{2n}}\!D_{^{\kappa_{2n}} \mu} \otimes \mathcal N_{^{\kappa_{2n}} \mu} ), 
\end{multline*}
where $q_0 = n^2/2$ is the middle-degree for the symmetric space ${}^\circ \! M_v/K_{{}^\circ \! M_v},$ 
and every summand on the right hand side is one-dimensional. Furthermore, the group $\pi_0(K_{{}^\circ \! M_v})$ of connected components 
acts on $H^{q_0}({}^\circ\fm_v, \fk_{{}^\circ \! M_v};\, \D_{\mu^v} \otimes \mathcal M_{\mu^v})$ as the regular representation. 
Recall from \eqref{eqn:S-circ-M} that 
$$\cS^{{}^\circ\!M}_{{}^{\circ}C_f} = {}^{\circ}\!M(\Q) \backslash {}^{\circ}\!M(\A) / K^{\circ}_{{}^\circ\!M(\R)} {}^{\circ}C_f,$$ 
hence we need to understand the relative Lie algebra cohomology group $H^q({}^\circ\!\fm(\R), K^{\circ}_{{}^\circ\!M(\R)};\, \sigma_\infty \otimes \M_\mu),$ as a $\pi_0({}^\circ\!M(\R))$-module, where, 
$\sigma_\infty = \otimes_{v \in S_\infty} \sigma_v.$ Using a 
K\"unneth theorem over the archimedean places $v \in S_\infty$ we get that this cohomology group is nonzero if and only if $q = q_m = {\sf r}_F \cdot q_0 = {\sf r}_F \cdot n^2/2$, and in this `middle-degree', as a $\pi_0({}^\circ\!M(\R))$-module we have
$$
H^{q_m}({}^\circ\!\fm(\R), K^{\circ}_{{}^\circ\!M(\R)};\, \sigma_\infty \otimes \M_\mu) \ = \ 
\bigoplus_{\varepsilon \in \widehat{\pi_0({}^\circ\!M(\R))}} \varepsilon, 
$$
i.e., every character of $\pi_0({}^\circ\!M(\R))$ appears exactly once.

\medskip

Similarly, we denote the maximal 
compact subgroup of $A_v = \R^\times$ by $K_{A_v} = \{\pm 1\},$ and their respective Lie algebras by $\fa_v$ and $\fk_{A_v} = 0.$ 
The $(\fa_v, \fk_{A_v})$-cohomology, which of course is the same as the $\fa_v$-cohomology, of the character $\chi_v = |\ |^{-d} \otimes {\rm sgn}^{\epsilon_v}$ is 
$$
H^0(\fa_v;\, \chi_v \otimes \M_d) \ = \ \C \, {\rm sgn}^{\epsilon_v+d}, 
$$
as a $\pi_0(A_v)$-module, where $\M_d$ is the representation of $\GL_1$ given by $t \mapsto t^d.$ A locally symmetric space for $A$ is of 
the form $\cS^{A}_{B_f} = A(\Q) \backslash A(\A) / S(\R)^{\circ}B_f$, hence we need to understand the relative Lie algebra cohomology group
$H^\bullet (\fa(\R), S(\R)^\circ;\, \chi_\infty \otimes \M_d).$ As an $S(\R)^\circ$-module, $\chi_\infty \otimes \M_d = \otimes_v \, {\rm sgn}^{\epsilon_v + d}$ is the trivial representation. As is usual in such a context (see, for example, \cite[Sect.\,4]{raghuram-shahidi-imrn}) let's denote $\varepsilon_\chi = \otimes_v \, {\rm sgn}_v^{\epsilon_v+d}$ for the signature of $\chi,$ construed as a character of $\pi_0(A(\R)).$ Note that the signature $\varepsilon_\chi$ keeps track of the local signature of the finite-part as well as the parity of the integer $d$ accounting for the non-unitary part of the algebraic Hecke character $\chi$. 
Let $\fs(\R) = {\rm Lie}(S(\R)^\circ).$ We have: 
$$
H^\bullet (\fa(\R), S(\R)^\circ;\, \chi_\infty \otimes \M_d) \ = \ \wedge^\bullet(\fa(\R)/\fs(\R))^*, 
$$
on which $\pi_0(A(\R))$ acts via $\varepsilon_\chi.$
In particular, the cohomology is one-dimensional in the extreme degrees $0$ and ${\sf r}_F-1.$ 

\medskip
Putting the above discussions on the cohomology for 
${}^\circ\!M(\R)$ and $A(\R)$ together via a K\"unneth theorem, in cohomology degrees $q \in \{q_b, q_t\}$ (see \eqref{eqn:qb-and-qt-for-M}), we have
$$
H^q\left(\fm(\R), S(\R)^\circ K^{\circ}_{{}^\circ\!M(\R)}; \, (\chi_\infty \otimes \sigma_\infty) \otimes \M_{de_0 + \mu}\right) \ = \ 
\bigoplus_{\varepsilon \in \widehat{\pi_0({}^\circ\!M(\R))}} \varepsilon_\chi  \otimes \varepsilon, 
$$
as a module for $\pi_0(M(\R)) = \pi_0(A(\R)) \times \pi_0({}^\circ\!M(\R)).$ Of course, the coefficient system $\M_{de_0 + \mu}$ stands for 
$\M_{de_0} \otimes \M_{\mu}.$

\medskip
\subsubsection{\bf Cohomology of induced representations}
\label{sec:coh-induced-rep}
Assume the conditions in the combinatorial lemma hold for $d$ and $\mu.$ Hence, we have a (unique) Kostant representative $w \in W^P$ such that 
$\lambda := w^{-1}\cdot(de_0 + \mu)$ is dominant. (Recall that the uniqueness of $w$ comes from the proof of Lem.\,\ref{lem:comb-lemma}; the sign 
of $\mu^v_n$ determines $w^v.$) Recall from Sect.\,\ref{sec:groups-at-infty}: for $v \in S_\infty$, $K_v$ denotes the maximal compact subgroup of $G_0(F_v) = \rO(n+1,n+1)(\R)$; via the map $h$, we have $K_v \simeq \rO(n+1)(\R) \times \rO(n+1)(\R);$ $\overline{K}_v$ is as in 
\eqref{eqn:K-infty-to-divide-by-for-G}; $K_\infty = \prod_{v \in S_\infty} K_v$ and 
$\overline{K}_\infty = \prod_{v \in S_\infty} \overline{K}_v;$ note that $K_\infty^\circ \subset \overline{K}_\infty \subset K_\infty.$   
Our immediate aim is to describe 
$$
H^\q\left(\g(\R), \overline{K}_\infty; \, \aInd_{P(\R)}^{G(\R)}(\chi_\infty \otimes \sigma_\infty) \otimes \M_{\lambda}\right), 
$$
as a $\pi_0(G(\R))$-module, especially in cohomology degrees $\q \in \{\q_b, \q_t\}$ (see \eqref{eqn:qb-and-qt-for-G}). Applying K\"unneth theorem over
$v \in S_\infty$, this is the same as  
\begin{equation}
\label{eqn:g-v-bar-k-v-coh}
\bigoplus_{\sum_v \q_v = \q} \ 
\bigotimes_{v \in S_\infty} \ 
H^{\q_v}\left(\g_0(F_v), \overline{K}_v; \, \aInd_{P_0(F_v)}^{G_0(F_v)}(\chi_v \otimes \sigma_v) \otimes \M_{\lambda^v}\right). 
\end{equation}
The basic tool to understand these cohomology groups is 
\cite[Thm.\,III.3.3]{borel-wallach}, but in {\it loc.\,cit.}\ the ambient group is a connected reductive group, unlike our $G_0(F_v) = \rO(n+1,n+1)(\R)$ which has 
$\Z/2 \times \Z/2$ as its group of connected components, entailing a careful analysis of the relevant disconnected groups. 
We have 
$$
\pi_0(G_0(F_v)) \ =\ \pi_0(K_v) \ = \ \{I_{2n+2}, \, s_{2n+2}, \, \kappa_{2n+2}, \, s_{2n+2}\kappa_{2n+2}\}.
$$
Recall: $s_m = h(I_m, \delta_m)$, $\kappa_m = h(J_m\delta_mJ_m, I_m)$; hence $h(J_m\delta_mJ_m, \delta_m) = s_m\kappa_m.$ Hence, 
$\overline{K}_v$ is the group generated by $K_v^\circ$ and  $s_{2n+2}\kappa_{2n+2}.$ To compute the 
$(\g_0(F_v), \overline{K}_v)$-cohomology as in \eqref{eqn:g-v-bar-k-v-coh}, we first compute the $(\g_0(F_v), K_v^\circ)$-cohomology as a 
$\pi_0(K_v)$-module, and then 
take invariants under the element $s_{2n+2}\kappa_{2n+2} \in \pi_0(K_v).$ There is still the action of $K_v/\overline{K}_v = \Z/2$, the nontrivial element of which is represented by either $s_{2n+2}$ or by $\kappa_{2n+2}$, on $(\g_0(F_v), \overline{K}_v)$-cohomology. We have the following proposition:

\begin{proposition}
\label{prop:g-v-bar-k-v-coh} 
Recall $\q_b$ from \eqref{eqn:qb-and-qt-for-G}. In this degree, the cohomology group 
$$
H^{\q_b}\left(\g(\R), \overline{K}_\infty; \, \aInd_{P(\R)}^{G(\R)}(\chi_\infty \otimes \sigma_\infty) \otimes \M_{\lambda}\right)
$$
is one-dimensional, on which $K_\infty/\overline{K}_\infty$ acts via the character $\varepsilon_\chi$. 
\end{proposition}

\begin{proof}
Since $v$ is fixed, we will simply denote $G_0(\R) = G_0(F_v), P_0(\R) = P_0(F_v)$, etc. By an easy exercise in Mackey theory, we have:
\begin{multline*}
{\rm Res}_{G_0(\R)^\circ}\left(\aInd_{P_0(\R)}^{G_0(\R)}(\chi_v \otimes \sigma_v)\right) \ = \\ 
\aInd_{P_0(\R) \cap G_0(\R)^\circ }^{G_0(\R)^\circ} \left({\rm Res}_{M_0(\R) \cap G_0(\R)^\circ}(\chi_v \otimes \sigma_v) \right),
\end{multline*}
since, $G_0(\R) = P_0(\R) \cdot G_0(\R)^\circ,$ and the parabolic subgroup $P_0(\R) \cap G_0(\R)^\circ$ of $G_0(\R)^\circ$ has Levi decomposition 
$(M_0(\R) \cap G_0(\R)^\circ) \cdot U_{P_0(\R)},$ where the Levi factor 
$M_0(\R) \cap G_0(\R)^\circ = (A_0(\R) \times {}^\circ\!M_0(\R)) \cap G_0(\R)^\circ$ may be described as:
$$
\left\{ m_{(x,g)} = \left(\begin{matrix} x & & \\ & g & \\ & & x^{-1}\end{matrix}\right) : \ x \in \R^\times, \ \ g \in \rO(n,n)(\R) \right\} \ \cap \ \SO(n+1,n+1)(\R)^\circ.
$$
The inclusion $M_0(\R) \hookrightarrow G_0(\R)$ induces $\pi_0(M_0(\R)) \to \pi_0(G_0(\R))$ that factors as:
$$
M_0(\R)/M_0(\R)^\circ \ \twoheadrightarrow \ M_0(\R)/M_0(\R)\cap G_0(\R)^\circ \ \hookrightarrow \ G_0(\R)/G_0(\R)^\circ. 
$$
The group $M_0(\R)/M_0(\R)^\circ \simeq \Z/2 \times \Z/2 \times \Z/2$ is generated by the commuting involutions: 
$$
\{\tilde i_{2n} = m_{(-1, I_{2n})}, \quad \tilde s_{2n} = m_{(1, s_{2n})}, \quad \tilde \kappa_{2n} = m_{(1, \kappa_{2n})} \},
$$
and the group $G_0(\R)/G_0(\R)^\circ \simeq \Z/2 \times \Z/2$ is generated by $\{s_{2n+2}, \ \kappa_{2n+2}\}.$ The 
inclusion $M_0(\R) \hookrightarrow \ G_0(\R)$ identifies $\tilde s_{2n} = s_{2n+2}, \ \tilde \kappa_{2n} = \kappa_{2n+2}.$ Hence, 
$\pi_0(M_0(\R)) \to \pi_0(G_0(\R))$ is surjective; we conclude that $M_0(\R)/M_0(\R)\cap G_0(\R)^\circ \simeq G_0(\R)/G_0(\R)^\circ.$
Furthermore, 
$$
\overline{\pi_0}(M_0(\R)) \ := \ 
{\rm Ker}\left(\pi_0(M_0(\R)) \to \pi_0(G_0(\R)) \right) \ = \ 
\{1, \tilde i_{2n}\,\tilde s_{2n}\}, 
$$
since, $\tilde i_{2n}\,\tilde s_{2n} \in \SO(n+1) \times \SO(n+1).$ We have: 
$$
M_0(\R)\cap G_0(\R)^\circ \ = \ 
(A(\R)^\circ \times {}^\circ\!M_0(\R)^\circ) \, \bigsqcup \, \tilde i_{2n}\,\tilde s_{2n} (A(\R)^\circ \times {}^\circ\!M_0(\R)^\circ).
$$
The restriction of $\chi_v \otimes \sigma_v$ to $A(\R)^\circ \times {}^\circ\!M_0(\R)^\circ$ is the sum
$$
|\ |_v^{-d} \otimes D_{\mu^v} \ \oplus \ 
|\ |_v^{-d} \otimes {}^{s_{2n}} \!D_{\mu^v}  \ \oplus \
|\ |_v^{-d} \otimes {}^{\kappa_{2n}}\!D_{\mu^v}  \ \oplus \
|\ |_v^{-d} \otimes {}^{s_{2n}\kappa_{2n}}\!D_{\mu^v}, 
$$
of four inequivalent irreducible representations. When restricted to $M_0(\R)\cap G_0(\R)^\circ$, the first two summands will 
fuse together via an intertwining corresponding to $\tilde i_{2n}\,\tilde s_{2n}$ to give an irreducible representation which we will denote as 
$[\chi_v \otimes \calD_{\mu^v}];$ the element $\tilde i_{2n}$ will keep track of the parity of $\chi$, and $\tilde s_{2n}$ intertwines 
$D_{\mu^v}$ with ${}^{s_{2n}} \!D_{\mu^v}.$ Similarly, the last two summands fuse together to give an inequivalent irreducible 
$[\chi_v \otimes {}^{\kappa_{2n}}\calD_{\mu^v}].$ We have 
\begin{multline}
\label{eqn:res-G-0}
{\rm Res}_{G_0(\R)^\circ}\left(\aInd_{P_0(\R)}^{G_0(\R)}(\chi_v \otimes \sigma_v)\right) \ = \\ 
\aInd_{P_0(\R) \cap G_0(\R)^\circ }^{G_0(\R)^\circ} \left( [\chi_v \otimes \calD_{\mu^v}]\right) \ \oplus \ 
\aInd_{P_0(\R) \cap G_0(\R)^\circ }^{G_0(\R)^\circ} \left( [\chi_v \otimes {}^{\kappa_{2n}}\calD_{\mu^v}]\right). 
\end{multline}

\smallskip

For the coefficient system, we have
$$
{\rm Res}_{G_0(\R)^\circ}(\M_{\lambda^v}) \ = \ \cN_{\lambda^v} \oplus {}^{\kappa_{2n+2}}\cN_{\lambda^v}, 
$$
where $\cN_{\lambda^v}$ is the irreducible representation of $\SO(n+1,n+1)(\R)$ 
with highest weight 
$\lambda^v.$ For the coefficient systems for various relative Lie algebra cohomology groups considered here, the reader should always bear in mind Wigner's lemma which will clarify any possible confusion. 

\medskip

We will fix basis elements for the various one-dimensional spaces: 
\begin{enumerate}
\item[(i)] For $H^{q_0}({}^\circ\fm_v, \fk_{{}^\circ \! M_v};\,  D_{\mu^v} \otimes \mathcal \cN_{\mu^v} ),$ fix ${\sf z}_v$ as a generator. (See Sect.\,\ref{sec:weights-DS-cohomology}). 
\smallskip
\item[(ii)] For $H^{q_m}({}^\circ\fm(\R), K_{{}^\circ \! M(\R)}^\circ; \, D_{\mu}  \otimes \mathcal \cN_{\mu} ),$ fix ${\sf z} = \otimes_v {\sf z}_v,$  
via K\"unneth theorem.
\smallskip
\item[(iii)] For $\wedge^0(\fa(\R)^*)$, fix a generator ${\sf a}_0$. 
\smallskip
\item[(iv)] For $H^{q_b}(\fm(\R), K_{M(\R)}^\circ; \,  (\chi_\infty \otimes D_{\mu}) \otimes \mathcal \cN_{de_0+\mu} ),$ fix ${\sf a}_0 \otimes {\sf z}.$
\end{enumerate}

From the discussion above, we get: 
$$
H^{\q_b}\left(\g(\R), \overline{K}_\infty; \, \aInd_{P(\R)}^{G(\R)}(\chi_\infty \otimes \sigma_\infty) \otimes \M_{\lambda}\right)
$$
is one-dimensional that has for its generator, after applying \cite[Thm.\,III.3.3]{borel-wallach} to the cohomologies of the two summands
of \eqref{eqn:res-G-0}, the element: 
\begin{equation}
\label{eqn:basis-one-dim-space}
\left( {\sf a}_0 \otimes {\sf z} +  
\varepsilon_\chi \, {\sf a}_0 \otimes {}^{s_{2n}} {\sf z}  \right) \ + \ 
{}^{s_{2n}\kappa_{2n}} 
\left( {\sf a}_0 \otimes {\sf z} +  
\varepsilon_\chi \, {\sf a}_0 \otimes {}^{s_{2n}} {\sf z} \right), 
\end{equation}
where, ${}^{s_{2n}} {\sf z} = \otimes_v {}^{s_{2n}} {\sf z}_v$ and similarly for ${}^{\kappa_{2n}} {\sf z}$ and ${}^{s_{2n}\kappa_{2n}} {\sf z}.$ 
Note that the first term in the right hand side, i.e., $\left( {\sf a}_0 \otimes {\sf z} + \varepsilon_\chi \, {\sf a}_0 \otimes {}^{s_{2n}} {\sf z}  \right),$  
is a generator for the cohomology of first summand of \eqref{eqn:res-G-0}; similarly, 
the second term of \eqref{eqn:basis-one-dim-space} is a generator for the cohomology of the second summand of \eqref{eqn:res-G-0}. 
The generator in \eqref{eqn:basis-one-dim-space} may also be written as: 
\begin{equation}
\label{eqn:basis-one-dim-space-rewritten}
{\sf a}_0 \otimes {\sf z} \ + \ 
\varepsilon_\chi \, {\sf a}_0 \otimes {}^{s_{2n}} {\sf z}  \ + \ 
{\sf a}_0 \otimes {}^{s_{2n}\kappa_{2n}} {\sf z} \ + \  
\varepsilon_\chi \, {\sf a}_0 \otimes {}^{\kappa_{2n}} {\sf z}, 
\end{equation}
from which it is clear that either $s_{2n+2}$ or $\kappa_{2n+2}$ acts via $\varepsilon_\chi$. 
This concludes the proof of Prop.\,\ref{prop:g-v-bar-k-v-coh}. 
\end{proof}

\medskip
\begin{remark}\label{rem:top-degree-injection}
{\rm 
The cohomology group as in Prop.\,\ref{prop:g-v-bar-k-v-coh} but in degree $\q_t$ defined in \eqref{eqn:qb-and-qt-for-G} is not one-dimensional if $F \neq \Q$ as we now explain. The argument
goes exactly as in degree $\q_b$, but in degree $\q_t$ we will end up with the direct sum of the ${\sf r}_F$-dimensional space
$$
H^{q_m}({}^\circ\fm(\R), K_{{}^\circ \! M(\R)}^\circ; \, D_{\mu}  \otimes \mathcal \cN_{\mu} ) \otimes \wedge^{{\sf r}_F-1}(\fa(\R)^*), 
$$
and its appropriate conjugates by $s_{2n+2}$ and $\kappa_{2n+2}.$ 
This will be especially relevant when we consider Poincar\'e duality pairing 
between 
$$H^{\q_b}(\partial\SGK, \tM_{\lambda, E}) \ \ {\rm and} \ \ H^{\q_t}(\partial\SGK, \tM_{\lambda^{\sf v}, E})
$$ 
as in the proof of the Manin--Drinfeld
principle in Thm.\,\ref{thm:Manin-Drinfeld}. 
}
\end{remark}

\medskip
\subsubsection{\bf Intertwining operator in cohomology}
\label{sec:int-op-coh}

Let $K_{v,0} = K_0 = \rO(n+1) \times \rO(n+1),$ as an algebraic group over $\Q,$ where each factor $\rO(n+1)$ is the orthogonal group preserving 
the quadratic form $x_1^2+\dots+x_{n+1}^2 = 1$.  We think of $K_{0}$ as embedded in $G_0$ via the map $h$. 
Then, of course, $K_{v,0}(\R) = \rO(n+1)(\R) \times \rO(n+1)(\R)$ is the compact Lie group which is embedded, 
via the map $h,$ as a maximal compact subgroup of $G_0(F_v) = \rO(n+1,n+1)(\R).$  
Recall the vectors ${\bf f}_0 \in \sI_v$ and $\tilde{\bf f}_0 \in \tilde\sI_v,$
which are the normalised highest weight vectors of the lowest $K_v^\circ$-type. Put 
$$
\sI^0_v \ \mbox{(resp., $\tilde{\sI}^0_v$)} \ := \ \mbox{$\Q$-span of the $G_0(\Q)$-orbit of ${\bf f}_0$ (resp., $\tilde{\bf f}_0$)}.
$$
Then, $\sI^0_v \otimes_\Q \C  =  \sI_v,$ and $\tilde\sI^0_v \otimes_\Q \C  =  \tilde\sI_v.$ Furthermore, recall that $T_v({\bf f}_0) = c_v \tilde{\bf f}_0$, 
where $T_v = T_{\st,v}(-n, w_{P})$ and $c_v = c_v(  \chi_v, \sigma_v)$ as in Prop.\,\ref{prop:c-chi-sigma-archimedean}. Define: 
$$
T_v^0 \ := \ c_v^{-1} \, T_v.
$$
Then, $T_v^0 : \sI^0_v \to \tilde{\sI}^0_v$ is an isomorphism of $G_0(\Q)$-modules; and moreover, any such isomorphism is unique up to homotheties 
by $\Q^\times.$ 

\medskip
Next, let $\cM_{\lambda^v}^0$ be a $\Q$-structure on the finite-dimensional representation $\cM_{\lambda^v}$ which is stable under $G_0(\Q).$ Let $\g_{0}$ be the Lie algebra of $G_0$; the split orthogonal group $G_0$ is defined over 
$\Q$, and so $\g_0$ is defined over $\Q$. 
Let $\overline{K}_{v,0}$ be the subgroup of $\overline{K}_v$ generated by 
$K_{v,0}(\Q)$ and the element $s_{2n+2}\kappa_{2n+2}$. Define $\q_{b,v} = n^2/2 + n$; then, $\sum_{v \in S_\infty} \q_{b,v} = \q_b.$ 
Consider the relative Lie algebra cohomology group 
$H^{\q_{b,v}}(\g_{0}, \overline{K}_{v,0}; \, \mathscr \sI_v^0 \otimes \cM_{\lambda^v}^0),$ which is a $\Q$-vector space. We have:
$$
H^{\q_{b,v}}(\g_{0,v}, \overline{K}_v; \, \mathscr \sI_v\otimes \cM_{\lambda^v}) \ = \
H^{\q_{b,v}}(\g_{0}, \overline{K}_{v,0}; \, \mathscr \sI_v^0 \otimes \cM_{\lambda^v}^0) \otimes_\Q \C.
$$
{\it A fortiori}, $H^{\q_{b,v}}(\g_{0}, \overline{K}_{0}; \mathscr \sI_v^0 \otimes \cM_{\lambda^v}^0),$ is a one-dimensional $\Q$-vector space. Actually, the 
proof of Prop.\,\ref{prop:g-v-bar-k-v-coh} is purely algebraic and works over $\Q$ directly giving one-dimensionality. We fix a $\Q$-basis element 
$[\sI_v^0]$ which has the following form:  
$$
[\sI^0_v] \ = \ \sum \limits_{\underline{i}, \alpha} X^{\ast}_{\underline{i}} \otimes \phi^{0}_{\underline{i}, \alpha} \otimes w_{\alpha}, 
$$ 
where $\{X_{\underline{i}}^{\ast}\}_{\underline{i}}$ is a basis of $(\g_{0} / \mathfrak k_0)^{\ast}$; if $\underline{i} = (i_1,\dots, i_{\q_{b,v}})$ is 
a $\q_{b,v}$-tuple of indices then we put 
$X^{\ast}_{\underline{i}} = X^{\ast}_{i_1}\wedge \cdots \wedge X^{\ast}_{i_{\q_{b,v}}};$ $\{ w_{\alpha} \}_\alpha $ is a basis of $ \cM_{\lambda^v, 0};$ 
and $\phi^{0}_{\underline{i}, \alpha} \in \sI^0_{v}$. 
Similarly, we get a basis $[\tilde{\sI}^0_v]$ of $H^{\q_{b,v}}(\g_{0,v}, \overline{K}_{v}; \tilde{\sI}_v \otimes \cM_{\lambda^v})$ that is, up to multiplying by a nonzero
rational number, of the form
$$
[\tilde{\sI}^0_v] \ = \ \sum \limits_{\underline{i}, \alpha} X^{\ast}_{\underline{i}} \otimes T_v^0(\phi^{0}_{\underline{i}, \alpha}) \otimes w_{\alpha}. 
$$ 
The intertwining operator $T_v$, which is an isomorphism between the induced representations, induces a linear isomorphism
$$
T_v^\bullet : H^{\q_{b,v}}(\g_{0,v}, \overline{K}_v; \, \sI_v\otimes \cM_{\lambda^v}) \ \to \ 
H^{\q_{b,v}}(\g_{0,v}, \overline{K}_v; \, \tilde{\sI}_v\otimes \cM_{\lambda^v}), 
$$
and from the above discussion we have 
\begin{equation}
\label{eqn:intertwining-operator-in-cohomology}
T_v^\bullet([\sI^0_v]) \ \approx_{\Q^\times} \  c_v \, [\tilde{\sI}^0_v]. 
\end{equation}
Taking the tensor product over $v \in S_{\infty},$  
we have a $\Q$-basis element   
$[\sI^0] := \mathbin{\mathop{\otimes}\displaylimits_{v \in S_\infty}  }[\sI^0_v]$ of the one-dimensional space
$$
H^{\q_b}\left(\g(\R), \overline{K}_\infty; \, \aInd_{P(\R)}^{G(\R)}(\chi_\infty \otimes \sigma_\infty) \otimes \M_{\lambda}\right). 
$$
Similarly, 
$[\tilde{\sI}^0] := \mathbin{\mathop{\otimes}\displaylimits_{v \in S_\infty}  }[\tilde{\sI}^0_v]$ generates the one-dimensional
$$
H^{\q_b}\left(\g(\R), \overline{K}_\infty; \, \aInd_{P(\R)}^{G(\R)}(\chi_\infty^{-1}[2n] \otimes {}^{\kappa_{2n}}\sigma_\infty) \otimes \M_{\lambda}\right). 
$$
From \eqref{eqn:intertwining-operator-in-cohomology} and Prop.\,\ref{prop:c-chi-sigma-archimedean} we get the following proposition:

\medskip
\begin{proposition} 
\label{prop:intertwining-operator-in-cohomology}
The local intertwining operators $T_{\st,v}(-n, w_{P})$, for $v \in S_{\infty},$ give an archimedean intertwining operator 
$T_{\st, \infty}: =  \mathbin{\mathop{\otimes}\displaylimits_{v \in S_\infty}} T_{\st,v},$ that induces a linear isomorphism between 
the one-dimensional cohomology spaces: 
\begin{multline*}
T_{\st, \infty}^{\bullet}: 
H^{\q_b}\left(\g(\R), \overline{K}_\infty; \, \aInd_{P(\R)}^{G(\R)}(\chi_\infty \otimes \sigma_\infty) \otimes \M_{\lambda}\right) \ \to \\  
H^{\q_b}\left(\g(\R), \overline{K}_\infty; \, \aInd_{P(\R)}^{G(\R)}(\chi_\infty^{-1}[2n] \otimes {}^{\kappa_{2n}}\sigma_\infty) \otimes \M_{\lambda}\right). 
\end{multline*}
Using the basis elements on either side it is given by:
$$
T_{\st, \infty}^{\bullet} ([\sI^0] ) \ \approx_{\Q^\times} \ 
\frac{L_\infty(-n, \chi \times \sigma) }{L_\infty(1-n, \chi \times \sigma)} \,
[\tilde{\sI}^0], 
$$
where, as before, $\approx_{\Q^\times}$ means equality up to a nonzero rational number. 
\end{proposition}

\medskip
\subsection{Local intertwining operators at a non-archimedean place: general case}
\label{sec:p-adic-ramified}

We discuss certain arithmetic properties of a local standard intertwining operator at a nonarchimedean place $v$; by the `general case' we mean that no distinction
is made whether $v$ is unramified or not, although the assertions below are used especially when $v$ is ramified. This discussion is almost exactly as in 
\cite[7.3.2.1]{harder-raghuram} and so we will be brief here just pointing to the important steps to take, and especially to one significant difference for orthogonal groups that requires the local hypothesis in $(ii)$ of Thm.\,\ref{thm:main:introduction}. 

\medskip

The normalised standard intertwining operator is defined as: 
\begin{equation}
\label{eqn:norm-int-op-ram-v}
T_{\norm, v}(s) \ = \ 
\left(\frac{L(s, {}^\iota\chi_v \times {}^\iota\sigma_v)}{L(1+s, {}^\iota\chi_v \times {}^\iota\sigma_v)}\right)^{-1}
T_{\st, v}(s).
\end{equation}
In the Langlands--Shahidi machinery it is usual for the normalising factor to include the $\varepsilon$-factor also (see, for example, \cite[(9.1.2)]{shahidi-book}), 
but for us the $\varepsilon$-factor plays no role here, and so we use only the relevant local $L$-factors.

\medskip
\begin{proposition}
\label{prop:norm-st-int-op}
The normalised operator $T_{\norm, v}(s)$ is holomorphic and nonvanishing for $\Re(s) \geq d + \tfrac12$. 
If, furthermore, ${}^\iota\sigma_v$ is tempered then $T_{\norm, v}(s)$ is holomorphic and nonvanishing for $\Re(s) \geq d$. 
\end{proposition}
The first assertion is due to Wook Kim (\cite[Sect.\,5]{kim-wook}) and the second assertion in the tempered case is due to Henry Kim 
(see \cite[Prop.\,12.3]{kim-notes} and the references therein). The condition $\Re(s) \geq d + \tfrac12$, at our point of evaluation $s = -n$ 
translates to $-n-d \geq 1/2$ and by integrality to $-n-d \geq 1;$ and $\Re(s) \geq d$ for $s = -n$ means $-n-d \geq 0.$ 
In the context of general linear groups 
\cite[7.3.2.1]{harder-raghuram}, one has a stronger statement due to M\oe glin and Waldspurger \cite[Prop.\,I.10]{moeglin-waldspurger} (see also 
\cite[Prop.\,12.4]{kim-notes}) but this does not seem to be 
available for orthogonal groups. For us, the worst case scenario when $-n-d = 0,$ necessitates the local tempered assumption. Therefore, we assume henceforth that 
\begin{multline}
\label{eqn:local-assumption}
\mbox{either `$-n-d \geq 1$' or}  \\ 
\mbox{`$-n-d = 0$ and ${}^\iota\sigma$ is locally tempered at finite places,'}
\end{multline}
to guarantee $T_{\norm, v}(-n)$ is finite and nonvanishing. 

\medskip
The rest of the discussion for a local nonarchimedean intertwining operator is analogous to \cite[7.3.2.1]{harder-raghuram}, and in fact 
works in much greater generality; see \cite[Thm.\,4.3.1]{raghuram-p-adic}. 
One appeals to certain rationality results: (i) rationality for intertwining operators as in 
Waldspurger\,\cite[Thm.\,VI.1.1]{waldspurger}, and (ii) rationality for local factors as in Shahidi\,\cite[Thm.\,8.3.2, (2)]{shahidi-book}, and one deduces that the normalised standard intertwining operator at the point of evaluation, under the assumptions 
delineated in \eqref{eqn:local-assumption}, is the base-change via $\iota$ of an arithmetic intertwining operator, i.e, one has: 
\begin{equation}
\label{eqn:arith-int-op}
T_{\norm, v}(-n) \ = \ T_{\arith, v} \otimes_{E, \iota} 1_\C, 
\end{equation}
where $T_{\arith, v}$ is an intertwining operator for modules over $E$: 
$$
T_{\arith, v} \ : \ \aInd_{P_{0,v}}^{G_{0,v}}(\chi_v \times \sigma_v)^{K_v} \ \to \ \aInd_{P_{0,v}}^{G_{0,v}}(\chi_v^{-1}[2n] \times {}^{\kappa_{2n}}\sigma_v)^{K_v}. 
$$

\medskip
\section{The Manin--Drinfeld principle}

This principle says that inside the cohomology of the boundary there are certain Hecke stable subspaces that split off from the total boundary cohomology. We will first of all need to split strongly inner-cohomology for the Levi $M_P$ from its total cohomology.

\medskip
\subsection{A splitting principle for strongly-inner cohomology}
\label{sec:strongly-inner-split} 
Let $\mu \in X^*({}^\circ T \times E)$ be a dominant integral weight for ${}^\circ\!M$ with $\mu_{\rm min} \geq 1.$ 
Consider the cohomology of ${}^\circ\!M$ for the weight $\mu$ and a level structure ${}^\circ C_f \subset {}^\circ\!M(\A_f).$ 
For a finite set $S$ of places including $S_\infty$  and all ramified places, recall that 
$\HH^{{}^\circ\!M, S}$ is a central subalgebra of $\HH^{{}^\circ\!M}_{{}^\circ C_f}$ by taking the restricted tensor product of all the local spherical Hecke algebras outside of 
the places in $S$.

 \begin{proposition}
 \label{prop:splitting-inner-cohomology}
 Let $\mu$ be a dominant integral weight for ${}^\circ\!M.$ 
 There exists a ${\mathcal H}^{{}^\circ\!M, S}$-submodule 
 $$
 H^{\bullet}_{c-!!} (\mathcal S^{^\circ M}_{^\circ C_{f}}, \widetilde{\mathcal M}_{\mu,E}) \ \subset \  
 H^{\bullet} (\mathcal S^{^\circ M}_{^\circ C_{f}}, \widetilde{\mathcal M}_{\mu,E})
 $$ 
 which is complementary to strongly-inner cohomology, i.e., we have as ${\mathcal H}^{{}^\circ\!M, S}$-modules: 
 $$
 H^{\bullet} (\mathcal S^{^\circ M}_{^\circ C_{f}}, \widetilde{\mathcal M}_{\mu,E}) \ = \ 
 H^{\bullet}_{!!} (\mathcal S^{^\circ M}_{^\circ C_{f}}, \widetilde{\mathcal M}_{\mu,E}) \ \oplus \ 
 H^{\bullet}_{c-!!} (\mathcal S^{^\circ  M}_{^\circ C_{f}}, \widetilde{\mathcal M}_{\mu,E}).
 $$ 
 \end{proposition}
 
 \begin{proof}
 For a commutative $\Q$-algebra $\cA$, and an $E$-vector space $V$ which is also a module for $\cA$, we let 
 $\Spec_\cA(V)$ stand for the set of $\Q$-algebra homomorphisms $\chi : \cA \to E$ that appear as generalised 
 eigencharacters for the $\cA$-action on $V$. With this notation, to show that strongly-inner cohomology splits, we
 need to show that
 \begin{multline*}
 {\rm Spec}_{{\mathcal H}^{{}^\circ\!M, S}}(H^{\bullet}_{!!} (\mathcal S^{^\circ \!M}_{^\circ C_{f}}, \widetilde{\mathcal M}_{\mu,E}) ) 
 \ \cap \ \\ 
 {\rm Spec}_{{\mathcal H}^{{}^\circ\!M, S}}
 \left(H^{\bullet}(\mathcal S^{^\circ \!M}_{^\circ C_{f}}, \widetilde{\mathcal M}_{\mu,E}) / 
 H^{\bullet}_{!!} (\mathcal S^{^\circ \!M}_{^\circ C_{f}}, \widetilde{\mathcal M}_{\mu,E}) \right) 
  \ =\ \emptyset. 
 \end{multline*}
 If some $\tau_f \in {\rm Coh}_{!!}(^\circ M, \mu)$ appears in the intersection then after a base change by $\iota: E \rightarrow \C$, the finite part ${}^\iota \tau_f$ 
 of the cuspidal automorphic representation ${}^\iota\tau$ 
 appears in 
 $H^{\bullet}(\mathcal S^{^\circ M}_{^\circ C_{f}}, \widetilde{\mathcal M}_{^\iota\mu,\C}) / H^{\bullet}_{!!} (\mathcal S^{^\circ M}_{^\circ C_{f}}, \widetilde{\mathcal M}_{^\iota\mu,\C}).$ 
Hence also in 
$$
H^{\bullet} (\mathcal S^{^\circ M}_{^\circ C_{f}}, \widetilde{\mathcal M}_{^\iota\mu,\C}) / H^{\bullet}_{!} (\mathcal S^{^\circ M}_{^\circ C_{f}}, \widetilde{\mathcal M}_{^\iota\mu,\C})  = H^{\bullet}(\partial S^{^\circ M}_{^\circ C_{f}},  \widetilde{\mathcal M}_{^\iota \mu,\C}).
$$ 
 Therefore $^\iota \tau$ is almost everywhere (i.e., outside of $S$) equivalent to a parabolically induced representation, because any simple Hecke-subquotient of   
boundary cohomology is of this form as they build up the $E_1^{pq}$ terms of the spectral sequence converging to total boundary cohomology. But then this implies
that $^\iota \tau$ is a CAP representation (CAP = `cuspidal associated to a parabolic'), which is not possible as one of the defining conditions for 
$\tau$ to be strongly-inner is that the Arthur parameter 
of $^\iota\tau$ is cuspidal (see, for example, \cite{ginzburg-jiang-soudry}). 
\end{proof}
  
\medskip
For the torus $A$ and for $d \in \Z$, 
we define $H^q_{!!}(\mathcal S^{A}_{B_{f}}, \widetilde \cM_{d,E}): = H^q(\mathcal S^{A}_{B_{f}}, \widetilde \cM_{d,E})$. It is nonzero and in fact one-dimensional only 
in the extremal degrees $q = 0 $ or $ q = {\sf r}_F-1$. Define 
$$ 
H^\bullet_{!!}(\mathcal S^{M}_{C_{f}}, \widetilde \cM_{de_0+\mu,E}) : = 
H^{\bullet}_{!!}(\mathcal S^{A}_{B_{f}}, \widetilde \cM_{d,E})  \otimes H^{\bullet}_{!!}(\mathcal S^{^{\circ}\!M}_{^\circ C_{f}}, \widetilde \cM_{\mu,E}).
$$ 
Similar to the splitting in Prop.\,\ref{prop:splitting-inner-cohomology}, we get that $H^{\bullet}_{!!}(\mathcal S^{M}_{C_{f}}, \widetilde \cM_{de_0+\mu,E})$ splits from the total cohomology $H^{\bullet}(\mathcal S^{M}_{C_{f}}, \widetilde{\mathcal M}_{de_0 +\mu,E}) $ as an ${\mathcal H}^{M, S}$-module.

\medskip
For general linear groups the complementary module for strongly-inner cohomology is Eisenstein type (see \cite[5.1.2]{harder-raghuram}). However, for orthogonal groups, strongly-inner cohomology after a base-change to $\C$, need not capture all of cuspidal cohomology, but captures that part of cuspidal cohomology that is somehow intrinsic to the orthogonal group--due to the requirement in Def.\,\ref{def:strongly-inner-spectrum} 
that the Arthur parameter is cuspidal. 

\medskip
\subsection{Manin--Drinfeld principle}
\label{sec:manin-drinfeld} 

\subsubsection{\bf Induced modules in boundary cohomology}  
\label{sec:ind-mod-bdry-coh}
Let $\mu \in X^*({}^\circ T \times E)$ be a dominant integral weight for ${}^\circ\!M$ with $\mu_{\rm min} \geq 1,$ and $\sigma_f \in {\rm Coh}_{!!}(^{\circ}M, \mu).$ 
Let $\chi$ be an algebraic Hecke character of $F$ taking values in $E,$ and $d$ be the integer such that for any $\iota : E \to \C$, we get a continuous character 
${}^\iota\chi : F^\times\backslash \A_F^\times \to \C^\times$ of the form ${}^\iota\chi = {}^\iota\chi^\circ \otimes |\ |^{-d},$ and ${}^\iota\chi^\circ$ is of finite-order.

Certain induced modules will appear frequently for which it will help to have a simplified notation. In what follows we will freely use the facts that $\sigma_f^{\sf v} = \sigma_f$ and ${}^{\kappa_{2n}}\sigma_f = \sigma_f.$ First of all, we let 
$$
H^{q_b}_{!!} \left(\cS^{M}, \cM_{w \cdot \lambda, E} \right)^{\overline{\pi}_0(M(\R))}[\chi_f \times \sigma_f]
$$
stand for the $\chi_f \times \sigma_f$ isotypic component in $H^{q_b}_{!!} \left(\cS^{M}, \cM_{w \cdot \lambda, E} \right)$ as an  
$\HH^{M,S}$-module on which $\overline{\pi}_0(M(\R))$ acts trivially. 
We assume the combinatorial lemma holds for the data $(d, \mu),$ and in particular, the properties of the Kostant representatives $w, w', w^{\sf v},$ and 
$w^{\sf v}{}'$ in Prop.\,\ref{prop:w-w'-etc} will be relevant below. Define:
\begin{multline}
\label{eqn:I-b-S-bdry-coh}
I_b^S(\chi_f, \sigma_f)_w  \ := \\   
\left(\aInd^{\pi_0(G(\R)) \times G(\A_f)}_{\pi_0(P(\R)) \times P(\A_f)} 
\left( H^{q_b}_{!!} \left(\cS^{M}, \cM_{w \cdot \lambda, E} \right)^{\overline{\pi}_0(M(\R))} [\chi_f \times \sigma_f] \right)\ \right)^{K_{f}}, 
\end{multline} 
Similarly, define: 
\begin{multline*}
I_b^S(\chi_f^{-1}[2n], \sigma_f)_{w'}  \ := \\ 
\left(\aInd^{\pi_0(G(\R)) \times G(\A_f)}_{\pi_0(P(\R)) \times P(\A_f)} 
\left(H^{q_b}_{!!} \left(\cS^M, \cM_{w' \cdot \lambda, E}\right)^{\overline{\pi}_0(M(\R))} [\chi^{-1}_f[2n] \times {}^{\kappa_{2n}}\sigma_f]\right) \right)^{K_f}.
\end{multline*}

\noindent
Both these are modules for $\HH^{G,S}$ that appear in $H^{\q_b}(\partial S^G_{K_{f}}, \cM_{\lambda, E})$; this needs a word of explanation. Using 
$\theta_M^* : H^\bullet(\cS^{M}_{C_f}, \tM) \ \to \ H^\bullet(\tilde{\cS}^{M}_{C_f}, \tM)$ (see \eqref{eqn:coh-S-M-coh-tilde-S-M}), and passing to 
the limit over all $C_f$, we embed 
$$
H^{q_b}_{!!} \left(\cS^{M}, \cM_{w \cdot \lambda, E} \right)^{\overline{\pi}_0(M(\R))} [\chi_f \times \sigma_f]
\ \hookrightarrow \ 
H^{q_b}_{!!} \left(\tilde{\cS}^{M}, \cM_{w \cdot \lambda, E} \right)^{\overline{\pi}_0(M(\R))} [\chi_f \times \sigma_f], 
$$
which is an isomorphism, since on the $A$ part we have cohomology in degree $0$, i.e., 
$\wedge^0(\fa/\fs)^* \cong \wedge^0(\fa)^*.$ From \eqref{eqn:coh-tilde-S-M-coh-ul-S-M} we get 
$$
H^{q_b}_{!!} \left(\tilde{\cS}^{M}, \cM_{w \cdot \lambda, E} \right)^{\overline{\pi}_0(M(\R))} [\chi_f \times \sigma_f] \ = \ 
H^{q_b}_{!!} \left(\underline{\cS}^{M}, \cM_{w \cdot \lambda, E} \right)[\chi_f \times \sigma_f], 
$$
which upon inducing, i.e., after applying $\aInd^{\pi_0(G(\R)) \times G(\A_f)}_{\pi_0(P(\R)) \times P(\A_f)}$, as in Prop.\,\ref{prop:bdry-coh-2}, 
is a summand in $H^{\q_b}(\partial_P\SG, \tM_{\lambda, E}),$ and taking $K_f$-invariants lands us in $H^{\q_b}(\partial S^G_{K_{f}}, \cM_{\lambda, E}).$ 

\medskip
There is a dual version of the induced modules using `top-degree'. 
\begin{multline*}
I_t^S(\chi_f^{-1}[2n], \sigma_f)_{w^{\sf v}}  \ := \\ 
\left(\aInd^{\pi_0(G(\R)) \times G(\A_f)}_{\pi_0(P(\R)) \times P(\A_f)} 
\left(H^{q_t}_{!!} \left(\cS^M, \cM_{w^{\sf v}\cdot \lambda^{\sf v}, E}\right)^{\overline{\pi}_0(M(\R))} [\chi^{-1}_f[2n] \times \sigma_f^{\sf v}]\right) \right)^{K_f}, 
\end{multline*}
\begin{multline*}
I_t^S(\chi_f, \sigma_f)_{w^{\sf v}{}'}  \ := \\ 
\left(\aInd^{\pi_0(G(\R)) \times G(\A_f)}_{\pi_0(P(\R)) \times P(\A_f)} 
\left( H^{q_t}_{!!} \left(\cS^{M}, \cM_{w^{\sf v}{}' \cdot \lambda^{\sf v}, E} \right)^{\overline{\pi}_0(M(\R))}  [\chi_f \times {}^{\kappa_{2n}}\sigma_f^{\sf v}] \right)\ \right)^{K_{f}}. 
\end{multline*}

\noindent
Both these are modules for $\HH^{G,S}$ that appear in $H^{\q_t}(\partial S^G_{K_{f}}, \cM_{\lambda^{\sf v}, E}).$ In this situation, it is similar to the above discussion on bottom-degree with one exception: the embedding 
$$
H^{q_t}_{!!} \left(\cS^{M}, \cM_{w \cdot \lambda, E} \right)^{\overline{\pi}_0(M(\R))} [\chi_f \times \sigma_f]
\ \hookrightarrow \ 
H^{q_t}_{!!} \left(\tilde{\cS}^{M}, \cM_{w \cdot \lambda, E} \right)^{\overline{\pi}_0(M(\R))} [\chi_f \times \sigma_f], 
$$
is not an isomorphism if $F \neq \Q$, because in the general case we only have an injection $\wedge^{{\sf r}_F-1} (\fa/\fs)^* \hookrightarrow \wedge^{{\sf r}_F-1}(\fa)^*$; 
see Rem.\,\ref{rem:top-degree-injection}.

\medskip

\subsubsection{\bf Statement and proof of the Manin--Drinfeld principle}
\label{sec:para-manin-drinfeld}
The main theorem on isotypic components in boundary cohomology, called the Manin--Drinfeld principle, is the following theorem, that generalizes
\cite[Thm.\,5.12]{harder-raghuram} from general linear groups dealt therein to our context of orthogonal groups. We will henceforth assume 
that the level structure $K_f$ is a neat open-compact subgroup of $G(\A_f)$ of 
the form $K_f = \prod_{v \notin S_{\infty}} K_v,$ where $K_v = G(\mathcal O_v)$ for almost all $v,$ and for the finite set of ramified places (i.e., places where it is a proper subgroup of $G(\mathcal O_v)$) we 
assume that $K_v$ is a principal congruence subgroup of $G(\mathcal O_v)$; this is compatible with our previous requirement 
that $C_f = \kappa_P(K_f \cap P(\A_f))$ is a product $B_f \times {}^\circ C_f$. 

\medskip
\begin{theorem}
\label{thm:Manin-Drinfeld} 
\begin{enumerate}
\item[(i)] The $\pi_0(G(\R)) \times \HH^{G,S}$-modules  
$$
I_b^S(\chi_f, \sigma_f)_w, \ \ 
I_b^S(\chi_f^{-1}[2n], \sigma_f)_{w'}, \ \ 
I_t^S(\chi_f^{-1}[2n], \sigma_f)_{w^{\sf v}}, \ {\rm and} \ 
I_t^S(\chi_f, \sigma_f)_{w^{\sf v}{}'} 
$$
are finite-dimensional $E$-vector spaces and all of which have the same $E$-dimension, denoted ${\sf k}.$

\medskip
\item[(ii)] The direct sum
$$
I_b^S(\chi_f, \sigma_f)_w  \ \oplus \  
I_b^S(\chi_f^{-1}[2n], \sigma_f)_{w'}
$$
is a $2{\sf k}$-dimensional isotypic subspace of $H^{\q_b} (\partial \cS^G_{K_{f}}, \cM_{\lambda,E}).$ 
Furthermore, there exists a 
$\pi_0(G(\R)) \times \HH^{G,S}$-equivariant projection: 
$$
\RR^{b}_{\chi_f, \sigma_f} \ : \ H^{q_b} (\partial \cS^G_{K_{f}}, \cM_{\lambda,E}) \longrightarrow 
I_b^S(\chi_f, \sigma_f)_w  \ \oplus \  
I_b^S(\chi_f^{-1}[2n], \sigma_f)_{w'}. 
$$ 
 
\medskip
\item[(iii)] 
The direct sum
$$
I_t^S(\chi_f^{-1}[2n], \sigma_f)_{w^{\sf v}}  \ \oplus \  
I_t^S(\chi_f, \sigma_f)_{w^{\sf v}{}'} 
$$
is a $2{\sf k}$-dimensional isotypic subspace of $H^{\q_t} (\partial \cS^G_{K_{f}}, \cM_{\lambda^{\sf v},E}).$ 
Furthermore, there exists a 
$\pi_0(G(\R)) \times \HH^{G,S}$-equivariant projection: 
$$
\RR^{t}_{\chi_f, \sigma_f} \ : \ H^{\q_t} (\partial \cS^G_{K_{f}}, \cM_{\lambda,E}) \ \longrightarrow \ 
I_t^S(\chi_f^{-1}[2n], \sigma_f)_{w^{\sf v}}  \ \oplus \  
I_t^S(\chi_f, \sigma_f)_{w^{\sf v}{}'}. 
$$
\end{enumerate}
\end{theorem}

\medskip
\begin{proof}
The proof of $(i)$ is purely local and is delineated as the following lemma: 

\begin{lemma}
\label{lem:equal-dimensions}
Recall the hypotheses on $K_f \subseteq G(\A_f)$ as in the first paragraph of \ref{sec:para-manin-drinfeld}. We have:
$$ 
{\rm dim}\left(\aInd^{G_f}_{P_f} (\chi_f \times \sigma_f)^{K_f}\right) \ = \ 
{\rm dim}\left(\aInd^{G_f}_{P_f}(\chi_f^{-1}[2n] \times {}^{\kappa_{2n}}\sigma_f)^{K_f}\right). 
$$
\end{lemma}

\begin{proof}
This is a purely local statement since the groups and the representations factor over the set of all finite places. For brevity, we will suppress the subscript $v$ from $G_v$, $P_v$, $\chi_v$, $\sigma_v$, $K_v$, etc. We need to prove
$$ 
{\rm dim}\left(\aInd^{G}_{P}(\chi \times \sigma)^K\right) \ = \  
{\rm dim} \left(\aInd^{G}_{P}(\chi^{-1}[2n] \times {}^{\kappa_{2n}}\sigma)^K\right). 
$$
For $v$ unramified, both sides are $1$. The statement needs a proof only for ramified $v$ (although for the proof we do not need to make any distinction). 

\smallskip
Recall Frobenius reciprocity (\cite[Prop.\,2.29]{bernstein-zelevinskii}): If $\pi$ is an admissible representation of a reductive $p$-adic group $G$ and $K$ is an open compact subgroup then 
$$
{\rm Hom}_{K}({\mathbf 1}, \pi) \ \cong \ {\rm Hom}_G({\rm ind}^G_K({\mathbf 1}), \pi), 
$$
where ${\rm ind}^G_K({\mathbf 1})$ is the compact induction of the trivial representation of $K$ to $G$ whose representation space os 
$\mathcal C_c^{\infty} (K \backslash G)$. Secondly, let us recall the following version of Mackey theory (see \cite{kutzko}), which when applied to our situation reads:
\begin{equation} 
\label{eqn:mackey-kutzko}
 {\rm Hom}_{G}\left({\rm ind}^G_K({\mathbf 1}), \, \aInd_P^G(\tau) \right) \ = \ 
\bigoplus_{x \in P \backslash G / K} {\rm Hom}_{K \, \cap\, x^{-1}Px} (\mathbf 1, \tau^x),
\end{equation} 
where $\tau$ is a tentative notation for either $\chi \times \sigma$ or $\chi^{-1}[2n] \times {}^{\kappa_{2n}}\sigma.$ The summand on the right hand side indexed by $x$ can be conjugated by $x$ and we may rewrite this as:
\begin{multline} 
\label{eqn:mackey-kutzko-2} 
 {\rm Hom}_{G}\left({\rm ind}^G_K({\mathbf 1}), \, \aInd^P_G(\tau) \right)  
\ = \\ 
\bigoplus_{x \in P \backslash G / K} {\rm Hom}_{xK x^{-1} \cap  P} (\mathbf 1, \tau) \ = \bigoplus_{x \in P \backslash G / K} V_{\tau}^{xKx^{-1} \cap P}.
\end{multline} 

\smallskip
By the Iwasawa decomposition $G = P\cdot G(\O)$ (where $\O$ is our abbreviated notation for $\O_v$ in $F_v$), every representative $x \in P \backslash G / K$, may be taken to be in $G(\O)$. A principal congruence subgroup of $G(\O)$ being normal, from \eqref{eqn:mackey-kutzko} and \eqref{eqn:mackey-kutzko-2} we have
\begin{equation} 
\label{eq3}
{\rm dim} \left(\aInd^P_G(\tau)^K \right) \ = \ | P \backslash G / K| \, {\rm dim}(V_{\tau}^{K \cap P}),
\end{equation}
where $V_\tau$ is the representation space of $\tau,$ which is either $\chi \times \sigma$ or $\chi^{-1}[2n] \times {}^{\kappa_{2n}}\sigma.$ Now $K$ being principal congruence subgroup of $G$, it has an Iwahori factorisation:
$$ 
K  \ = \  
(K \cap U_{P^{-}}) \cdot (K \cap M_{P}) \cdot (K \cap U_{P}). 
$$
Hence $K \cap P = (K \cap M_P) \cdot ( K \cap U_P).$ Furthermore, $\tau$ is a representation of $M_P$ that is inflated to $P$ before inducing. 
Hence,
\begin{equation} \label{eq4} 
V_{\tau}^{K \cap P} = V_{\tau}^{K \cap M_P}. 
\end{equation}
Since $M_P = \GL_1 \times \rO(2n)$, we may write $K \cap M_P = K_{\GL_1} \times {}^{\circ}K$, with 
$K_{\GL_1}  \subseteq \O^{\times}$ and $^{\circ}K \unlhd  \rO(2n) (\O)$. The character $\chi^{-1}[2n]$, restricted to $\O^{\times}$ is the same as $\chi^{-1}$, hence both are either trivial or nontrivial on $K_{\GL_1}$ simultaneously. Since $\sigma \cong {}^{\kappa_{2n}}\sigma$ we get:
\begin{equation} 
\label{eq5} 
{\rm dim} \left(V^{K \cap M_P}_{\chi \times \sigma}\right) = {\rm dim} \left(V^{K \cap M_P}_{\chi^{-1}[2n] \times {}^{\kappa_{2n}} \sigma}\right). 
\end{equation}
The proof follows from \eqref{eq3}, \eqref{eq4} and \eqref{eq5}. 
\end{proof}

The proof shows that the common dimension ${\sf k}$ of various $E$-vector spaces obtained in $(i)$ of Thm.\,\ref{thm:Manin-Drinfeld} is of the form $l {\sf k'}$, which follows from the description of the isotypic component of strongly inner cohomology as described in Sect.\,\ref{sec:isotypic-component-strongly-inner-cohomology}.

\smallskip

We now prove $(ii)$ of Thm.\,\ref{thm:Manin-Drinfeld}; and leave $(iii)$ to the reader as it is almost identical to the proof of $(ii)$. 
Define the map $\RR^{b}_{\chi_f, \sigma_f}$ by the following diagram, in which $[G]$ (resp., $[P]$) 
denotes $\pi_0(G(\R)) \times G(\A_f)$ (resp., $\pi_0(P(\R)) \times P(\A_f)$).   

{\small
\begin{equation}\label{manin-drinfeld-diagram}
\xymatrix{
H^{\q_b} (\partial \cS^G_{K_{f}}, \cM_{\lambda,E}) \ar[dd]_{\RR^{b}_{\chi_f, \sigma_f}} 
\ar[r]^{\substack {\text{restriction to the stratum for }P}} & 
H^{\q_b} (\partial_P \cS^G_{K_{f}}, \cM_{\lambda,E}) \ar[d]_{\cong}^{\text{Prop.\,\ref{prop:bdry-coh-2}}} \\
&   \bigoplus_{u \in W^P} \left(\aInd^{[G]}_{[P]} (H^{\q_b-\ell(u)} (\uSMP, \cM_{u \cdot \lambda, E})) \right)^{K_f}  \ar[d] \\
I_b^S(\chi_f, \sigma_f)_w  \oplus I_b^S(\chi_f^{-1}[2n], \sigma_f)_{w'}
& \left(\aInd^{[G]}_{[P]} (H_{!!}^{q_b} (\uSMP, \cM_{w \cdot \lambda, E})) \right)^{K_f}  \ar[l] 
}
\end{equation}}
(The bottom horizontal arrow needs the word of explanation as in \ref{sec:ind-mod-bdry-coh}.) The rest of the proof is similar to, but finer than,  
the proof of Prop.\,\ref{prop:splitting-inner-cohomology}. We need to show that 
\begin{equation}
\label{eqn:empty-intersect-spectra}
\Spec_{\HH^{G,S}} \left(I_b^S(\chi_f, \sigma_f)_w  \oplus I_b^S(\chi_f^{-1}[2n], \sigma_f)_{w'}\right) 
~\cap~ 
\Spec_{\HH^{G,S}}\left({\rm Ker}(R^{b}_{\chi_f, \sigma_f})\right) \ = \ \emptyset. 
\end{equation}
Towards this, suppose we have a simple $\HH^{G,S}$-module that occurs as a subquotient of 
$H^{\q_b} (\partial \cS^G_{K_{f}}, \cM_{\lambda,E} )$ and whose  
$\HH^{G,S}$-eigencharacter appears in 
$$
\Spec_{\HH^{G,S}}\left(I_b^S(\chi_f, \sigma_f)_w  \oplus I_b^S(\chi_f^{-1}[2n], \sigma_f)_{w'}\right), 
$$
then using Prop.\,\ref{prop:bdry-coh-2}, for some $\iota: E \rightarrow \C$, there exists a proper parabolic subgroup $R$ of $G$ with Levi quotient 
$M_R$ which is the restriction of scalars from $F$ to $\Q$ of 
$$\GL_{n_1}/F \times  \cdots \times \GL_{n_k}/F \times \rO(\ell, \ell)/F
$$ 
and there exist 
cuspidal automorphic representations $\tau_i$ of $\GL_{n_i}(\A_F)$ and $\xi$ of  $\rO(\ell, \ell)(\A_F)$ such that the action of $\HH^{G,S}$ on 
$(\Ind_{R(\A)}^{G(\A)}(\tau_1 \times \tau_2 \times \cdots \tau_k \times \xi))^{K_f}$ is via the same eigencharacter as its action on 
$(\Ind_{P(\A)}^{G(\A)}({}^\iota\chi[-n] \times {}^\iota\sigma_j))^{K_f},$ for some $1 \leq j  \leq n$. 
In other words, 
\begin{equation*} 
\Ind_{R(\A)}^{G(\A)}(\tau_1 \times \tau_2 \times \cdots \tau_k \times \xi)
\ \cong_{\rm a.e.} \ 
\Ind_{P(\A)}^{G(\A)}({}^\iota\chi[-n] \times {}^\iota\sigma_{j}), 
\end{equation*}  
where $\cong_{\rm a.e.}$ means the local modules are equivalent for all $v \notin S.$ 
Suppose $\Psi(\xi)$ is the Arthur parameter for $\xi$, then the Arthur parameters on both sides of the above equation are related as: 
$$
(\boxplus_{i=1}^{k} \tau_i) \ \boxplus \ \Psi(\xi)  \ \boxplus \ (\boxplus_{i=1}^{k} \tau_i ^{\sf v}) 
\ \ = \ \  
{}^\iota \chi[-n]  \ \boxplus \ 
\Psi({}^\iota \sigma_j)  \ \boxplus \ 
{}^\iota \chi^{-1}[n].  
$$
Using Jacquet and Shalika \cite[Thm.\,4.4]{jacquet-shalika-II} we conclude:
$\ell = n, \ k =1, \ R = P,$ and $\Psi(\xi) = \Psi({}^\iota \sigma_j)$, then using Arthur \cite{arthur} and Atobe--Gan \cite{atobe-gan} 
we conclude that either $\xi \cong _{\rm a.e.} {}^\iota \sigma_j$ or $\xi \cong_{\rm a.e.} {}^{\kappa_{2n} \iota} \sigma_j \cong {}^\iota \sigma_j;$ furthermore, we also get 
$\tau_1 = {}^\iota \chi[-n] $ or $ {}^\iota \chi^{-1}[n].$ 
This proves that $I_b^S(\chi_f, \sigma_f)_w  \oplus I_b^S(\chi_f^{-1}[2n], \sigma_f)_{w'}$ is isotypic in 
$H^{\q_b} (\partial \cS^G_{K_{f}}, \cM_{\lambda,E})$. Furthermore, recalling that in the definition of the modules $I_b^S(\chi_f, \sigma_f)_w$ and 
$I_b^S(\chi_f^{-1}[2n], \sigma_f)_{w'}$ we first take the isotypic components of $\chi_f \times \sigma_f$ and $\chi_f^{-1}[2n] \times {}^{\kappa_{2n}}\sigma_f$
and then induce, the same argument as above proves \eqref{eqn:empty-intersect-spectra}. \end{proof}

\medskip
\section{Rank-one Eisenstein cohomology}
\label{sec:rank-one-eis-coh}

An indispensable tool for our main theorem on Eisenstein cohomology is Poincar\'e duality which gives that the image of total cohomology in the cohomology of the 
boundary is in fact a maximal isotropic subspace. We briefly review Poincar\'e duality in our context, and refer the reader to \cite[Sect.\,6.1]{harder-raghuram} for more details.

\medskip
\subsection{Poincar\'e duality}
\label{sec:poincare}

The maps in the long exact sequence in cohomology from \ref{sec:long-exact-seq} relate Poincar\'e duality for $\cS^G_{K_{f}}$ and 
$\partial \cS^G_{K_{f}}$ as in the following diagram:

\begin{equation}
\label{eqn:poincare-duality}
\xymatrix 
{
 { {H^{\bullet} ( \cS^G_{K_{f}}, \widetilde \cM_{\lambda,E})} \ar[d]^-{\mathfrak r^{\bullet}} } & \quad \times & {H_c^{{\sf d}- \bullet} (\cS^G_{K_{f}}, \widetilde\cM_{\lambda^{\sf v},E})} & \longrightarrow & E \\
  {H^{\bullet} (\partial \cS^G_{K_{f}}, \widetilde\cM_{\lambda,E})} & \quad \times& {H^{{\sf d}-1- \bullet} (\partial \cS^G_{K_{f}}, \widetilde\cM_{\lambda^{\sf v},E})}  \ar[u]^-{\partial^{\bullet}} & \longrightarrow &  E 
 }
 \end{equation}
where $\sf d := {\rm dim}( \cS^G_{K_{f}})$ and $ \sf d-1= {\rm dim}( \partial \cS^G_{K_{f}})$, and the duality maps are the horizontal arrows.   
 Define Eisenstein cohomology as 
 $$
 H^q_{\rm Eis}(\partial \cS^G_{K_{f}},\widetilde \cM_{\lambda,E}) \ := \ \text{Image}\left(H^{q} ( \cS^G_{K_{f}}, \widetilde\cM_{\lambda,E})\xrightarrow{\mathfrak r^{\bullet}} H^q (\partial \cS^G_{K_{f}}, \widetilde\cM_{\lambda,E})\right).
$$
It follows from the above diagram as in \cite[Prop.\,6.1]{harder-raghuram} 
that $H^q_{\rm Eis}(\partial \cS^G_{K_{f}}, \cM_{\lambda,E})$ is a maximal isotropic subspace of boundary cohomology, i.e.,
 \begin{equation}
 \label{eqn:max-isotropic}
 H^q_{\rm Eis}(\partial \cS^G_{K_{f}},\widetilde \cM_{\lambda,E}) = H^{{\sf d} - 1 -q}_{\rm Eis}((\partial \cS^G_{K_{f}}, \widetilde\cM_{\lambda^{\sf v},E})^{\perp}.
 \end{equation}

\medskip
\subsection{The main theorem on rank-one Eisenstein cohomology}
\label{sec:theorem-rank-one-eis-coh}

The following theorem is the generalization of Thm.\,6.2 of \cite{harder-raghuram} to our context of orthogonal groups: 

\begin{theorem}
\label{thm:rank-one-eis-coh}
Define the images of Eisenstein cohomology under the Hecke-equivariant maps $\RR$ of Thm.\,\ref{thm:Manin-Drinfeld}  as: 
$$ 
\fI_b^S(\chi_f, \sigma_f) \ :=  \ 
\RR^{b}_{\chi_f, \sigma_f}(H^{\q_b}_{\rm Eis}(\partial \cS^G, \tM_{\lambda,E})^{K_f}),
$$
\smallskip
$$
\fI_t^S(\chi_f, \sigma_f)  \ :=  \ 
\RR^{t}_{\chi_f, \sigma_f}(H^{\q_t}_{\rm Eis}(\partial \cS^G, \tM_{\lambda^{\sf v},E})^{K_f}). 
$$
We have: 
\smallskip
\begin{enumerate}
\item[(i)] $\fI_b^S(\chi_f, \sigma_f)$ is a ${\sf k}$-dimensional subspace of 
$$
I_b^S(\chi_f, \sigma_f)_w  \ \oplus \  I_b^S(\chi_f^{-1}[2n], \sigma_f)_{w'}.
$$ 

\item[(ii)] $\fI_t^S(\chi_f, \sigma_f)$ is a ${\sf k}$-dimensional subspace of 
$$
I_t^S(\chi_f^{-1}[2n], \sigma_f)_{w^{\sf v}}  \ \oplus \  
I_t^S(\chi_f, \sigma_f)_{w^{\sf v}{}'}.  
$$
\end{enumerate}

\end{theorem}

\begin{proof}
The proof is very similar to the proof of Thm.\,6.2 of \cite{harder-raghuram} with some fine differences; we adumbrate the proof while explaining such 
differences, and refer the reader to {\it loc.\,cit.}\ for more details. There are two steps. The first step (as in \cite[6.2.2.1]{harder-raghuram}) 
is to show that these images are at least ${\sf k}$-dimensional; this is 
achieved after a base-change via an $\iota \in {\rm Hom}(E,\C)$ and showing the resulting spaces are at least ${\sf k}$-dimensional. 
Suppose $(d, \mu)$ is on the right of unitary axis.
To begin, we describe the base-change to $\C$ via $\iota$ of 
$I_b^S(\chi_f, \sigma_f)_w$; similar descriptions apply to the other induced modules in (i) and (ii).  
Tensoring \eqref{eqn:I-b-S-bdry-coh} by $\C$ via $\iota$, using \eqref{eqn:isotypic-comp-after-base-change-1} for the inducing 
data, and Prop.\,\ref{prop:g-v-bar-k-v-coh}, we get: 
\begin{equation}
\label{eqn:l-many-summands}
I_b^S(\chi_f, \sigma_f)_w \otimes_{E, \iota} \C \ \cong \ 
\bigoplus_{j=1}^l 
H^{\q_b}\left(\g(\R), \overline{K}_\infty; \ \aInd_{P(\A)}^{G(\A)}({}^\iota\chi \otimes V_{{}^\iota\sigma_j})^{K_f} \otimes \M_{{}^\iota\lambda}\right), 
\end{equation}
where the $j$-th summand, using the inverse of $\Phi_j$ in the inducing data, is equivalent to: 
$$
H^{\q_b}\left(\g(\R), \overline{K}_\infty; \ \aInd_{P(\R)}^{G(\R)}({}^\iota\chi_\infty \otimes {}^\iota\sigma_{j, \infty}) \otimes \M_{{}^\iota\lambda}\right)
\otimes \aInd_{P(\A_f)}^{G(\A_f)}({}^\iota\chi_f \otimes {}^\iota\sigma_{j,f})^{K_f}, 
$$
with the relative Lie algebra cohomology group being one-dimensional. Similarly, we have: 
\begin{multline*}
I_b^S(\chi_f^{-1}[2n], \sigma_f)_w \otimes_{E, \iota} \C \ \cong \\ 
\bigoplus_{j=1}^l 
H^{\q_b}\left(\g(\R), \overline{K}_\infty; \ \aInd_{P(\A)}^{G(\A)}({}^\iota\chi^{-1}[2n] \otimes V_{{}^{\kappa_{2n}\iota}\sigma_j})^{K_f} \otimes \M_{{}^\iota\lambda}\right).
\end{multline*}
For each $j$, for brevity, we let $T_{{\rm st},j}^\bullet$ denote the map induced by applying the functor 
$H^{\q_b}\left(\g(\R), \overline{K}_\infty; \, - \otimes \M_{{}^\iota\lambda} \right)$ to  
the standard intertwining operator $T_{\rm st}(-n, w_{P}, {}^\iota\chi \times {}^\iota\sigma_j)$: 
$$
H^{\q_b}\left(\g(\R), \overline{K}_\infty; \, - \otimes \M_{{}^\iota\lambda} \right)(T_{\rm st}(-n, w_{P}, {}^\iota\chi \times {}^\iota\sigma_j)). 
$$
Using Thm.\,\ref{thm:Langlands} and Sec.\,\ref{sec:isotypic-component-strongly-inner-cohomology}, 
we see that $\fI_b^S(\chi_f, \sigma_f) \otimes_{E,\iota} \C$ at least contains all cohomology classes of the form: 
\begin{equation}
\label{eqn:sum_T-st-j}
\left\{ 
\left(f_1, \ldots, f_l \right), \ 
\left((f_1 + T_{{\rm st},1}^\bullet(f_1)), \ldots, (f_l + T_{{\rm st},l}^\bullet(f_l))\right)
\right\},
\end{equation}
where, 
\begin{enumerate}
\item $f_j \in  H^{\q_b}\left(\g(\R), \overline{K}_\infty; \ \aInd_{P(\A)}^{G(\A)}({}^\iota\chi \otimes V_{{}^\iota\sigma_j})^{K_f} \otimes \M_{{}^\iota\lambda}\right);$ 
\medskip
\item $(f_1, \ldots, f_l) \in I_b^S(\chi_f, \sigma_f)_w \otimes_{E,\iota} \C ;$ and   
\medskip
\item 
$\left( (f_1 + T_{{\rm st},1}^\bullet(f_1)), \ldots, (f_l + T_{{\rm st},l}^\bullet(f_l)) \, \right) \in I_b^S(\chi_f^{-1}[2n], \sigma_f)_{w'} \otimes_{E,\iota} \C;$ 
\end{enumerate}

\medskip
Hence 
$$
\dim_E(\fI_b^S(\chi_f, \sigma_f)) = \dim_\C(\fI_b^S(\chi_f, \sigma_f) \otimes_{E,\iota} \C) \geq  
\dim_\C(I_b^S(\chi_f, \sigma_f)_w \otimes_{E,\iota}\C) = {\sf k}.
$$
If $(d, \mu)$ is on the left of unitary axis, then $(-d-2n, {}^{\kappa_{2n}} \mu)$ is on the right of unitary axis; hence by working with 
$T_{\rm st}(n, w_{P_0}, {}^\iota\chi^{-1}[-2n] \times {}^{\kappa_{2n} \iota}\sigma_j)$ we get the desired result. Similarly, we also get $\fI_t^S(\chi_f, \sigma_f)$ 
is at least ${\sf k}$-dimensional.  (In \cite{harder-raghuram}, by strong multiplicity-one for general linear groups we have $l = 1$; however, for orthogonal groups, in general, $l > 1$, and this necessitates extra care in the final proof in Sect.\,\ref{sec:proof}.)

\medskip

The second step (as in \cite[6.2.2.2]{harder-raghuram}) involves the Poincar{\'e} duality pairing (\ref{eqn:poincare-duality}) to show that 
both are exactly $\sf k$-dimensional $E$-vector spaces; one of the fine differences alluded to concerns the setting up of the relevant exercise in linear algebra to make this conclusion, which we now discuss. For brevity, 
let $\cV = I_b^S(\chi_f, \sigma_f)_w  \ \oplus \  I_b^S(\chi_f^{-1}[2n], \sigma_f)_{w'}$ and 
$\cW = I_t^S(\chi_f^{-1}[2n], \sigma_f)_{w^{\sf v}}  \ \oplus \  
I_t^S(\chi_f, \sigma_f)_{w^{\sf v}{}'}.$ Since, $\cV \subset H^{\q_b}(\partial \cS^G, \tM_{\lambda,E})^{K_f}$ and 
$\cW \subset H^{\q_t}(\partial \cS^G, \tM_{\lambda^{\sf v},E})^{K_f}$, the Poincar\'e duality pairing in the bottom horizontal arrow 
in \eqref{eqn:poincare-duality} restricts to a pairing on 
$(\ , \ ) : \cV \times \cW \to E.$ This pairing need not be perfect in general, but in our situation the induced map $\cV \to \cW^*$ is injective. This is 
because in the bottom-degree we have all possible occurrences of $\chi_f \times \sigma_f$ accounted for, which is reflected by $\wedge^0\fa^*$ 
being one-dimensional. In contrast, in the top-degree, we would have $\wedge^{{\sf r}_F-1}\fa^*,$ which is ${\sf r}_F$-dimensional that is greater 
than one if $F \neq \Q$; see Rem.\,\ref{rem:top-degree-injection}. Now, for brevity, let $\fI_b = \fI_b^S(\chi_f, \sigma_f) \subset \cV$ and 
$\fI_t := \fI_t^S(\chi_f, \sigma_f) \subset \cW.$ From \eqref{eqn:max-isotropic} we have $(\fI_b, \fI_t) = 0.$ The injective map $\cV \to \cW^*$ 
induces an injection $\fI_b \to (\cW/\fI_t)^*$. Together with the conclusion of the first step we get:  
$$
{\sf k} \ \leq \ \dim(\fI_b) \ \leq \  \dim(\cW/\fI_t)^* \ \leq \ {\sf k}, 
$$
giving us ${\sf k} = \dim(\fI_b) = \dim(\fI_t).$ 
\end{proof}

\medskip

The above theorem is best appreciated by pretending ${\sf k} = 1$, in which case, it says that the image of total Eisenstein cohomology, via
the rank-one Eisenstein cohomology for $P$, inside a $2$-dimensional isotypic subspace is one-dimensional, i.e., a line in a $2$-dimensional vector space over $E$. The slope of this line is an element of $E$, which after passing to a transcendental situation is related to a ratio of $L$-values, hence giving us a rationality
result for that ratio; this is the essence of the proof of our main theorem on special values of $L$-functions. 

\medskip

The embedding $\Phi_j$ in \eqref{eqn:Phi_j} 
of the cuspidal representation ${}^\iota\sigma_j$ into the space of cusp forms on ${}^\circ M$, 
which we recall is unique up to a nonzero complex number, is used in the above proof. However, since the same $\Phi_j$ shows up on either side of the intertwining operator: 
\begin{multline*}
T_{\rm st}(-n, w_{P}, {}^\iota\chi \times {}^\iota\sigma_j) : 
\aInd_{P(\A)}^{G(\A)}\left({}^\iota\chi \otimes \Phi_j({}^\iota\sigma_j)\right) \ \longrightarrow \\  
\aInd_{P(\A)}^{G(\A)}\left({}^\iota\chi^{-1}[2n] \otimes \Phi_j({}^{\kappa_{2n}\iota}\sigma_j)\right), 
\end{multline*}
we deduce that the map $T_{{\rm st},j}^\bullet$ independent of the choice of $\Phi_j.$

\medskip
 \section{Proof of the main theorem on $L$-values}
\label{sec:main-theorem}

We are now in a position to give a proof of Thm.\,\ref{thm:main:introduction}. 
\medskip

\subsection{``Right Vs Left" of the unitary axis}
\label{sec:right-vs-left} 
As in Sect.\,\ref{sec:arith-int-op}, just for this subsection, the $\iota : E \to \C$ is fixed and so, for brevity, we suppress it from notation. 
In Sect.\,\ref{sec:proof} we will show, under the condition $1 - \mu_{\rm min} \leq -(d+n) \leq \mu_{\rm min} -1$ for $(d,\mu)$, 
an algebraicity result for the ratio 
$$
\frac{L(-n, \, \chi \times \sigma)}{L(1-n, \, \chi \times \sigma)} \ = \ \frac{L(-n-d, \, \chi^\circ \times \sigma)}{L(1-n-d, \, \chi^\circ \times \sigma)}
$$
when $(d, \mu)$ is on the right of the unitary axis, i.e., when $-(d+n) \geq 0$. The above restrictions on $-(d+n)$ together gives: 
$0 \leq -(n+d) \leq \mu_{\rm min} - 1.$ Letting $d$ vary within these bounds we get an algebraicity theorem for the ratios: 
\begin{equation}
\label{eqn:list-right}
\left\{  
\frac{L(0, \, \chi^\circ \times \sigma)}{L(1, \, \chi^{\circ} \times \sigma)}, \ \ 
\frac{L(1, \, \chi^\circ \times \sigma)}{L(2, \, \chi^{\circ} \times \sigma)}, \ \ \dots \ , \  \
\frac{L(\mu_{\rm min}-1, \, \chi^\circ \times \sigma)}{L(\mu_{\rm min}, \, \chi^{\circ} \times \sigma)} 
\right\}.
\end{equation}
The critical set $\{ 1- \mu_{\rm min}, \, 2- \mu_{\rm min}, \dots, \mu_{\rm min}\}$ for $L(s,\, \chi^{\circ} \times \sigma)$ contains $2\mu_{\rm min}-1$ pairs of 
successive integers, and \eqref{eqn:list-right} covers only $\mu_{\rm min}$ successive pairs. For the remaining $\mu_{\rm min}-1$ ratios of successive $L$-values: 
\begin{equation}
\label{eqn:list-left-1}
\left\{\frac{L(1-\mu_{\rm min}, \, \chi^\circ \times \sigma)}{L(2-\mu_{\rm min}, \, \chi^{\circ} \times \sigma)}, \ \ \dots \ , \  \
\frac{L(-1, \, \chi^\circ \times \sigma)}{L(0, \, \chi^{\circ} \times \sigma)} 
\right\}, 
\end{equation}
we can start from the other side of the intertwining operator: 
$$
T_{\rm st}(s)|_{s = -n} \ : \ 
\aInd_{P(\A)}^{G(\A)}  \left(\chi \times \sigma \right) \ \longrightarrow \ 
\aInd_{P(\A)}^{G(\A)}  \left(\chi^{-1}[2n] \times {}^{\kappa_{2n}}\sigma \right); 
$$
that is, we consider the standard intertwining operator: 
$$
T_{\rm st}(s, w_{P_0}, \chi^{-1} \times {}^{\kappa_{2n}} \sigma) \ : \ I(s, \chi^{-1} \times {}^{\kappa_{2n}} \sigma ) \ \longrightarrow \ I(-s, \chi \times \sigma) 
$$ 
at its point of evaluation which is $s = n$. If $(d,\mu)$ corresponds to $(\chi, \sigma)$, then $(-d-2n, {}^{\kappa_{2n}}\mu)$ corresponds to $(\chi^{-1}[2n] , {}^{\kappa_{2n}}\sigma).$ 
Since, locally, 
$$
{}^{\kappa_{2n}}(\mu_1, \dots, \mu_{n-1}, \mu_n) = (\mu_1, \dots, \mu_{n-1}, -\mu_n),
$$ 
we deduce that 
$\mu_{\rm min} = ({}^{\kappa_{2n}}\mu)_{\rm min}.$ The condition 
$1 - \mu_{\rm min} \leq -(d+n) \leq \mu_{\rm min} -1$ imposed by the 
combinatorial lemma for $(d,\mu)$ is identical to the those on $(-d-2n, {}^{\kappa_{2n}}\mu).$ Next, if $(d,\mu)$ is on the left of the unitary axis, i.e., 
if $-(d+n) < 0$ which is $(d+n) \geq 1$, then for  $(-d-2n, {}^{\kappa_{2n}}\mu)$ we have $-(-d-2n)-n \geq 1$, or that $(-d-2n, {}^{\kappa_{2n}}\mu)$ is strictly to 
the right of its unitary axis, giving us an algebraicity theorem for 
$$
\frac{L(n, \, \chi^{-1} \times {}^{\kappa_{2n}} \sigma )}{L(1+n, \, \chi^{-1} \times {}^{\kappa_{2n}} \sigma )} \ = \ 
\frac{L(n+d, \, \chi^\circ{}^{-1} \times {}^{\kappa_{2n}} \sigma )}{L(1+n+d, \, \chi^\circ{}^{-1} \times {}^{\kappa_{2n}} \sigma )}. 
$$
So when $(d, \mu)$ is on the left of the unitary axis, together with the combinatorial lemma, 
the restrictions are: $1 \ \leq \ d+n \ \leq \ \mu_{\rm min}-1$; and if we let $d$ vary within this range, we
get an algebraicity theorem for the ratios of $L$-values:
\begin{equation}
\label{eqn:list-left-2}
\left\{
\frac{L(1, \, \chi^\circ{}^{-1} \times {}^{\kappa_{2n}} \sigma )}{L(2, \, \chi^\circ{}^{-1} \times {}^{\kappa_{2n}} \sigma )}, \ \ \dots \ , \ \ 
\frac{L(\mu_{\rm min}-1, \, \chi^\circ{}^{-1} \times {}^{\kappa_{2n}} \sigma )}{L(\mu_{\rm min}, \, \chi^\circ{}^{-1} \times {}^{\kappa_{2n}} \sigma )} 
\right\}.
\end{equation}
To compare the lists in \eqref{eqn:list-left-1} and \eqref{eqn:list-left-2}, we use the global functional equation: 
$$
L(s, \, \chi^\circ \times \sigma) \ = \ \varepsilon(s, \, \chi^\circ \times \sigma) L(1-s, \, \chi^\circ{}^{-1} \times \sigma^\sv), 
$$
together with the easy observations (i) ${}^{\kappa_{2n}} \sigma = \sigma$ since $\kappa_{2n} \in \rO(2n)$, and (ii) $\sigma^\sv = \sigma$ since 
both are nearly equivalent (as easy observation on their Satake parameters), having the same cuspidal Arthur parameter and being globally generic 
with respect to the same Whittaker datum. The global $\varepsilon$-factor is a product of local $\varepsilon$-factors. At a non-archimedean place $v$, the local 
$\varepsilon$-factor is of the form $W(\chi_v \times \sigma_v)\, q_v^{(1/2 -s)(c_v)}$ where $W(\chi_v \times \sigma_v)$ is the local root number and $c_v$ is the sum of conductoral exponents of the data. Thus the ratio of local $\varepsilon$-factors at two successive integers is a nonzero rational number. At an archimedean place, the local $\varepsilon$-factor is a constant. Hence the ratio of global $\varepsilon$-factors at  two successive integers is a nonzero rational number. Finally, we note all the individual $L$-values in \eqref{eqn:list-left-1} and \eqref{eqn:list-left-2} are nonvanishing; since $L(s, \chi^\circ \times \Psi(\sigma)) \neq 0$ for any integer value of $s$, by the unitarity of $\Psi(\sigma)$ and $\chi^\circ$, and appealing
to the theorem of Jacquet--Shalika \cite{jacquet-shalika-invent}. We get the equalities: 
\begin{multline*}
\frac{L(1, \, \chi^\circ{}^{-1} \times {}^{\kappa_{2n}} \sigma )}{L(2, \, \chi^\circ{}^{-1} \times {}^{\kappa_{2n}} \sigma )} \ \approx_{\Q^\times} \ 
\frac{L(0, \, \chi^\circ \times \sigma)}{L(-1, \, \chi^{\circ} \times \sigma)}, \ \ \dots \ \ , \\ 
\frac{L(\mu_{\rm min}-1, \, \chi^\circ{}^{-1} \times {}^{\kappa_{2n}} \sigma )}{L(\mu_{\rm min}, \, \chi^\circ{}^{-1} \times {}^{\kappa_{2n}} \sigma )}
 \ \approx_{\Q^\times} \ 
\frac{L(2-\mu_{\rm min}, \, \chi^\circ \times \sigma)}{L(1-\mu_{\rm min}, \, \chi^{\circ} \times \sigma)}, 
\end{multline*}
which show that an algebraicity result for the ratios of $L$-values in \eqref{eqn:list-left-2} implies an algebraicity result for the (reciprocals of the) ratios
of $L$-values in \eqref{eqn:list-left-1}.

\medskip
\subsection{Conclusion of proof}
\label{sec:proof}

We need to prove that when $(d, \mu)$ is on the right of the unitary axis and also satisfies the condition imposed by the combinatorial lemma, 
i.e, when $0 \leq -(d+n) \leq \mu_{\rm min} -1,$ then 
$L(-n, \, {}^\iota\chi \times {}^\iota\sigma)/L(1-n, \ {}^\iota \chi \times {}^\iota\sigma) \in \iota(E),$
and that for all $\eta \in {\rm Gal}(\bar{\Q}/ \Q)$ we have the reciprocity law as in $(iii)$ of Thm.\,\ref{thm:main:introduction}. 

\medskip
From Thm.\,\ref{thm:rank-one-eis-coh}, we have $\fI_b^S(\chi_f, \sigma_f)$ is a ${\sf k}$-dimensional subspace inside the $2 {\sf k}$-dimensional space 
$I_b^S(\chi_f, \sigma_f)_w  \ \oplus \  I_b^S(\chi_f^{-1}[2n], \sigma_f)_{w'}.$  From 
the proof of Thm.\,\ref{thm:rank-one-eis-coh} it follows that we have an $E$-linear map: 
$T_{\Eis} := T_{\Eis}(\chi_f, \sigma_f) : I_b^S(\chi_f, \sigma_f)_w  \ \to \ I_b^S(\chi_f^{-1}[2n], \sigma_f)_{w'},$
such that 
\begin{equation}
\label{eqn:image}
\fI_b^S(\chi_f, \sigma_f) \ = \ 
 \left\{
(\xi, \, \xi + T_{\Eis}(\xi) \ | \ \xi \in I_b^S(\chi_f, \sigma_f)_w \right\}. 
\end{equation}
Take $\iota: E \rightarrow \C,$ and base change to $\C$ via $\iota$ to consider the map $T_\Eis \otimes_{E,\iota} 1_\C$. 

The domain and codomain of $T_\Eis \otimes_{E,\iota} 1_\C$ have $l$ many summands (for the domain see \eqref{eqn:l-many-summands}, with a similar description
for the codomain), and it seems like we have a period matrix, but from the proof of Thm.\,\ref{thm:rank-one-eis-coh} (see especially \eqref{eqn:sum_T-st-j}), 
it follows that $T_\Eis \otimes_{E,\iota} 1_\C$ is the `diagonal matrix'  
$({\bf 1}_j + T_{\st,j}^\bullet)_{j=1}^{l};$ the notation for the $j$-th term being clear from context.  We contend that this diagonal matrix, up to the ratio of 
critical values  $L(-n, {}^\iota \chi \times {}^\iota\sigma)/L(1-n, {}^\iota \chi \times {}^\iota\sigma)$ is $\iota(E)$-rational; this is the essence of the rest of the proof. 

\smallskip

For $v \in S_{\infty}$, and $1 \leq j \leq l$, the local representation $^{\iota} \sigma_{j,v}$ is a discrete series representation in the same 
$L$-packet of ${\mathbb D}_{\mu^{v}}^\dagger$; 
hence, for all $1 \leq j \leq l$, we get an equality of the archimedean $L$-factors 
$L(s, {}^\iota\chi_v \times {}^\iota\sigma_{j,v}) = L(s, {}^\iota\chi_v \times {}^\iota\sigma_{v}).$ 
Whence, from Prop.\,\ref{prop:intertwining-operator-in-cohomology}, for each $1 \leq j \leq l,$ we have the relative Lie algebra cohomology classes at infinity, 
$[\sI^0_j]$ and $[\tilde{\sI}^0_j],$ and that 
$$
T_{\st, j, \infty}^{\bullet} ([\sI^0_j] ) \ \approx_{\Q^\times} \ 
\frac{L_\infty(-n, {}^\iota\chi \times {}^\iota\sigma_{j}) }{L_\infty(1-n, {}^\iota\chi \times {}^\iota\sigma_{j})}
[\tilde{\sI}^0_j] \ =  \ 
\frac{L_\infty(-n, {}^\iota\chi \times {}^\iota\sigma) }{L_\infty(1-n, {}^\iota\chi \times {}^\iota\sigma)} 
[\tilde{\sI}^0_j].
$$
Let us  define a map 
\begin{multline*}
T_{\loc, j, \infty}^{\bullet} : 
H^{\q_b}\left(\g(\R), \overline{K}_\infty; \, \aInd_{P(\R)}^{G(\R)}({}^\iota\chi_\infty \otimes {}^\iota\sigma_{j, \infty}) \otimes \M_{{}^\iota\!\lambda,\C}\right) \ \to \\  
H^{\q_b}\left(\g(\R), \overline{K}_\infty; \, \aInd_{P(\R)}^{G(\R)}({}^\iota\chi_\infty^{-1}[2n] \otimes {}^\iota{}^{\kappa_{2n}}\sigma_{j, \infty}) 
\otimes \M_{{}^\iota\!\lambda,\C}\right), 
\end{multline*}
which is characterised by 
$T_{\loc, j, \infty}^{\bullet} ([\sI^0_j]) \ \approx_{\Q^\times} \ [\tilde{\sI}^0_j]$ and such that 
\begin{equation}
\label{eqn:T-st-infinty-loc}
T_{\st, j, \infty}^{\bullet} \ = \ 
\frac{L_\infty(-n, {}^\iota\chi \times {}^\iota\sigma) }{L_\infty(1-n, {}^\iota\chi \times {}^\iota\sigma)} 
T_{\loc, j, \infty}^{\bullet}.
\end{equation}

\medskip

For $v \notin S_\infty$, we have ${}^\iota\sigma_{j,v} = {}^\iota\sigma_v.$ In particular, the discussion for the standard intertwining operator for the $j$-th summand is
independent of $j$, and we suppress the subscript $j$ for what follows at finite places. 

\medskip

For $v \notin S$ (recall that $S$ includes $S_\infty$ and all finite ramified places),  
if $f_v^0$ and $\tilde{f_v^0}$ are the normalised spherical vectors respectively then from \eqref{eqn:L-G-K} at our point of evaluation $s = -n$ we get:
$$
T_{\st,v} (f_v^\circ) = \dfrac{L(-n, ~\!\!^\iota \chi_v \times~\!\! ^\iota\sigma_{v})}{L(1-n, ~\!\!^\iota \chi_v \times~\!\! ^\iota\sigma_{v})} \tilde{f}_v^\circ. 
$$
For $v \notin S$ define: 
$$
T_{\loc,v} : =  \left( \dfrac{L(-n, {}^\iota \chi_v \times {}^\iota\sigma_v)}{L(1-n, {}^\iota \chi_v \times {}^\iota\sigma_v)} \right)^{-1} T_{\st,v}, 
$$ 
which is characterised by $T_{\loc,v}(f_v^\circ) = \tilde{f}_v^\circ.$ It is clear then that $T_{\loc,v}$ is the base change via $\iota$ of a map defined over $E$, i.e., there exists an $E$-linear map $T_{\loc,v,0} $ such that $T_{\loc,v} = T_{\loc,v,0} \otimes_{E, \iota} 1_{\C}.$ Now define 
$T^S_{\loc} =  \otimes_{v \notin S} T_{\loc,v}$ and $T^S_{\loc,0} =  \otimes_{v \notin S} T_{\loc,v,0};$ we have 
$$
T^S_{\loc} = T^S_{\loc,0} \otimes_{E, \iota} 1_{\C}.
$$
Putting together the discussion for $v \notin S$ we have: 
\begin{equation}
\label{eqn:T-st-outside-S}
T_{\st}^S \ := \ 
\bigotimes_{v \notin S} T_{\st,v}  \ = \   
\dfrac{L^{S}(-n, {}^\iota \chi \times {}^\iota\sigma)}{L^{S}(1-n, {}^\iota \chi \times {}^\iota\sigma)}  \left(T^S_{\loc,0} \otimes_{E, \iota} 1_{\C} \right), 
\end{equation}
where $L^{S}(s, {}^\iota \chi \times {}^\iota\sigma) = \prod_{v \notin S} L^{S}(s, {}^\iota \chi_v \times {}^\iota\sigma_v)$ is the partial $L$-function.

\medskip

For $v \in S \setminus S_\infty$, the normalised intertwining operator in \eqref{eqn:norm-int-op-ram-v} and its arithmetic version 
in \eqref{eqn:arith-int-op} can be put together, while invoking the hypothesis in \eqref{eqn:local-assumption}, 
to give (for any $1 \leq j \leq l$): 
\begin{equation}
\label{eqn:T-st-S-but-finite}
T_{\st, v}(-n) \   = \ 
\left( \dfrac{L(-n, {}^\iota \chi_v \times {}^\iota\sigma_v)}{L(1-n, {}^\iota \chi_v \times {}^\iota\sigma_v)} \right) \, 
T_{\arith, v} \otimes_{E, \iota} 1_\C. 
\end{equation}

\medskip
Putting these calculations together, while partitioning the set of all places as the disjoint union 
$S_\infty \cup (S\setminus S_\infty) \cup \{v : v \notin S\}$, we have:
\begin{equation}
\label{eqn:putting-together}
T_{\Eis}(\chi_f, \sigma_f) \otimes_{E,\iota} 1_{\C} 
\ = \ 
\left({\bf 1}_j, \ (T_{\st, j, \infty}^{\bullet}  \ \otimes \ 
(\otimes_{v \in S \setminus S_{\infty}} T_{\st,v}) \ \otimes  \ 
(T_{\st}^S) \right)_{1 \leq j \leq l}. 
\end{equation}
Let $T_{\Eis}^1(\chi_f, \sigma_f) := T_{\Eis}(\chi_f, \sigma_f) - {\bf 1}$. 
Using \eqref{eqn:T-st-infinty-loc}, \eqref{eqn:T-st-outside-S},  \eqref{eqn:T-st-S-but-finite}, and \eqref{eqn:putting-together} we get: 
\begin{multline}
\label{eqn:final-proof}
T_{\Eis}^1(\chi_f, \sigma_f) \otimes_{E,\iota} 1_{\C} 
\ = \\ 
\dfrac{L(-n, {}^\iota \chi \times {}^\iota\sigma)}{L(1-n, {}^\iota \chi \times {}^\iota\sigma)}  \ 
\left(\left(
T_{\loc, j, \infty}^{\bullet} \otimes_E 
T^S_{\loc,0} \otimes_E
(\otimes_{v \in S\setminus S_\infty} T_{\arith, v}) 
\right) \otimes_{E, \iota} 1_\C \right)_{1 \leq j \leq l}.
\end{multline}
It follows that
 $$
 \dfrac{L(-n, {}^\iota \chi \times {}^\iota\sigma)}{L(1-n, {}^\iota \chi \times {}^\iota\sigma)} \in \iota(E).  
$$

\medskip 

For Galois equivariance in $(iii)$ of Thm.\,\ref{thm:main:introduction}, consider the action of $\eta \in {\rm Gal}(\bar{\Q} / \Q)$ on 
$\iota \in {\rm Hom}(E, \bar{\Q}) = {\rm Hom}(E, \C)$ by $\eta_{\ast}(\iota) = \eta \circ \iota$. We hit the diagram in 
\eqref{manin-drinfeld-diagram} by $\eta$, and by functoriality of the cohomology groups as in \cite[2.3.3]{harder-raghuram}, 
we get 
$$
\eta^{\bullet}(T_{\Eis}^1(\chi_f, \sigma_f)) \ = \ 
T_{\Eis}^1({}^\eta\chi_f, {}^\eta\sigma_f).
$$
For brevity, let us put 
$$
T_\loc(\chi_f, \sigma_f) =  
\left(
T_{\loc, j, \infty}^{\bullet} \otimes_E 
T^S_{\loc,0} \otimes_E
(\otimes_{v \in S\setminus S_\infty} T_{\arith, v}) 
\right)_j
$$
for what shows up in the right hand side of \eqref{eqn:final-proof}, 
and also put 
$$
\sL({}^\iota \chi \times {}^\iota\sigma) \ = \  L(-n, {}^\iota \chi \times {}^\iota\sigma)/L(1-n, {}^\iota \chi \times {}^\iota\sigma).
$$ 
Then, we may rewrite \eqref{eqn:final-proof}, while using $\iota$ to base-change to $\bar{\Q}$, as: 
$$
T_{\Eis}^1(\chi_f, \sigma_f) \otimes_{E,\iota} 1_{\bar{\Q}}  \ = \ 
T_\loc(\chi_f, \sigma_f) \otimes_{E,\iota} \sL({}^\iota \chi \times {}^\iota\sigma) 1_{\bar{\Q}}.
$$
On the one hand we have
\begin{multline*}
(1\otimes \eta) \circ (T_{\Eis}^1(\chi_f, \sigma_f) \otimes_{E,\iota} 1_{\bar{\Q}}) \ = \ 
\eta^{\bullet}(T_{\Eis}^1(\chi_f, \sigma_f)) \otimes_{E, \eta \circ \iota} \eta \ = \\  
T_{\Eis}^1({}^\eta\chi_f, {}^\eta\sigma_f) \otimes_{E, \eta \circ \iota} \eta \ = \ 
T_{\loc}({}^\eta\chi_f, {}^\eta\sigma_f) \otimes_{E, \eta \circ \iota} \sL({}^{\eta\circ\iota} \chi \times {}^{\eta\circ\iota}\sigma) \eta, 
\end{multline*}
and, on the other hand, we have
\begin{multline*}
(1\otimes \eta) \circ (T_{\Eis}^1(\chi_f, \sigma_f) \otimes_{E,\iota} 1_{\bar{\Q}}) \ = \ 
(1\otimes \eta) \circ (T_\loc(\chi_f, \sigma_f) \otimes_{E,\iota} \sL({}^\iota \chi \times {}^\iota\sigma) 1_{\bar{\Q}}) \ = \\
\eta^\bullet T_\loc(\chi_f, \sigma_f) \otimes_{E, \eta \circ \iota} \eta(\sL({}^\iota \chi \times {}^\iota\sigma)) \eta \ = \ 
T_{\loc}({}^\eta\chi_f, {}^\eta\sigma_f) \otimes_{E, \eta \circ \iota} \eta(\sL({}^\iota \chi \times {}^\iota\sigma)) \eta.
\end{multline*}
We conclude $\eta(\sL({}^\iota \chi \times {}^\iota\sigma))  \ = \ \sL({}^{\eta\circ\iota} \chi \times {}^{\eta\circ\iota}\sigma).$ 
This concludes the proof of Thm.\,\ref{thm:main:introduction}. \hfill$\Box$ \medskip
 
\subsection{A final comment}
\label{sec:final-comment}
Let us  amplify a remark made in \cite{bhagwat-raghuram-cras}, that 
it is important to prove the main theorem at the level of $L$-functions for $\GL_1 \times \rO(2n)$, and {\it not} as 
$L$-functions for $\GL_1 \times \GL_{2n}$ after transferring from $\rO(2n)$ to $\GL_{2n}$. This is already seen in the context of Shimura's theorem on the special values of Rankin--Selberg $L$-functions for elliptic modular forms (see \cite{shimura-mathann}, Thm.\,1,(iv), and Thm.\,4)
because (i) the Langlands transfer, $f \boxtimes g,$ of a pair of elliptic modular forms $f$ and $g$ of distinct weights, is a cuspidal representation of $\GL_4$ that does not see the Petersson norm $\langle f, f \rangle$ of only one of the constituents; and (ii) for an $L$-function $L(s, \pi)$ with $\pi$ cuspidal on $\GL_4/\Q$, successive $L$-values would see two periods $c^+(\pi)$ and $c^-(\pi)$ attached to $\pi$, and in the automorphic world, it is not (yet) known that if $\pi$ came via transfer from $\GL_2 \times \GL_2$ then $c^+(\pi) \approx c^-(\pi).$  More generally, 
one may ask whether the main result of \cite{harder-raghuram} applied to 
$\GL_1 \times \GL_{2n}$ implies the main result of this paper; this would be the case if we could prove that the relative periods 
$\Omega^\varepsilon(\Psi(\sigma))$ therein attached to the cuspidal representation $\Psi(\sigma)$ of $\GL_{2n}$ are trivial because of $\Psi(\sigma)$ being a 
transfer from a cuspidal $\sigma$ on $\rO(2n)$. 
At this moment we have no idea how one might even begin to prove such a period relation--hence our insistence on working intrinsically in the context of orthogonal groups.


\end{document}